\RequirePackage{fix-cm}
\documentclass[smallextended,numbook,runningheads]{svjour3}
\smartqed  
\usepackage{verbatim}
\usepackage{xcolor}
\usepackage{graphicx}
\usepackage{amsmath,amssymb,bm}
\usepackage{graphicx,graphics,color}
\definecolor{refgreen}{rgb}{0,0.5,0}
\usepackage[colorlinks,linkcolor=blue,citecolor=blue]{hyperref}
\usepackage{tikz}
\usepackage{booktabs,subfig} 

\usepackage[utf8]{inputenc}
\journalname{...}
\date{ \phantom{b} \vspace{45mm}\phantom{e}}
\def\makeheadbox{\hspace{-4mm}Version of \today}

\usepackage{color}

\def\half{\frac{1}{2}}

\newcommand\bfd{{\mathbf d}}
\newcommand\bfe{{\mathbf e}}

\newcommand\bff{{\mathbf f}}
\newcommand\bfg{{\mathbf g}}

\newcommand\bfn{{\mathbf n}}

\newcommand\bfu{{\mathbf u}}
\newcommand\bfv{{\mathbf v}}
\newcommand\bfw{{\mathbf w}}
\newcommand\bfx{{\mathbf x}}

\newcommand\bfy{{\mathbf y}}
\newcommand\bfz{{\mathbf z}}
\newcommand\bfA{{\mathbf A}}

\newcommand\bfH{{\mathbf H}}
\newcommand\bfF{{\mathbf F}}

\newcommand\bfK{{\mathbf K}}
\newcommand\bfM{{\mathbf M}}

\newcommand\bfV{{\mathbf V}}

\newcommand\bfzero{{\mathbf 0}}

\newcommand\calE{{\mathcal E}}

\newcommand\calP{{\mathcal P}}

\newcommand\andquad{\quad\hbox{ and }\quad}
\newcommand\for{\quad\hbox{ for }\quad}

\renewcommand{\d}{\textnormal{d}}

\newcommand{\D}{\mathcal{D}}

\newcommand{\Ga}{\Gamma}

\newcommand{\laplace}{\Delta}

\newcommand{\nbg}{\nabla_{\Gamma}}
\newcommand{\nbgh}{\nabla_{\Gamma_h}}
\newcommand{\mat}{\partial^{\bullet}}

\newcommand{\diff}{\frac{\d}{\d t}}
\newcommand{\eps}{\varepsilon}

\newcommand{\nb}{\nabla}

\newcommand{\pa}{\partial}
\newcommand{\R}{\mathbb{R}}

\newcommand{\spn}{\textnormal{span}}
\def \t {(t)}
\def \ct {(\cdot,t)}
\def \s {(s)}

\def \to {\rightarrow}
\newcommand{\vphi}{\varphi}

\newcommand{\Ih}{\widetilde{I}_h}

\newcommand{\Id}{\mathrm{Id}}

\newcommand{\us}{\bfu^\ast}
\newcommand{\vs}{\bfv^\ast}
\newcommand{\xs}{\bfx^\ast}
\newcommand{\ws}{\bfw^\ast}
\newcommand{\Vs}{\bfV^\ast}
\newcommand{\dotus}{\dot\bfu^\ast}

\newcommand{\dotxs}{\dot\bfx^\ast}
\newcommand{\dotws}{\dot\bfw^\ast}

\newcommand{\eu}{\bfe_\bfu}
\newcommand{\ev}{\bfe_\bfv}
\newcommand{\ex}{\bfe_\bfx}
\newcommand{\ew}{\bfe_\bfw}

\newcommand{\doteu}{\dot\bfe_\bfu}

\newcommand{\dotex}{\dot\bfe_\bfx}
\newcommand{\dotew}{\dot\bfe_\bfw}
\newcommand{\du}{\bfd_\bfu}
\newcommand{\dv}{\bfd_\bfv}
\newcommand{\dw}{\bfd_\bfw}

\newcommand{\Rh}{\widetilde{R}_h}

\newtheorem{result}{Lemma}[section]

\newcommand{\n}{\nu}
\newcommand{\N}{{\mathbb N}}

\newcommand{\dof}{N}

\newcommand{\doi}[1]{doi:\href{https://doi.org/#1}{\texttt{#1}}}

\newcommand{\redoff}{\color{black}}
\newcommand{\blueon}{\color{black}}
\newcommand{\blueoff}{\color{black}}

\newcommand{\bbk}{\color{black}}
\newcommand{\ebk}{\color{black}}

\begin{document}

\title{Numerical analysis for the interaction of mean curvature flow and diffusion on closed surfaces}

\titlerunning{Numerical analysis for the interaction of mean curvature flow and diffusion}        

\author{Charles M. Elliott \and Harald Garcke \and Bal\'azs Kov\'acs}

\authorrunning{C.~M.~Elliott, H.~Garcke, and B.~Kov\'acs} 

\institute{
	C. M. Elliott \at
	Mathematics Institute, University of Warwick \\
	Zeeman Building, Coventry CV4 7AL, UK \\
	\email{C.M.Elliott@warwick.ac.uk}
	\and
	H. Garcke and B. Kov\'{a}cs \at
	Faculty of Mathematics, University of Regensburg, \\
	Universit{\"a}tsstr. 31, 93040 Regensburg, Germany, \\
	\email{\{harald.garcke, balazs.kovacs\}@ur.de}\\
}

\date{\today}

\maketitle

\begin{abstract}
	An evolving surface finite element discretisation is analysed for the evolution of a closed two-dimensional surface governed by a system coupling a generalised forced mean curvature flow and a reaction--diffusion process on the surface, inspired by a gradient flow of a coupled energy. Two algorithms are proposed, both based on a system coupling the diffusion equation to evolution equations for  geometric quantities in the velocity law for the surface.
	One of the numerical methods is proved to be convergent in the $H^1$ norm with optimal-order for finite elements of degree at least two.
	We present numerical experiments illustrating the convergence behaviour and demonstrating the qualitative properties of the flow: preservation of mean convexity, loss of convexity, weak maximum principles, and the occurrence of self-intersections.
\end{abstract}

\section{Introduction}
\label{section:intro}

In this paper we propose and analyse an evolving  surface finite element semi-discretisation of a geometric partial differential equation (PDE) system that couples a forced mean curvature flow to a diffusion equation on the surface.  The unknowns are 
 a time dependent   two-dimensional closed, orientable, immersed surface $\Ga \subset\R^3$, and a time and spatially varying  surface concentration $u$. 

The coupled mean curvature--diffusion flow  system is
\begin{subequations}
\label{eq:MCFdiff}
	\begin{align}
		v = &\ - F(u,H) \nu, \label{mcfeqn}\\
	\mat u =&\ -(\nabla_{\Gamma}\cdot v) u + \nb_\Ga \cdot \blueon \big( \D(u) \nb_\Ga u \big), \blueoff  \label{diffeqn}
\end{align}
\end{subequations}
where $F$ and $\blueon \D$ are  given sufficiently smooth functions. Associated with the surface $\Ga$ are the geometric quantities  the mean curvature $H$, the oriented  continuous unit  normal field of the surface $\nu$, and $v$ the velocity of the evolving surface $\Ga$, where $V = v \cdot \nu$ denotes the normal velocity. In the case that $\Ga$ encloses a domain we always choose the unit outward pointing normal field.

A special case, inspiring this work, \blueon  with $F(u,H)=g(u)H$, $\D(u)=G''(u)$ where $g(u) = G(u) - G'(u)u$ and $G(\cdot)$ is given arises as   \blueoff the  $(L^2,H^{-1})$-gradient flow of the coupled energy, \cite{diss_Buerger,ABG_MCFdiff}, 
\begin{equation}
\label{eq:surface energy functional - intro}
	\calE = \calE(\Ga,u) = \int_\Ga G(u) ,
\end{equation}
yielding
\begin{subequations}
\label{eq:MCF diff - gradient flow}
	\begin{align}
		v = &\ - g(u) H \nu, \\
		\mat u = &\ -(\nabla_{\Gamma}\cdot v) u +\nb_\Ga \cdot \blueon \big( G''(u) \nb_\Ga u \big). \blueoff  	\end{align}
\end{subequations}

It is important to note that \eqref{eq:MCFdiff} contains not only the gradient flow  of \cite{diss_Buerger,ABG_MCFdiff} as a special case, but numerous other geometric flows as well. Examples are pure mean curvature flow \cite{Huisken1984}, the generalised mean curvature flows $v=-V(H)\nu$, see, e.g., \cite{HuiskenPolden},  examples in \cite{MCF_generalised} and \cite{BGN_survey}, additively forced mean curvature flow \cite{BarreiraElliottMadzvamuse2011,CGG}, and \cite{MCF_soldriven} (see also the references therein).  Also it arises as a sub-system  in  coupled bulk--surface models such as that for tumour growth considered in \cite{EylesKingStyles2019}.

\subsection{\bf Notation for evolving hypersurfaces}

We adopt  commonly used  notation for  surface and  geometric partial differential equations. Our setting is that the evolution  takes an initial $C^k$ hypersurface  $\Gamma^0\subset \mathbb R^3$ and an initial distribution $u^0 \colon \Gamma^0\rightarrow \mathbb R$ and evolves the surface 
 so that $\Gamma(t)\equiv\Ga[X]\subset\R^3$ is the image
$$
	\Ga[X] \equiv \Ga[X(\cdot,t)] = \{ X(p,t) \mid p \in \Ga^0 \}, \quad X(\cdot,0) 
	= \Id_{\Ga^0} $$
of a smooth mapping $X\colon \Ga^0\times [0,T]\to \R^3$ such that $X(\cdot,t)$ is the parametrisation of an orientable, immersed hypersurface  for every $t$.
We denote by  $v(x,t)\in\R^3$ at a point $x=X(p,t)\in\Ga[X(\cdot,t)]$ the   \emph{velocity} defined by \begin{equation}
\label{eq:velocity ODE}
	v(X(p,t),t)=\partial_t X(p,t).
\end{equation}
For a function $\eta(x,t)$ ($x\in \Ga[X]$, $0\le t \le T$) we denote the \emph{material derivative} (with respect to the parametrization $X$) as
$$
	\mat \eta(x,t) = \frac \d{\d t} \,\eta(X(p,t),t) \quad\hbox{ for } \ x=X(p,t).
$$
 On any regular surface $\Ga\subset\R^3$, we denote by $\nabla_{\Ga}\eta\colon \Ga\to\R^3$ the  \emph{tangential gradient} of a function $\eta\colon \Ga\to\R$, and in the case of a vector-valued function $\eta=(\eta_1,\eta_2,\eta_3)^T\colon \Ga\to\R^3$, we let
$\nabla_{\Ga}\eta=
(\nabla_{\Ga}\eta_1,
\nabla_{\Ga}\eta_2,
\nabla_{\Ga}\eta_3)$. We thus use the convention that the gradient of $\eta$ has the gradient of the components as column vectors, (in agreement with gradient of a scalar function is a column vector). 
We denote by $\nabla_{\Ga} \cdot \eta = \text{tr}(\nbg \eta)$ the \emph{surface divergence} of a vector field $\eta$ on $\Ga$, 
and by 
$\varDelta_{\Ga} \eta=\nabla_{\Ga}\cdot \nabla_{\Ga}\eta$ the \emph{Laplace--Beltrami operator} applied to $\eta\colon \Ga\to\R$; see the review \cite{DeckelnickDE2005} or \cite[Appendix~A]{Ecker2012}, or any textbook on differential geometry for these notions. 


We suppose that $\Gamma\t$ is an orientable, immersed hypersurface for all $t$.
In the case that $\Gamma$ is  the boundary of a bounded open set $\Omega
\subset \mathbb R^3$ we orient the  unit  normal vector field
$\n\colon\Ga\to\R^3$ to point out of $\Omega$.  The  surface gradient of the
normal field contains the (extrinsic) curvature data of the surface $\Ga$. At every $x\in\Ga$, the matrix of the extended Weingarten map,
$$
	A(x)=\nabla_\Ga \n(x) , 
$$ 
is symmetric and of size $3\times 3$ (see, e.g., \cite[Proposition~20]{Walker2015}). Apart from the eigenvalue $0$ (with eigenvector $\n$), its other two eigenvalues are the principal curvatures $\kappa_1$ and $\kappa_2$. They determine the fundamental quantities
\begin{align}
\label{eq:def - H and A2}
	H:={\rm tr}(A)=\nabla_\Gamma\cdot \nu=\kappa_1+\kappa_2, \qquad |A|^2 = \kappa_1^2 +\kappa_2^2 ,
\end{align}
where $|A|$ denotes the Frobenius norm of the matrix $A$.
Here, the mean curvature $H$ is, as in most of the literature, taken without the factor 1/2. 
In this setting, the mean curvature of a sphere is positive. 

For an evolving surface $\Ga$ with normal velocity $v=V\nu$, using that $\nbg f \cdot \nu = 0$ for any function $f$, we have the fundamental equation
\begin{equation}
\label{eq:divergence of velocity}
	\nabla_\Gamma\cdot v = \nabla_\Gamma\cdot (V\nu) = \nbg V \cdot \nu + V \nbg \cdot \nu = V H.
\end{equation}


The following   geometric identities hold for any sufficiently smooth evolving surface $\Gamma(t)$, (see for example  \cite{Huisken1984,Ecker2012,BGN_survey}):
\begin{align} 
	\nabla_{\Ga} H =&\ \Delta_{\Ga} \nu + |A|^2 \nu, \label{eq:Hui2}\\
	\mat \nu = &\ - \nb_{\Ga} V, \label{eq:Hui1}\\
	\mat H = &\ - \Delta_{\Ga} V -  |A|^2 V \label{eq:Hui3}.
\end{align}
They  are fundamental in the derivation of the system  of evolution equations discretised in this paper.

\subsection{\bf Our approach}
\label{section:approach}

The key idea of our approach is that it is based on a system of  evolution equations  coupling the two equations \eqref{mcfeqn} and \eqref{diffeqn} to parabolic equations for geometric variables in the velocity law. This approach was first used for mean curvature flow \cite{MCF}. The system is derived using the geometric identities 
\eqref{eq:Hui2}, \eqref{eq:Hui1} and \eqref{eq:Hui3}. Using the notation 
$\partial_i\cdot,\; i=1,2$  for appropriate partial derivatives, we prove the following lemma.

\begin{lemma}
\label{lemma:evolution equations}
	Let $\Ga[X]$ and $u$ be sufficiently smooth solutions of the  equations  \eqref{mcfeqn}--\eqref{diffeqn}. Suppose that $F,K:\mathbb R^2\rightarrow \mathbb R$ are sufficiently smooth, satisfy
	\begin{equation}
	\label{eq:velocity law-alt}
		r=-F(s,q) \quad \Longleftrightarrow \quad q=-K(s,r) , \qquad \forall r,s,q\in \mathbb R,
	\end{equation}
	\blueon and in addition assume that $\pa_2 F(u,H)$ is positive. \blueoff 
Then the normal vector $\nu$, the mean curvature $H$ and the normal velocity $V$ satisfy  the two following systems of  non-linear parabolic evolution equations:
	\begin{align}
		\mat H = &\  \laplace_{\Ga[X]} \big( F(u,H) \big) +  |A|^2 F(u,H), \label{H_F}\\
		\frac{1}{\partial_2F(u,H)} \, \mat \nu = &\  \laplace_{\Ga[X]} \nu +  |A|^2 \nu +\frac{\partial_1F(u,H)}{ \partial_2F(u,H)}  \nb_{\Ga[X]} u, \label{normalF}
		\end{align}
and		  
 	\begin{align}
		 \pa_2 K(u,V) \, \mat V = &\  \laplace_{\Ga[X]} V  +  |A|^2 V  - \pa_1 K(u,V) \, \mat u , \label{V_K}\\
                 \partial_2K(u,V) \, \mat \nu = &\  \laplace_{\Ga[X]} \nu +  |A|^2 \nu + \partial_1K(u,V)  \nb_{\Ga[X]} u \label{normalK}. 
			\end{align} 

\end{lemma}
\begin{proof}

These two sets of equations are an easy consequence of the geometric identities
\eqref{eq:def - H and A2},  \eqref{eq:Hui2}--\eqref{eq:Hui3}, and the following calculations
\begin{align*}
	\mat \nu = &\ - \nb_{\Ga[X]} V 
	=  \nb_{\Ga[X]} \big( F(u,H) \big) \\
	= &\ \partial_1F(u,H) \nb_{\Ga[X]} u+\partial_2F(u,H) \nb_{\Ga[X]} H,  
\end{align*}
as well as
\begin{align*}
	\nbg H = &\ - \nb_{\Ga[X]} \big( K(u,V) \big) \\
	= &\ - \pa_2 K(u,V) \nb_{\Ga[X]} V - \pa_1 K(u,V) \nb_{\Ga[X]} u . 
\end{align*}\qed
\end{proof}


Employing the lemma above we see that \blueon a sufficiently smooth \blueoff solution of the original initial value problem \eqref{eq:MCFdiff} \blueon also satisfies  two  other \blueoff different problems involving parabolic PDE systems in which the  dependent variables are a  parametrised surface $\Ga[X]$, the velocity $v$ of $\Ga$, a surface  concentration field $u$, and either the variables $\nu$ and $V$ or $\nu$ and $H$.  In these problems  the variables $\nu,V$ or $\nu,H$ are considered to be \emph{independently evolving unknowns}, rather than being determined by the  associated geometric quantities of the surface $\Ga[X]$ (in contrast to the methods of Dziuk \cite{Dziuk90} or Barrett, Garcke, and N\"urnberg \cite{BGN2008}, etc.).
 \begin{problem}\label{P1}
Given $\{\Gamma^0,u^0,\nu^0,V^0\}$, find 
for $t\in (0,T]$  functions $
\{
X(\cdot, t)\colon\Gamma^0\rightarrow \mathbb R^3,\; v(\cdot,t)\colon\Ga[X(\cdot,t)]
\to \R^3,\; u\colon
\Ga[X(\cdot,t)]
\rightarrow \mathbb R,\; \nu(\cdot,t)\colon\Gamma[X(\cdot,t)]\rightarrow
\mathbb R^3$, $V(\cdot,t)\colon\Gamma[X(\cdot,t)]\rightarrow \mathbb R\}$ 
such that 
\begin{subequations}
\label{eq:coupled system - P1}
\begin{align}
	\label{eq:coupled system - P1 - ODE for positions} 
	\pa_t X = &\ v \circ X, \\
	\label{eq:coupled system - P1 - velocity law} 
	v = &\  V \nu, \\
	\label{eq:coupled system - P1 - normal}
	\pa_2 K(u,V) \, \mat \nu = &\  \laplace_{\Ga[X]} \nu + |A|^2 \nu + \partial_1K(u,V) \, \nb_{\Ga[X]} u , \\
	\label{eq:coupled system - P1 - V}
	\pa_2 K(u,V) \, \mat V = &\  \laplace_{\Ga[X]} V  +  |A|^2 V  -
\pa_1 K(u,V) \, \mat u,  \\
	\label{eq:coupled system - P1 - u}
	\mat u + u \, \nb_{\Ga[X]} \cdot v = &\ \nb_{\Ga[X]} \cdot \big( \blueon \D(u) \nb_{\Ga[X]} u \blueoff \big),  
	\end{align}
\end{subequations}
 with initial data
\begin{alignat*}{3}
	X(\cdot,0) =&\ \Id_{\Gamma^0}, & \qquad \nu(\cdot,0)= &\ \nu^0 ,\\ 
	V(\cdot,0) =&\ V^0 , & \qquad u(\cdot,0) = &\ u^0 ,
\end{alignat*}
where $\nu^0$ is the unit normal to $\Gamma^0$ and $V^0 = -F(u^0,H^0)$ with $H^0$ being the mean curvature of $\Gamma^0$.

\end{problem}

\begin{problem}\label{P2}
Given $\{\Gamma^0,u^0,\nu^0,H^0\}$, find  for $t\in (0,T]$ the functions 
$
\{X(\cdot,t)\colon\Gamma^0\rightarrow \mathbb R^3,u\colon
\Ga[X(\cdot,t)] 
\rightarrow \mathbb R, v(\cdot,t)\colon\Ga[X(\cdot,t)] \to \R^3$,  $\nu(\cdot,t)\colon\Gamma[X(\cdot,t)]\rightarrow \mathbb R^3$, \linebreak $V(\cdot,t)\colon\Gamma[X(\cdot,t)]\rightarrow \mathbb R\}$ 
such that 
\begin{subequations}
\label{eq:coupled system - P2}
\begin{align}
	\label{eq:coupled system - P2 - ODE for positions} 
	\pa_t X = &\ v \circ X ,\\
	\label{eq:coupled system - P2 - velocity law-H} 
	v = &\  -F(u,H) \nu, \\
	\label{eq:coupled system - P2 - normal-H}
	\frac{1}{\partial_2 F(u,H)} \, \mat \nu = &\  \laplace_{\Ga[X]} \nu +  |A|^2 \nu + \frac{\partial_1 F(u,H)}{ \partial_2 F(u,H)}  \nb_{\Ga[X]} u , \\
	\label{eq:coupled system - P2 - H}
	\mat H = &\  \laplace_{\Ga[X]} \big( F(u,H) \big) +  |A|^2
F(u,H),  \\
	\mat u + u \, \nb_{\Ga[X]} \cdot v = &\ \nb_{\Ga[X]} \cdot \big( \blueon \D(u) \nb_{\Ga[X]} u \blueoff  \big),  
	\end{align}
\end{subequations}
 with initial data
\begin{alignat*}{3}
	X(\cdot,0) =&\ \Id_{\Gamma^0}, & \qquad \nu(\cdot,0)= &\ \nu^0 ,\\ 
	H(\cdot,0) =&\ H^0 , & \qquad u(\cdot,0) = &\ u^0 ,
\end{alignat*}
where $\nu^0$ and $H^0$ are, respectively,   the unit  normal to  and mean curvature  of  $\Gamma^0$.

\end{problem}

The idea is to discretise these systems using the evolving surface finite element method, see, e.g., \cite{DziukElliott_ESFEM}, and also \cite{Demlow2009,highorderESFEM}. 
The same approach was successfully used previously for mean curvature flow \cite{MCF}, also with additive forcing \cite{MCF_soldriven}, and in arbitrary codimension \cite{MCF_codim}, for Willmore flow \cite{Willmore}, and for generalised mean curvature flows \cite{MCF_generalised}.

\subsection{\bf Main results}

In Theorem~\ref{theorem:semi-discrete error estimates}, we state and prove optimal-order time-uniform $H^1$ norm error estimates for the spatial semi-discre\-ti\-sation, with finite elements of degree at least $2$, in all variables of Problem~\ref{P1}, over time intervals on which the solution remain sufficiently regular. This excludes the formation of singularities, but not self-intersections. 
We expect that an analogous proof would suffice for the other system Problem~\ref{P2} but due to length this is not presented here.
The convergence proof separates the questions of consistency and stability. 
Stability is proved via \emph{energy estimates}, testing with the errors and also with their time derivatives. Similarly to previous works, the energy estimates are performed in the matrix--vector formulation, and they use technical lemmas comparing different quantities on different surfaces, cf.~\cite{KLLP2017,MCF}. Due to the non-linear structure of the evolution equations in the coupled system we will also need similar but new lemmas estimating differences of solution-dependent matrices, cf.~\cite{MCF_generalised}. A key issue in the stability proof is to establish a $W^{1,\infty}$ norm error bounds for all variables. These are obtained from the time-uniform $H^1$ norm error estimates via an inverse inequality. 

In \cite[Chapter~5]{diss_Buerger} B\"urger  proved qualitative properties for the continuous coupled flow \eqref{eq:MCF diff - gradient flow} with energy \eqref{eq:surface energy functional - intro}, for example the preservation of mean convexity,  the  possible loss of convexity, \blueon the existence  of \blueoff a  weak maximum principle for the diffusion equation, the decay of energy, and the existence of self-intersections. \blueon These properties are enjoyed by our evolving surface finite element method as   illustrated  in the  numerical  simulations  in Section~\ref{section:numerics}.\blueoff
\subsection{\bf Related numerical analysis}
 
Numerical methods for related problems have been proposed and studied in many papers. We first restrict our literature overview for numerical methods for at least two-dimensional \emph{surface} evolutions.

Algorithms for mean curvature flow were proposed, e.g., by Dziuk in \cite{Dziuk90}, in \cite{BGN2008}, and in \cite{ElliottFritz_DT} based on the DeTurck trick. The first provably convergent algorithm was proposed and analysed in \cite{MCF}, while \cite{MCF_soldriven} extends these convergence results to additively forced mean curvature flow coupled to a semi-linear diffusion equation on the surface. Recently, Li \cite{MCF_Dziuk_Li} proved convergence of Dziuk's algorithm, for two-dimensional surfaces requiring surface finite elements of degree $k \geq 6$. 

Evolving surface finite element based algorithms for diffusion equations on evolving surface were analysed, for example, in \cite{DziukElliott_ESFEM,DziukElliott_L2}, in particular non-linear equations were studied in \cite{KPower_quasilinear,surface_maxreg}.
On the numerical analysis of both problems we also refer to the comprehensive survey articles \cite{DeckelnickDE2005,DziukElliott_acta}, and \cite{BGN_survey}.
For curve shortening flow coupled to a diffusion on a closed \emph{curve} optimal-order finite element semi-discrete error estimates were shown in \cite{PozziStinner_curve}, while \cite{BDS} have proved convergence of the corresponding backward Euler full discretisation.  The case of open curves with a fix boundary was analysed in \cite{StylesVanYperen2020}.  
For forced-elastic flow of curves semi-discrete error estimates were proved in \cite{PozziStinner_elastic_curve}. 
For mean curvature flow coupled to a diffusion process on a \emph{graph} optimal-order fully discrete error bounds were recently shown in \cite{DeckelnickStyles2021}.

\subsection{\bf Outline}

The paper is organised as follows. 
Section~\ref{section:intro} introduces basic notation and geometric quantities, and it is mainly devoted to the derivation of the two coupled systems.
In Section~\ref{section:weak formulation and properties} we present the weak formulations of the coupled problems, and explore the properties of the coupled flow.
In Section~\ref{section:finite element discretisation} we briefly recap the evolving surface finite element method, define interpolation operators and Ritz maps.
Section~\ref{section:relating surfaces} presents important technical results relating different surfaces.
In Section~\ref{section:semi-discretisation} we present the semi-discrete systems, while Section~\ref{section:matrix-vector form} presents their matrix--vector formulations, and the error equations.
Section~\ref{section:stability analysis} contains the most important results of the paper: consistency and stability analysis, as well as our main result which proves optimal-order semi-discrete error estimates.
Section~\ref{section:proof - estimates} and \ref{section:proof - stability} are devoted to the proofs of the results presented in Section~\ref{section:stability analysis}.
Finally, in Section~\ref{section:numerics} we describe an efficient fully discrete scheme, based on linearly implicit backward differentiation formulae. Then we present numerical experiments which illustrate and complement our theoretical results. We present numerical experiments testing convergence, and others which preserve mean convexity, but lose convexity, report on weak maximum principles, energy decay, and on an experiment with self-intersection.

\section{Weak formulation, its properties, and examples}
\label{section:weak formulation and properties}

Throughout the paper we will assume the following properties of the nonlinear functions:
%
%
%
%
%
\blueon 
\begin{enumerate}
	\item \label{eq:assumptions on pa_1 F}
	$ \frac{\pa_1F}{\pa_2 F}$  is locally Lipschitz continuous, 
	
	\item \label{eq:assumptions on pa_2 F}
	$\frac{1}{\pa_2 F}$ is positive and locally Lipschitz continuous,
	
	\item \label{eq:assumptions on pa_1 K}
	$\pa_1 K$ and $\pa_2 K$ are locally Lipschitz continuous,
	
	\item \label{eq:assumptions on pa_2 K}
	$\pa_2 K(u,V)$ is positive,
	
	\item \label{eq:assumptions D}
	$\D$ satisfies	$0 < D_0 \leq \D(\cdot) \leq D_1$ and $ \D'$ is locally Lipschitz continuous.
\end{enumerate}
The domain of definitions of the above nonlinearities are depending on the particular problem at hand. 
These properties hold on a compact neighbourhood of the exact smooth solution, on which $\pa_2 K(u,V)$ and $1/\pa_2F(u,H)$ are bounded from above and below by positive constants, and all functions are Lipschitz continuous.
\blueoff

\subsection{\bf Weak formulations}

\subsubsection*{Weak formulation of Problem~\ref{P1}}
The weak formulation of Problem~\ref{P1} reads: Find $X\colon \Ga^0 \to \R^3$ defining the (sufficiently smooth) surface $\Ga[X]$ with velocity $v$, and $\nu \in L^2_{H^1(\Ga[X])^3}$ with $\mat \nu \in L^2_{L^2(\Ga[X])^3}$, $V \in L^2_{H^1(\Ga[X])}$ with $\mat V \in L^2_{L^2(\Ga[X])}$, and $u \in L^2_{H^1(\Ga[X])}$ with $\mat u \in L^2_{L^2(\Ga[X])}$ such that, denoting $A = \nb_{\Ga[X]} \nu$ and $|\cdot|$ the Frobenius norm, 
\begin{subequations}
\label{eq:weak formulation - P1}
	\begin{align}
		&\ \quad v = V \nu , \\[1mm]
		&\ \int_{\Ga[X]} \!\!\!\! \pa_2 K(u,V) \, \mat \nu \cdot \vphi^\nu + \int_{\Ga[X]} \!\!\!\! \nb_{\Ga[X]} \nu \cdot \nb_{\Ga[X]} \vphi^\nu 
		\nonumber  \\ &\ \qquad\qquad\qquad 
		= \int_{\Ga[X]} \!\!\!\! |A|^2 \nu \cdot \vphi^\nu + \int_{\Ga[X]} \!\!\!\! \pa_1 K(u,V) \nb_{\Ga[X]} u \cdot \vphi^\nu  , \\ 
		&\ \int_{\Ga[X]} \!\!\!\! \pa_2 K(u,V) \, \mat V \, \vphi^V + \int_{\Ga[X]} \!\!\!\! \nb_{\Ga[X]} V \cdot \nb_{\Ga[X]} \vphi^V 
		\nonumber  \\ &\ \qquad\qquad\qquad 
		= \int_{\Ga[X]} \!\!\!\! |A|^2 V \vphi^V - \int_{\Ga[X]} \!\!\!\! \pa_1 K(u,V) \, \mat u \, \vphi^V , \\
		\label{eq:weak formulation - P1 - diffusion eq}
		&\ \diff \bigg( \int_{\Ga[X]}  \!\!\! u \, \vphi^u \bigg) + \int_{\Ga[X]} \!\!\!\! \D(u) \, \nb_{\Ga[X]} u \cdot \nb_{\Ga[X]} \vphi^u = \int_{\Ga[X]} \!\!\!\! u \, \mat \vphi^u ,
	\end{align}
\end{subequations}
holds for all test functions $\vphi^\nu \in L^2_{H^1(\Ga[X])^3}$, $\vphi^V \in L^2_{H^1(\Ga[X])}$, and $\vphi^u \in L^2_{H^1(\Ga[X])}$ with $\mat \vphi^u  \in L^2_{L^2(\Ga[X])}$, together with the ODE for the positions \eqref{eq:velocity ODE}. The coupled weak system is endowed with initial data $\Ga^0$, $\nu^0$, $V^0$, and $u^0$. 
For the definition of the Bochner-type spaces $L^2_{L^2(\Ga[X])}$ and $L^2_{H^1(\Ga[X])}$, which consist of time-dependent functions spatially defined on an evolving hypersurface, we refer to \cite{AlphonseElliottStinner}.

\subsubsection*{Weak formulation of Problem~\ref{P2}}
The weak formulation of Problem~\ref{P2} reads:
Find $X\colon \Ga^0 \to \R^3$ defining the (sufficiently smooth) surface $\Ga[X]$ with velocity $v$, and $\nu \in L^2_{H^1(\Ga[X])^3}$ with $\mat \nu \in L^2_{L^2(\Ga[X])^3}$, $H\in L^2_{L^2(\Ga[X])}$ with $\mat H \in L^2_{L^2(\Ga[X])}$, $V \in L^2_{H^1(\Ga[X])}$, and $u \in L^2_{H^1(\Ga[X])}$ with $\mat u \in L^2_{L^2(\Ga[X])}$ such that, denoting $A = \nb_{\Ga[X]} \nu$ and $|\cdot|$ the Frobenius norm, 
\begin{subequations}
	\label{eq:weak formulation - P2}
	\begin{align}
	&\ \quad v = V \nu , \\[1mm]
	&\ \int_{\Ga[X]} \! \frac{1}{ \partial_2 F(u,H)} \, \mat \nu \cdot \vphi^\nu + \int_{\Ga[X]} \!\!\!\! \nb_{\Ga[X]} \nu \cdot \nb_{\Ga[X]} \vphi^\nu 
	\nonumber  \\ &\ \qquad\qquad\qquad 
	= \int_{\Ga[X]} \!\!\!\! |A|^2 \nu \cdot \vphi^\nu + \int_{\Ga[X]} \! \frac{\partial_1 F(u,H)}{ \partial_2 F(u,H)} \nb_{\Ga[X]} u \cdot \vphi^\nu  , \\ 
	&\ \int_{\Ga[X]} \!\!\!\! \mat H \, \vphi^H \blueon -
	\blueoff  \int_{\Ga[X]} \!\!\!\! \nb_{\Ga[X]} V \cdot \nb_{\Ga[X]} \vphi^H = - \int_{\Ga[X]} \!\!\!\! |A|^2 V \vphi^H , \\
	&\ \int_{\Ga[X]} \!\!\!\! V \, \vphi^V + \int_{\Ga[X]} \!\!\!\! F(u,H) \, \vphi^V = 0 , \\
	&\ \diff \bigg( \int_{\Ga[X]} u \, \vphi^u \bigg) + \int_{\Ga[X]} \D(u) \, \nb_{\Ga[X]} u \cdot \nb_{\Ga[X]} \vphi^u = \int_{\Ga[X]} u \, \mat \vphi^u ,
	\end{align}
\end{subequations}
holds for all test functions $\vphi^\nu \in L^2_{H^1(\Ga[X])^3}$, $\vphi^H \in L^2_{H^1(\Ga[X])}$, $\vphi^V \in L^2_{L^2(\Ga[X])}$, and $\vphi^u \in L^2_{H^1(\Ga[X])}$ with $\mat \vphi^u  \in L^2_{L^2(\Ga[X])}$, together with the ODE for the positions \eqref{eq:velocity ODE}. The coupled weak system is endowed with initial data $\Ga^0$, $\nu^0$, $H^0$, and $u^0$.

\subsection{\bf Properties of the weak solution}
\label{section:examples and properties}

\begin{enumerate}
	\item \emph{Conservation of mass:} This is easily seen by testing the weak formulation \eqref{eq:weak formulation - P1 - diffusion eq} with $\vphi^u \equiv 1$.
	\item
	\emph{Weak maximum principle:} By testing the diffusion equation with $\min{(u,0)}$  and assuming that
	$$0\le u^0\le M^0,~~\mbox{a.e.~on}~~\Gamma^0 ,$$
	we find, cf.~\cite[Section~5.4]{diss_Buerger},
	\begin{equation}
	0\le u(\cdot,t),~~\mbox{a.e.~on}~~\Gamma[X].
	\end{equation}
	\item \emph{Energy bounds:}
	\blueon Let $G$ be any convex function, for which $g(u) = G(u) - G'(u)u \geq 0$. 
	Taking the time derivative of the energy $\int_{\Ga[X]} G(u)$, and using the diffusion equation \eqref{diffeqn} and \eqref{eq:divergence of velocity}, we obtain,
	\begin{align*}
	&\ \diff \bigg( \int_{\Ga[X]} G(u) \bigg) \\
	= &\  \int_{\Ga[X]} G'(u) \mat u + \int_{\Ga[X]}  (\nb_{\Ga[X]} \cdot v) G(u) \\
	= &\ \int_{\Ga[X]} G'(u) \Big( \nb_{\Ga[X]} \cdot \big( \D(u) \nb_{\Ga[X]} u \big) - u (\nb_{\Ga[X]} \cdot v) \Big) + \int_{\Ga[X]} (\nb_{\Ga[X]} \cdot v) G(u) \\
	= &\ \int_{\Ga[X]} G'(u) \Big( \nb_{\Ga[X]} \cdot \big( \D(u) \nb_{\Ga[X]} u \big) - u V H \Big) + \int_{\Ga[X]} V H G(u) \\
	= &\ \int_{\Ga[X]} G'(u) \nb_{\Ga[X]} \cdot \big( \D(u) \nb_{\Ga[X]} u \big)
	+ \int_{\Ga[X]} \big(G(u)- G'(u) u ) \big) V H \\
	= &\ - \int_{\Ga[X]} \D(u) G''(u)
	|\nb_{\Ga[X]} u|^2
	+ \int_{\Ga[X]} g(u) V H 
	\end{align*}
	yielding
	\begin{equation}
	\label{eq:energy identity}
		\diff \bigg( \int_{\Ga[X]} G(u) \bigg) +  \int_{\Ga[X]} \D(u) G''(u) |\nb_{\Ga[X]} u|^2 = \int_{\Ga[X]} g(u) V H.
	\end{equation}

Energy decrease and a priori estimates follow provided that  $VH \leq 0$, (note that $\D(u) G''(u) \geq 0$, $g(u) = G(u) - G'(u)u \geq 0$ are already assumed). This inequality holds assuming $K(u,V)V\geq 0$ and $F(u,H)H\geq 0$, respectively, for Problem~\ref{P1} and Problem~\ref{P2}.
For system \eqref{eq:MCF diff - gradient flow} the energy identity \eqref{eq:energy identity} leads to the natural energy decrease for the gradient flow \cite[Section~3.3--3.4]{diss_Buerger}, \cite{ABG_MCFdiff}.
\blueoff 
\end{enumerate}


\section{Finite element discretisation}
\label{section:finite element discretisation}
\subsection{\bf Evolving surface finite elements}\label{section:ESFEM}

For the spatial semi-discretisation of the weak coupled systems \eqref{eq:weak formulation - P1} and \eqref{eq:weak formulation - P2} we will use the evolving surface finite element method (ESFEM) \cite{Dziuk88,DziukElliott_ESFEM}.  We use curved simplicial finite elements and basis functions defined by continuous piecewise polynomial basis functions of degree~$k$ on triangulations, as defined in \cite[Section~2]{Demlow2009}, \cite{highorderESFEM} and \cite{EllRan21}.

\subsubsection{Surface finite elements}

The given smooth initial surface $\Ga^0$ is triangulated by an admissible family of triangulations $\mathcal{T}_h$ of degree~$k$ \cite[Section~2]{Demlow2009}, consisting of curved simplices  of maximal element diameter $h$; see \cite{DziukElliott_ESFEM} and \cite{EllRan21} for the notion of an admissible triangulation, which includes quasi-uniformity and shape regularity. Associated with the triangulation is a collection of unisolvent nodes   $p_j$ $(j=1,\dots,\dof)$ for which nodal variables define the   piecewise polynomial basis functions $\{\phi_j\}_{j=1}^N$.

Throughout we consider triangulations $\Gamma_h[\bfy]$ isomorphic  to $\Gamma_h^0$ with respect to the labelling of the vertices, faces, edges and nodes. We use the notation $\bfy\in \mathbb R^{3N}$  to denote the positions $y_j = \bfy|_j$, of nodes mapped to $p_j$ so that

$$\Gamma_h[\bfy]:= \bigg\{q=\sum_{j=1}^N y_j\phi_j(p) ~\Big|~p\in \Gamma_h^0 \bigg\}.$$
That is we assume there is a unique pullback $\tilde p \in \Ga_h^0$ such that for each $q\in \Ga_h[\bfy]$ it holds $q=\sum_{j=1}^\dof y_j\phi_j(\tilde p)$. 

We define globally continuous finite element \emph{basis functions} using the pushforward
$$
	\phi_i[{\bfy}]\colon \Ga_h[\bfy]\to\R, \qquad i=1,\dotsc,\dof
$$
such that 
$$
\phi_i[{\bfy}](q)=\phi_i(\tilde p), \quad q\in \Ga_h[\bfy].
$$
Thus they have the property that on every curved  triangle their pullback to the reference triangle is polynomial of degree $k$, which satisfy at the nodes $\phi_i[\bfy](y_j) = \delta_{ij}$ for all $i,j = 1,  \dotsc, \dof$.
These  basis functions define  a  finite element space on $\Ga_h[\bfy]$
\begin{equation*}
	S_h[\bfy] = S_h(\Ga_h[\bfy]) = \spn\big\{ \phi_1[\bfy], \phi_2[\bfy], \dotsc, \phi_\dof[\bfy] \big\} .
\end{equation*}

We associate with a vector $\bfz =\{z_j\}_{j=1}^\dof \in \mathbb R^\dof$ a finite element function $z_h\in S_h[\bfy]$ by
$$z_h(q)=\sum_{j=1}^\dof z_j\phi_j[\bfy](q), \qquad q\in \Ga_h[\bfy].$$
For a finite element function $z_h\in S_h[\bfy]$, the tangential gradient $\nabla_{\Ga_h[\bfy]}z_h$ is defined piecewise on each curved  element.

\subsubsection{Evolving surface finite elements}
We set $\Ga_h^0$ to be an admissible  initial  triangulation that interpolates $\Ga^0$ at  the nodes $p_j$ and we denote by $\bfx^0$ the vector in $\R^{3\dof}$ that collects all nodes so $x_j^0=p_j$. Evolving  the $j$th node $p_j$  in time by a  velocity $v_j(t)\in C([0,T])$,  yields a  collection  of surface  nodes denoted by ${\bfx}(t) \in \R^{3\dof}$, with $ x_j(t)=\bfx (t)|_j$ at time $t$ and $\bfx(0)=\bfx^0$.
Given such a collection of surface nodes we may define an evolving  discrete surface by
$$\Ga_h[{\bfx}(t)]:= \bigg\{{X}_h(p,t):=\sum_{j=1}^Nx_j(t)\phi_j(p)~\Big|~p\in \Gamma_h^0 \bigg\}.$$
That is,   the discrete surface at time $t$ is parametrized by the initial discrete surface via the map $X_h(\cdot,t)\colon\Ga_h^0\to\Ga_h[\bfx(t)]$ 
which has the properties that $X_h(p_j,t)=x_j(t)$ for $j=1,\dots,\dof$,  $X_h(p_h,0) = p_h$ for all $p_h\in\Ga_h^0$
and for each $q\in \Ga_h[\bfx(t)]$ there exists a unique pullback $p(q,t)\in \Ga_h^0$ such that $q=\sum_{j=1}^Nx_j(t)\phi_j(p(q,t)).$
We assume that the discrete surface remains admissible, which -- in view of the $H^1$ norm error bounds of our main theorem -- will hold provided the flow map $X$ is sufficiently regular, see~Remark~\ref{remark:convergence result}.

We define globally continuous finite element \emph{basis functions} using the pushforward
$$
	\phi_i[{\bfx(t)}]\colon \Ga_h[\bfx(t)]\to\R, \qquad i=1,\dotsc,\dof
$$
such that 
$$
\phi_i[{\bfx}(t)](q)=\phi_i(p(q,t)), \quad q\in \Ga_h[{\bfx}(t)].
$$
Thus they have the property that on every curved evolving triangle their pullback to the reference triangle is polynomial of degree $k$, and which satisfy at the nodes $\phi_i[\bfx(t)](x_j) = \delta_{ij}$ for all $i,j = 1,  \dotsc, \dof$.
These  basis functions define  an evolving  finite element space on $\Ga_h[\bfx(t)]$
\begin{equation*}
	S_h[\bfx(t)] = S_h(\Ga_h[\bfx(t)]) = \spn\big\{ \phi_1[\bfx(t)], \phi_2[\bfx(t)], \dotsc, \phi_\dof[\bfx(t)] \big\} .
\end{equation*}

We define a {\it material derivative},  $\mat_h\cdot$, on the time dependent finite element space as the push forward of the time derivative of the pullback function. Thus the basis functions satisfy the \emph{transport property} \cite{DziukElliott_ESFEM}:
\begin{equation}
\label{eq:transport property of basis functions}
	\mat_h \phi_j[\bfx(t)] = 0 .
\end{equation}
It follows that for $\eta_h(\cdot,t) \in S_h[\bfx(t)]$, (with nodal values $(\eta_j\t)_{j=1}^N$), we have
$$\mat_h\eta_h=\sum_{j=1}^N\dot \eta_j(t) \phi_j[\bfx(t)] ,$$
where the dot denotes the time derivative $\d/\d t$.
The \emph{discrete velocity} $v_h(q,t) \in \R^3$ at a point $q=X_h(p,t) \in \Ga[X_h(\cdot,t)]$ is given by
$$
	\partial_t X_h(p,t) = v_h(X_h(p,t),t)=\sum_{j=1}^N{\dot x}_j(t)\phi_j(p), \qquad p\in \Gamma_h^0.
$$
%

\begin{definition}[Interpolated-surface]
\label{def:interpolated surface}
Let  $\bfx^*(t)\in \mathbb R^{3\dof}$ and  $\bfv^*(t)\in \mathbb R^{3\dof}$  be the vectors with components $x_j^*(t)=X(p_j,t)$, $v^*_j(t):=\dot X(p_j,t)$ where $X(\cdot,t)$ solves Problem~\ref{P1} and \ref{P2}. The evolving triangulated surface $\Gamma_h[\bfx^*(t)]$ associated with $X_h^*(\cdot,t)$ is called the {\it interpolating surface}, with \emph{interpolating velocity} $v_h^*\t$.
\end{definition}
The interpolating surface $\Gamma_h[\bfx^*(t)]$ associated with $X_h^*(\cdot,t)$  is assumed to be admissible for all $t\in [0,T]$, which indeed holds provided the flow map $X$ is sufficiently regular, see~Remark~\ref{remark:convergence result}.

\subsection{\bf Lifts}
\label{section:lifts}


Any finite element function $\eta_h$ on the discrete surface $\Ga_h[\bfx\t]$, with nodal values $(\eta_j)_{j=1}^N$, is associated with a finite element function $\widehat \eta_h$ on the interpolated surface $\Ga_h[\xs\t]$ with the exact same nodal values. 
This can be further lifted to a function on the exact surface by using the \emph{lift operator} $\,^\ell$, mapping a function on the interpolated surface $\Ga_h[\xs\t]$ to a function on the exact surface $\Ga[X(\cdot,t)]$, via the identity, for $x \in \Ga_h[\xs\t]$,
\begin{equation*}
	x^\ell = x - d(x,t) \nu_{\Ga[X]}(x^\ell,t) , \qquad \text{and setting} \qquad \widehat \eta_h^\ell(x^\ell) = \widehat \eta_h(x) ,
\end{equation*}
using a signed distance function $d$, provided that the two surfaces are sufficiently close. For more details on the lift $\,^\ell$, see \cite{Dziuk88,DziukElliott_acta,Demlow2009}. The inverse lift is denoted by $\eta^{-\ell}\colon \Ga_h[\xs\t] \to \R$ such that $(\eta^{-\ell})^\ell = \eta$.

Then the \emph{composed lift operator} $\,^L$ maps finite element functions on the discrete surface $\Gamma_h[\bfx(t)]$ to functions on the exact surface $\Gamma[X(\cdot,t)]$, see \cite{MCF}, via the interpolated surface $\Gamma_h[\xs(t)]$, by 
$$
\eta_h^L = (\widehat \eta_h)^\ell.
$$

We introduce the notation
$$
	x_h^L(x,t) =  X_h^L(p,t) \in \Ga_h[\bfx(t)] \qquad \hbox{for} \quad x=X(p,t)\in\Ga[X(\cdot,t)],
$$
where, for $p_h \in \Ga_h^0$, from the nodal vector $\bfx\t$ we obtain the function $X_h(p_h,t) = \sum_{j=1}^N x_j\t \phi_j[\bfx(0)](p_h)$, while from $\bfu\t$, with $x_h \in \Ga_h[\bfx\t]$, we obtain $u_h(x_h,t) = \sum_{j=1}^N u_j\t \phi_j[\bfx(t)](x_h)$, and similarly for any other nodal vectors.

\subsection{\bf Surface mass and stiffness matrices and discrete norms}

For a triangulation $\Ga_h[\bfy]$ associated with the nodal vector $\bfy\in \mathbb R^{3N}$, we define the surface-dependent positive definite mass matrix $\bfM(\bfy) \in \R^{N \times N}$ and surface-dependent positive semi-definite stiffness matrix $\bfA(\bfy) \in \R^{N \times N}$:
\begin{align*}
	\bfM(\bfy)|_{ij} = \int_{\Ga_h[\bfy]} \!\!\!\! \phi_i[\bfy] \, \phi_j[\bfy] , \andquad \bfA(\bfy)|_{ij} = \int_{\Ga_h[\bfy]} \!\!\!\! \nb_{\Ga_h[\bfy]} \phi_i[\bfy] \cdot \nb_{\Ga_h[\bfy]} \phi_j[\bfy],
\end{align*}
and then set 
$$\bfK(\bfy) = \bfM(\bfy) + \bfA(\bfy).$$ 
For a pair of finite element functions $z_h,w_h\in S_h[\bfy]$  with nodal vectors $\bfz,\bfw$ we have
$$(z_h,w_h)_{L^2(\Ga_h[\bfy])}=\bfz^T\bfM(\bfy)\bfw ~~\mbox{and}~~~(\nabla_{\Ga_h[\bfy]}z_h,\nabla_{\Ga_h[\bfy]}w_h)_{L^2(\Ga_h[\bfy])}=\bfz^T\bfA(\bfy)\bfw . $$

These finite element matrices  induce discrete versions of Sobolev norms on the discrete surface $\Ga_h[\bfy]$. For any nodal vector $\bfz \in \R^{N}$, with the corresponding finite element function $z_h \in S_h[\bfy]$, we define the following (semi-)norms:
\begin{equation}
\label{eq:norms}
	\begin{aligned}
		\|\bfz\|_{\bfM(\bfy)}^{2} = &\ \bfz^T \bfM(\bfy) \bfz = \|z_h\|_{L^2(\Ga_h[\bfy])}^2 , \\
		\|\bfz\|_{\bfA(\bfy)}^{2} = &\ \bfz^T \bfA(\bfy) \bfz = \|\nb_{\Ga_h[\bfy]} z_h\|_{L^2(\Ga_h[\bfy])}^2 , \\
		\|\bfz\|_{\bfK(\bfy)}^{2} = &\ \bfz^T \bfK(\bfy) \bfz = \|z_h\|_{H^1(\Ga_h[\bfy])}^2 .
	\end{aligned}
\end{equation}

\subsection{\bf Ritz maps}
\label{section:defects}

Let the nodal vectors $\xs\t \in \R^{3N}$ and $\vs\t \in \R^{3N}$, collect the nodal values of the exact solution $X(\cdot,t)$ and $v(\cdot,t)$. Recalling Definition~\ref{def:interpolated surface}, the corresponding finite element functions $X_h^*(\cdot,t)$ and $v_h^*(\cdot,t)$ in $S_h[\xs\t]^3$ are the finite element interpolations of the exact solutions. 

The two Ritz maps below are defined following \cite[Definition~6.1]{highorderESFEM} (which is slightly different from \cite[Definition~6.1]{DziukElliott_L2} or \cite[Definition~3.6]{EllRan21}) and -- for the quasi-linear Ritz map -- following \cite[Definition~3.1]{KPower_quasilinear}.  
\begin{definition}
For any $w \in H^1(\Ga[X])$ the generalised Ritz map $\Rh w \in S_h[\xs]$ uniquely solves, for all $\vphi_h \in S_h[\xs]$,
\begin{equation}
\label{eq:definition Ritz map}
	\begin{aligned}
		&\ \int_{\Ga_h[\xs]} \!\!\!\! \Rh w \, \vphi_h + \int_{\Ga_h[\xs]} \!\!\!\! \nb_{\Ga_h[\xs]} \Rh w \cdot \nb_{\Ga_h[\xs]} \vphi_h \\
		= &\ \int_{\Ga[X]} \!\!\!\! w \, \vphi_h^\ell + \int_{\Ga[X]} \!\!\!\! \nb_{\Ga[X]} w \cdot \nb_{\Ga[X]} \vphi_h^\ell .
	\end{aligned}
\end{equation}
\end{definition}

\begin{definition}For any $u \in H^1(\Ga[X])$ and an arbitrary (sufficiently smooth) $\xi \colon \Ga[X] \to \R$ the  $\xi$-dependent Ritz map $\Rh^\xi u \in S_h[\xs]$  uniquely solves, for all $\vphi_h \in S_h[\xs]$,
\begin{equation}
\label{eq:definition quasi-linear Ritz map}
	\begin{aligned}
		&\ \int_{\Ga_h[\xs]} \!\!\! \Rh^\xi u \, \vphi_h + \int_{\Ga_h[\xs]} \!\!\!\! \D(\xi^{-\ell}) \, \nb_{\Ga_h[\xs]} \Rh^\xi u \cdot \nb_{\Ga_h[\xs]} \vphi_h \\
		= &\ \int_{\Ga[X]} \!\!\!\! u \vphi_h^\ell + \int_{\Ga[X]} \!\!\!\! \D(\xi) \, \nb_{\Ga[X]} u \cdot \nb_{\Ga[X]} \vphi_h^\ell ,
	\end{aligned}
\end{equation}
where $^{-\ell}$ denotes the inverse lift operator, cf.~Section~\ref{section:lifts}.
\end{definition}
We will also refer to $\Rh^\xi$ as quasi-linear Ritz map, since it is associated to a quasi-linear elliptic operator.

\begin{definition}
The Ritz maps $R_h$ and $R_h^\xi$ are then defined as the lifts of $\Rh$ and $\Rh^\xi$, i.e.~$R_h u  = (\Rh u)^\ell \in S_h[\xs]^\ell$ and $R_h^\xi u  = (\Rh^\xi u)^\ell \in S_h[\xs]^\ell$.
\end{definition}

\section{\bf{Relating different surfaces}}
\label{section:relating surfaces}
In this section from \cite{KLLP2017,MCF} we recall useful inequalities relating norms and semi-norms on differing surfaces. First recalling results in a general evolving surface setting, and then proving new results for the present problem.

Given a pair of triangulated surfaces $\Ga_h[\bfx]$ and $\Ga_h[\bfy]$ with nodal vectors $\bfx, \bfy \in \R^{3N}$, we may view  $\Ga_h[\bfx]$  as an evolution of $\Ga_h[\bfy]$ with a constant velocity $\bfe = (e_j)_{j=1}^N = \bfx - \bfy \in \R^{3N}$ yielding a family of intermediate surfaces.
\begin{definition}For $\theta \in [0,1]$  the {\it intermediate surface} $\Ga_h^\theta$ is defined by
	$$ \Ga_h^\theta = \Ga_h[\bfy+\theta\bfe].$$
For the  vectors $\bfx= \bfe + \bfy, \bfw,\bfz \in \R^N$, we define the corresponding finite element functions on $\Ga_h^\theta$:
\begin{align*}
	e_h^\theta = \sum_{j = 1}^N e_j \phi_j[\bfy+\theta\bfe] , \quad 
	w_h^\theta = \sum_{j = 1}^N w_j \phi_j[\bfy+\theta\bfe] , \quad \text{and} \quad
	z_h^\theta = \sum_{j = 1}^N z_j \phi_j[\bfy+\theta\bfe] .
\end{align*}
\end{definition}
Figure~\ref{figure:relating different surfaces} illustrates the described construction. 

\begin{figure}[htbp]
	\begin{center}
		\includegraphics[scale=1]{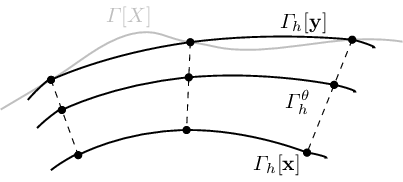}
		\caption{The construction of the intermediate surfaces $\Ga_h^\theta$ for quadratic elements.}
		\label{figure:relating different surfaces}
	\end{center}
\end{figure}

It follows from the evolving surface transport theorems for the $L^2$ and Dirichlet inner products, \cite{DziukElliott_acta}, that for arbitrary vectors $\bfw, \bfz \in \R^N$:
	\begin{align}
	\label{matrix difference M}
		\bfw^T (\bfM(\bfx)-\bfM(\bfy)) \bfz =&\ \int_0^1 \int_{\Ga_h^\theta} w_h^\theta (\nabla_{\Ga_h^\theta} \cdot e_h^\theta) z_h^\theta \; \d\theta, \\
	\label{matrix difference A}
		\bfw^T (\bfA(\bfx)-\bfA(\bfy)) \bfz =&\ \int_0^1 \int_{\Ga_h^\theta} \nb_{\Ga_h^\theta} w_h^\theta \cdot (D_{\Ga_h^\theta} e_h^\theta)\nb_{\Ga_h^\theta}  z_h^\theta \; \d\theta ,
	\end{align}
	where
	$D_{\Ga_h^\theta} e_h^\theta =  \textnormal{tr}(E^\theta) I_3 - (E^\theta+(E^\theta)^T)$ with $E^\theta=\nabla_{\Ga_h^\theta} e_h^\theta \in \R^{3\times 3}$.


The following results relate the mass and stiffness matrices for the discrete surfaces $\Ga_h[\bfx]$ and $\Ga_h[\bfy]$, they follow by the Leibniz rule, and are given in \cite[Lemma~4.1]{KLLP2017}, \cite[Lemma~7.2]{MCF}.
\begin{result}
\label{results collection}
In the above setting, if  
\begin{equation}
\label{Epsass}
	\eps := \| \nabla_{\Ga_h[\bfy]} e_h^0 \|_{L^\infty(\Ga_h[\bfy])} \leq \tfrac14 ,
\end{equation}
then the following hold:
\begin{enumerate}

%

\item
 For $0\le \theta\le 1$ and  $1 \le p \le \infty$  with a constant $c_p > 0$ independent of $h$ and $\theta$:

\begin{equation}
\label{eq:lemma - theta-independence}
	\begin{aligned}
		&\| w_h^\theta \|_{L^p(\Ga_h^\theta)} \leq c_p \, \|w_h^0 \|_{L^p(\Ga_h^0)} , 
		\qquad \| \nabla_{\Ga_h^\theta} w_h^\theta \|_{L^p(\Ga_h^\theta)} \le c_p \, \| \nabla_{\Ga_h^0} w_h^0 \|_{L^p(\Ga_h^0)}.
	\end{aligned}
\end{equation}

\item 
\begin{equation}
\label{norm equivalence}
	\begin{aligned}
		&\ \text{The norms $\|\cdot\|_{\bfM(\bfy+\theta\bfe)}$ and the norms $\|\cdot\|_{\bfA(\bfy+\theta\bfe)}$} 
		\\ &\
		\text{are $h$-uniformly equivalent for $0\le\theta\le 1$.}
	\end{aligned}
\end{equation}
\item For any $\bfw, \bfz \in \R^N$, with an $h$-independent constant $c>0$, we have the estimates
\begin{equation}
\label{matrix difference bounds}
	\begin{aligned}
		\bfw^T (\bfM(\bfx)-\bfM(\bfy)) \bfz \leq &\ c \, \eps \, \|\bfw\|_{\bfM(\bfy)} \|\bfz\|_{\bfM(\bfy)} , \\[1mm]
		\bfw^T (\bfA(\bfx)-\bfA(\bfy)) \bfz \leq &\ c \, \eps \, \|\bfw\|_{\bfA(\bfy)} \|\bfz\|_{\bfA(\bfy)} .
	\end{aligned}
\end{equation}
\item If $z_h \in W^{1,\infty}(\Ga_h[\bfy])$ then, for any $\bfw, \bfz \in \R^N$, with an $h$-independent constant $c>0$, we have
\begin{equation}
\label{matrix difference bounds e_x}
	\begin{aligned}
		\bfw^T (\bfM(\bfx)-\bfM(\bfy)) \bfz \leq &\ c \, \|\bfw\|_{\bfM(\bfy)} \|\bfe\|_{\bfA(\bfy)} \|z_h\|_{L^{\infty}(\Ga_h[\bfy])} , \\[1mm]
		\bfw^T (\bfA(\bfx)-\bfA(\bfy)) \bfz \leq &\ c \, \|\bfw\|_{\bfA(\bfy)} \|\bfe\|_{\bfA(\bfy)}  \|z_h\|_{W^{1,\infty}(\Ga_h[\bfy])}.
	\end{aligned}
\end{equation}

\end{enumerate}
\end{result}

\subsection{\bf Time evolving surfaces}
\begin{itemize}

\item
Let $\bfx:[0,T]\to\R^{3N}$ be a continuously differentiable vector defining a triangulated  surface $\Ga_h[\bfx(t)]$ for every $t\in[0,T]$ with time derivative $\bfv(t)=\dot\bfx(t)$ whose  finite element function $v_h(\cdot,t)$ satisfies
\begin{equation}
\label{vh-bound}
	\| \nabla_{\Ga_h[\bfx(t)]}v_h(\cdot,t) \|_{L^{\infty}(\Ga_h[\bfx(t)])} \le K_v, \qquad 0\le t \le T.
\end{equation}
With $\bfe=\bfx(t)-\bfx(s)=\int_s^t \bfv(r) \d r$, the bounds \eqref{matrix difference bounds} then yield the following bounds, which were first shown in Lemma~4.1 of \cite{DziukLubichMansour_rksurf}: for $0\le s, t \le T$ with $K_v|t-s| \le \tfrac14$, for arbitrary vectors $\bfw, \bfz \in \R^N$, 
we have with $C=c K_v$
\begin{equation}
\label{matrix difference bounds-t}
	\begin{aligned}
		\bfw^T \bigl(\bfM(\bfx(t))  - \bfM(\bfx(s))\bigr)\bfz \leq&\ C \, |t-s| \, \|\bfw\|_{\bfM(\bfx(t))}\|\bfz\|_{\bfM(\bfx(t))} , \\[1mm]
		\bfw^T \bigl(\bfA(\bfx(t))  - \bfA(\bfx(s))\bigr)\bfz \leq&\ C\,  |t-s| \, \|\bfw\|_{\bfA(\bfx(t))}\|\bfz\|_{\bfA(\bfx(t))}.   
	\end{aligned}
\end{equation}
Letting $s\to t$, this implies the bounds stated in Lemma~4.6 of~\cite{KLLP2017}:
\begin{equation}
\label{matrix derivative bounds}
	\begin{aligned}
		\bfw^T \diff \big( \bfM(\bfx(t)) \big) \bfz \leq &\ C \, \|\bfw\|_{\bfM(\bfx(t))}\|\bfz\|_{\bfM(\bfx(t))} , \\[1mm]
		\bfw^T \diff \big( \bfA(\bfx(t)) \big) \bfz \leq &\ C \, \|\bfw\|_{\bfA(\bfx(t))}\|\bfz\|_{\bfA(\bfx(t))} .
	\end{aligned}
\end{equation}
Moreover, by patching together finitely many intervals over which $K_v|t-s| \le \tfrac14$, we obtain that
\begin{equation}
\label{norm-equiv-t}
	\begin{aligned}
		&\ \text{the norms $\|\cdot\|_{\bfM(\bfx(t))}$ and the norms $\|\cdot\|_{\bfA(\bfx(t))}$}
		\\ &\
		\text{are $h$-uniformly equivalent for $0\le t \le T$.}
	\end{aligned}
\end{equation}
\end{itemize}

\subsection{\bf Variable coefficient  matrices}
\label{section:variable coefficient matrices}

Given  $\bfu, \bfV \in \R^N$ with associated finite element functions $u_h,V_h$ we define variable coefficient positive definite  mass matrix $\bfM(\bfx,\bfu,\bfV) \in \R^{N \times N}$ and positive semi-definite stiffness matrix $\bfA(\bfx,\bfu) \in \R^{N \times N}$:
\begin{align}
	\label{eq:solution-dependent mass matrix}
	\bfM(\bfx,\bfu,\bfV)|_{ij} = &\ \int_{\Ga_h[\bfx]} \!\!\!\! \pa_2 K(u_h,V_h) \ \phi_i[\bfx] \, \phi_j[\bfx] , \\
	\label{eq:solution-dependent stiffness matrix}
	\bfA(\bfx,\bfu)|_{ij} = &\ \int_{\Ga_h[\bfx]} \!\!\!\! \D(u_h) \, \nb_{\Ga_h[\bfx]} \phi_i[\bfx] \cdot \nb_{\Ga_h[\bfx]} \phi_j[\bfx] , 
\end{align}
for $i,j = 1,  \dotsc,N$.

The following lemma  is a variable coefficient variant of the estimates above  relating mass matrices, i.e.~\eqref{matrix difference bounds} and \eqref{matrix difference bounds e_x}. 

\begin{lemma}
\label{lemma:solution-dependent mass matrix diff}
	Let $\bfu, \us \in \R^N$ and $\bfV, \Vs \in \R^N$ be such that the corresponding finite element functions $u_h, u_h^*$ and $V_h, V_h^*$ have  $L^\infty$ norms bounded independently of $h$. Let \eqref{Epsass} hold. Then the following bounds hold, for arbitrary vectors $\bfw , \bfz \in \R^N$:
		\begin{align}
		\tag{i}
		\bfw^T \big( \bfM(\bfx,\bfu,\bfV) - \bfM(\bfy,\bfu,\bfV) \big) \bfz 
		&\leq C \, \| \nabla_{\Ga_h[y]} e_h^0 \|_{L^\infty(\Ga_h[\bfy])} \, \| \bfw \|_{\bfM(\bfy)} \, \| \bfz \|_{\bfM(\bfy)}, \\
		\tag{ii}
		\bfw^T \big( \bfM(\bfx,\bfu,\bfV) - \bfM(\bfy,\bfu,\bfV) \big) \bfz 
		&\leq C \, \|\bfe\|_{\bfA(\bfy)} \| \bfw \|_{\bfM(\bfy)} \, \| z_h \|_{L^\infty(\Ga_h[\bfy])} ,
		\end{align}
		and
		\begin{align}
		\nonumber
		&\ \bfw^T \big( \bfM(\bfx,\bfu,\bfV)-\bfM(\bfx,\us,\Vs) \big) \bfz \\
		\tag{iii}
		\leq &\ C \, \big( \| u_h - u_h^* \|_{L^\infty(\Ga_h[\bfy])} + \| V_h - V_h^* \|_{L^\infty(\Ga_h[\bfy])} \big) \, \| \bfw \|_{\bfM(\bfx)}  \, \| \bfz \|_{\bfM(\bfx)} , \\
		\nonumber
		&\ \bfw^T \big( \bfM(\bfx,\bfu,\bfV)-\bfM(\bfx,\us,\Vs) \big) \bfz \\
		\tag{iv}
		\leq &\ C \, \big( \| \bfu - \us \|_{\bfM(\bfx)} + \| \bfV - \Vs \|_{\bfM(\bfx)} \big) \, \| \bfw \|_{\bfM(\bfx)} \, \| z_h \|_{L^\infty(\Ga_h[\bfx])} .
		\end{align}
	The constant $C > 0$ is independent of $h$ and $t$, but depends on $\pa_2 K(u_h,V_h)$ for (i)--(ii) and on $\pa_2 K(u_h^*,V_h^*)$ for (iii)--(iv).
\end{lemma}
\begin{proof} The proof is an adaptation of the proof in \cite[Lemma~6.1]{MCF_generalised}, \blueon and it uses similar techniques as the proof of Lemma~\ref{lemma:solution-dependent stiffness matrix diff} below. \blueoff \qed
\end{proof}

We will also need the stiffness matrix analogue of Lemma~\ref{lemma:solution-dependent mass matrix diff}.
\begin{lemma}
	\label{lemma:solution-dependent stiffness matrix diff}
	Let $\bfu \in \R^N$ and $\us \in \R^N$ be such that the corresponding finite element functions $u_h$ and $u_h^*$ have bounded $L^\infty$ norms. 
	Let \eqref{Epsass} hold. Then the following bounds hold:
		\begin{align}
		\tag{i}
		\bfw^T \big( \bfA(\bfx,\us) - \bfA(\bfy,\us) \big) \bfz 
		&\leq C \, \|\bfe\|_{\bfA(\bfx)} \| \bfw \|_{\bfA(\bfx)} \, \|\nb_{\Ga_h[\bfy]} z_h \|_{L^\infty(\Ga_h[\bfy])} ,
		\end{align}
		and
		\begin{align}
		\tag{ii}
		\bfw^T \big( \bfA(\bfx,\bfu) - \bfA(\bfx,\us) \big) \bfz 
		&\leq C \, \| \bfu - \us \|_{\bfM(\bfx)} \, \| \bfw \|_{\bfA(\bfx)} \, \| \nb_{\Ga_h[\bfx]} z_h \|_{L^\infty(\Ga_h[\bfx])} .
		\end{align}
	The constant $C > 0$ is independent of $h$ and $t$.
\end{lemma}

\begin{proof}
The proof is similar to the proof of \blueon \cite[Lemma~6.1]{MCF_generalised}. \blueoff 

(i) Using the fundamental theorem of calculus and the Leibniz formula \cite[Lemma~2.2]{DziukElliott_ESFEM}, and recalling \eqref{matrix difference A}, we obtain
\begin{equation}
\label{eq:c-stiff diff - first identity}
	\begin{aligned}
		&\ \bfw^T \big( \bfA(\bfx,\us) - \bfA(\bfy,\us) \big) \bfz \\
		= &\ \int_{\Ga_h^1} \D(u_h^{*,1}) \, \nb_{\Ga_h^1} w_h^1 \cdot \nb_{\Ga_h^1} z_h^1 
		- \int_{\Ga_h^0} \D(u_h^{*,0}) \, \nb_{\Ga_h^0} w_h^0 \cdot \nb_{\Ga_h^0} z_h^0 \\
		= &\  \int_0^1 \frac{\d}{\d \theta} \int_{\Ga_h^\theta} \D(u_h^{*,\theta}) \, \nb_{\Ga_h^\theta} w_h^\theta \cdot \nb_{\Ga_h^\theta} z_h^\theta \d \theta \\
		= &\
		\int_0^1 \int_{\Ga_h^\theta} \D(u_h^{*,\theta}) \, \mat_{\Ga_h^\theta} (\nb_{\Ga_h^\theta} w_h^\theta) \cdot \nb_{\Ga_h^\theta} z_h^\theta \d \theta \\
		&\ + \int_0^1 \int_{\Ga_h^\theta} \D(u_h^{*,\theta}) \, \nb_{\Ga_h^\theta} w_h^\theta \cdot \mat_{\Ga_h^\theta} (\nb_{\Ga_h^\theta} z_h^\theta) \d \theta \\
		&\ + \int_0^1 \int_{\Ga_h^\theta} \D(u_h^{*,\theta}) \, \nb_{\Ga_h^\theta} w_h^\theta \cdot \big( D_{\Ga_h^\theta} e_h^\theta \big) \nb_{\Ga_h^\theta} z_h^\theta \d \theta ,
	\end{aligned}
\end{equation}
where we used that the due to the $\theta$-independence of $u_h^{*,\theta}$ we have $\mat_{\Ga_h^\theta} u_h^{*,\theta} = 0$, and hence $\mat_{\Ga_h^\theta} ( \D(u_h^{*,\theta})) = 0$. 

For the first two terms we use the interchange formula \cite[Lemma~2.6]{DziukKronerMuller}, for any $w_h \colon \Ga_h \to \R$:
\begin{equation}
\label{eq:mat-grad formula}
	\mat_{\Ga_h^\theta} ( \nb_{\Ga_h^\theta} w_h^\theta ) 
	= \nb_{\Ga_h^\theta} \mat_{\Ga_h^\theta} w_h^\theta	- \bigl(\nb_{\Ga_h^\theta} e_h^\theta - \nu_{\Ga_h^\theta} (\nu_{\Ga_h^\theta})^T (\nb_{\Ga_h^\theta} e_h^\theta )^T \bigr) \nb_{\Ga_h^\theta} w_h^\theta , 
\end{equation}
where $e_h^\theta$ is the velocity and $\nu_{\Ga_h^\theta}$ is the normal vector of the surface $\Ga_h^\theta$, the material derivative associated to $e_h^\theta$ is denoted by $\mat_{\Ga_h^\theta}$.

Using \eqref{eq:mat-grad formula} and recalling $\mat_{\Ga_h^\theta} w_h^\theta = \mat_{\Ga_h^\theta} z_h^\theta = 0$, for \eqref{eq:c-stiff diff - first identity} we obtain the estimate
\begin{equation*}
	\begin{aligned}
		&\ \bfw^T \big( \bfA(\bfx,\us) - \bfA(\bfy,\us) \big) \bfz \\
		\leq &\ c \int_0^1  \|\D(u_h^{*,\theta})\|_{L^\infty(\Ga_h^\theta)} \|\nb_{\Ga_h^\theta} w_h^\theta\|_{L^2(\Ga_h^\theta)} \|\nb_{\Ga_h^\theta} e_h^\theta\|_{L^2(\Ga_h^\theta)} \| \nb_{\Ga_h^\theta} z_h^\theta\|_{L^\infty(\Ga_h^\theta)} \d \theta \\
		\leq &\ C \, 
		\|\bfw\|_{\bfA(\bfx)} \|\bfe\|_{\bfA(\bfx)} \| \nb_{\Ga_h[\bfy]} z_h\|_{L^\infty(\Ga_h[\bfy])} ,
	\end{aligned}
\end{equation*}
where for the last estimate we used the norm equivalences \eqref{eq:lemma - theta-independence}, and the assumed $L^\infty$ bound on $u_h^*$.

(ii) The second estimate is proved using a similar idea, now working only on the surface $\Ga_h[\bfx]$:
\begin{align*}
	&\ \bfw^T \big( \bfA(\bfx,\bfu) - \bfA(\bfx,\us) \big) \bfz \\
	= &\ \int_{\Ga_h[\bfx]} \!\!\!\!\!\! \big( \D(u_h^*) - \D(u_h) \big) \nb_{\Ga_h[\bfx]} w_h \cdot \nb_{\Ga_h[\bfx]} z_h \\
	\leq &\ C \, \|u_h^* - u_h\|_{L^2(\Ga_h[\bfx])} \|\nb_{\Ga_h[\bfx]} w_h\|_{L^2(\Ga_h[\bfx])} \|\nb_{\Ga_h[\bfx]} z_h\|_{L^\infty(\Ga_h[\bfx])} ,
\end{align*}
using the $L^\infty$ boundedness of $u_h$ and $u_h^*$ together with the local Lipschitz continuity of $\D$. \qed 
\end{proof}

As a consequence of the boundedness below of the nonlinear functions $  \pa_2 K(\cdot,\cdot)$ and $\D(\cdot)$, we note here that the matrices $\bfM(\bfx,\bfu,\bfV)$ and $\bfA(\bfx,\bfu)$ (for any $\bfu$ and $\bfV$, with corresponding $u_h$ and $V_h$ in $S_h[\bfx]$)  generate solution-dependent (semi-)norms:
\begin{align*}
	\|\bfz\|_{\bfM(\bfx,\bfu,\bfV)}^{2} = &\ \bfz^T \bfM(\bfx,\bfu,\bfV) \bfz = \int_{\Ga_h[\bfx]} \pa_2 K(u_h,V_h)  \, |z_h|^2 , \\
	\|\bfz\|_{\bfA(\bfx,\bfu)}^{2} = &\ \bfz^T \bfA(\bfx,\bfu) \bfz = \int_{\Ga_h[\bfx]} \D(u_h) \, \blueon |\nbgh z_h|^2 \blueoff ,
\end{align*}
equivalent (independently of $h$ and $t$) to $\|\cdot\|_{\bfM(\bfx)}$ and $\|\cdot\|_{\bfA(\bfx)}$, respectively. The following $h$-independent  equivalence between the $\bfA(\bfx)$ and $\bfA(\bfx,\bfu)$ norms follows by Assumption \ref{eq:assumptions D} on $\mathcal D(\cdot)$: for any $\bfz \in \R^N$
\begin{equation}
\label{eq:Axu norm equivalence}
	c_0 \|\bfz\|_{\bfA(\bfx)}^2 \leq \|\bfz\|_{\bfA(\bfx,\bfu)}^2 \leq c_1 \|\bfz\|_{\bfA(\bfx)}^2 .
\end{equation} 
The equivalence for the $\bfM(\bfx)$ and $\bfM(\bfx,\bfu,\bfV)$ norms will be proved later on.

\subsection{Variable coefficient matrices for time evolving surfaces}

Similarly to \eqref{matrix difference bounds-t}, we will need a result comparing the matrices $\bfA(\bfx,\us)$ at different times. Particularly important will be the $\bfA(\bfx,\us)$ variant of \eqref{matrix derivative bounds}.

\begin{lemma}
\label{lemma:solution-dependent stiffness matrix time difference}
	Let $\us \colon [0,T] \to \R^N$ be such that for all $t$ the corresponding finite element function $u_h^*$ satisfies $\|\mat_h u_h^* \|_{L^\infty(\Ga_h[\bfx])} \leq R$ and $\|\D'(u_h^*)\|_{L^\infty(\Ga_h[\bfx])} \leq R$ for $0 \leq t \leq T$.
	Then the following bounds hold, for $0\le s, t \le T$ with $K_v|t-s| \le \tfrac14$,
	\begin{equation}
	\label{u-stiffness matrix difference bounds-t}
	\bfw^T \big( \bfA(\bfx(t),\us(t))  - \bfA(\bfx(s),\us(s)) \big) \bfz \leq C \, |t-s| \, \|\bfw\|_{\bfA(\bfx(t))}\|\bfz\|_{\bfA(\bfx(t))},
	\end{equation}
	\begin{equation}
	\label{u-stiffness matrix derivative}
	\bfw^T \diff \big( \bfA(\bfx(t),\us(t)) \big) \bfz \leq C \, \|\bfw\|_{\bfA(\bfx(t))}\|\bfz\|_{\bfA(\bfx(t))} .
	\end{equation}
	where the constant $C > 0$ is independent of $h$ and $t$, but depends on $R^2$.
	
\end{lemma}
\begin{proof} We follow the ideas of the proofs of \cite[Lemma~4.1]{DziukLubichMansour_rksurf} and \cite[Lemma~4.1]{KLLP2017}. Similarly to the proof of Lemma~\ref{results collection} (see~\cite{MCF}), by the fundamental theorem of calculus and the Leibniz formula, and recalling that $\mat_h w_h = \mat_h z_h = 0$, we obtain
	\begin{equation}
	\label{eq:c-stiffnes matrix difference}
	\begin{aligned}
	&\ \bfw^T \big( \bfA(\bfx(t),\us(t))  - \bfA(\bfx(s),\us(s)) \big) \bfz \\
	= &\ \int_s^t \frac{\d}{\d r} \int_{\Ga_h[\bfx(r)]} \!\!\!\!\!\! \D(u_h^*) \nb_{\Ga_h[\bfx(r)]} w_h \cdot \nb_{\Ga_h[\bfx(r)]} z_h \ \d r \\
	= &\ \int_s^t \int_{\Ga_h[\bfx(r)]} \! \frac{\d}{\d r} \Big( \D(u_h^*) \Big) \nb_{\Ga_h[\bfx(r)]} w_h \cdot \nb_{\Ga_h[\bfx(r)]} z_h \ \d r \\
	&\ + \int_s^t \int_{\Ga_h[\bfx(r)]} \!\!\!\!\!\! \D(u_h^*)  \nb_{\Ga_h[\bfx(r)]} w_h \cdot \big( D_{\Ga_h[\bfx(r)]} v_h \big) \nb_{\Ga_h[\bfx(r)]} z_h \ \d r ,
	\end{aligned}
	\end{equation}
	where the first order differential operator $D_{\Ga_h[\bfx]}$ is given after \eqref{matrix difference A}.
	
	Similarly as for \eqref{matrix difference bounds-t}, using the bound \eqref{vh-bound} we obtain that
	\begin{equation*}
	\| D_{\Ga_h[\bfx(r)]} v_h(\cdot,r) \|_{L^\infty(\Ga_h[\bfx(r)])} \leq c \| \nabla_{\Ga_h[\bfx(r)]}v_h(\cdot,r) \|_{L^{\infty}(\Ga_h[\bfx(r)])} \leq c \, K_v .
	\end{equation*}
	On the other hand, using the uniform upper bound on the growth of the diffusion coefficient $\D$ (Assumption~\ref{eq:assumptions D}) and the assumed $L^\infty$ bound $\|\mat_h u_h^* \|_{L^\infty(\Ga_h[\bfx])} \leq R$, we have the bound
	\begin{align*}
	\Big\| \frac{\d}{\d r} \Big( \D(u_h^*(\cdot,r)) \Big) \Big\|_{L^\infty(\Ga_h[\bfx(r)])} 
	= &\ \| \D'(u_h^*(\cdot,r)) \, \mat_h u_h^*(\cdot,r) \|_{L^\infty(\Ga_h[\bfx(r)])} 
	\leq R^2 .
	\end{align*} 
	
	By applying the H\"older inequality to \eqref{eq:c-stiffnes matrix difference}, and combining it with the above estimates, we obtain 
	\begin{equation*}
	\begin{aligned}
	&\ \bfw^T \big( \bfA(\bfx(t),\us(t))  - \bfA(\bfx(s),\us(s)) \big) \bfz \\
	\leq &\ c \int_s^t  \| \nb_{\Ga_h[\bfx(r)]} w_h \|_{L^2\Ga_h[\bfx(r)]} \| \nb_{\Ga_h[\bfx(r)]} z_h\|_{L^2\Ga_h[\bfx(r)]} \, \d r .
	\end{aligned}
	\end{equation*}
	The proof of \eqref{u-stiffness matrix difference bounds-t} is then finished using the $h$-uniform norm equivalence in time \eqref{norm-equiv-t}. 
	
	\blueon Dividing \blueoff \eqref{u-stiffness matrix difference bounds-t} by $t-s$ and letting $s \to t$ yields \eqref{u-stiffness matrix derivative}. \qed
\end{proof}

\section{\bf Finite element semi-discretisations of the coupled problem}
\label{section:semi-discretisation}
We present two evolving surface  finite element discretisations of {\bf Problems \ref{P1}, \ref{P2}}. 
In the following  we use the notation $A_h = \half (\nb_{\Ga_h[\bfx]} \nu_h + (\nb_{\Ga_h[\bfx]} \nu_h)^T)$ for the symmetric part of $\nb_{\Ga_h[\bfx]} \nu_h$,  $|\cdot|$ for the Frobenius norm and 
the abbreviations $\pa_j K_h := \pa_j K(u_h,V_h)$, $\pa_j F_h := \pa_j F(u_h,V_h)$ ($j = 1,2$).  We set  $\widetilde I_h = \widetilde I_h[\bfx] \colon C(\Ga_h[\bfx]) \to S_h(\Ga_h[\bfx])$ to be  the finite element interpolation operator on the discrete surface $\Ga_h[\bfx]$. 

%
%

\begin{problem}  \label{SemiDiscreteP1}
Find the finite element functions  $X_h(\cdot,t)\in S_h[\bfx^0]^3$,  $\nu_h(\cdot,t)\in S_h[\bfx(t)]^3$, $V_h(\cdot,t) \in S_h[\bfx(t)]$  and  $u_h(\cdot,t)\in S_h[\bfx(t)]$ such that for $t>0$ 
and for all $\vphi_h^\nu\ct \in S_h[\bfx(t)]^3$, $\vphi_h^V\ct \in S_h[\bfx(t)]$, and $\vphi_h^u\ct \in S_h[\bfx(t)]$ with discrete material derivative $\mat_h \vphi_h^u\ct \in S_h[\bfx(t)]$:
\begin{subequations}
\begin{align}\label{eq:semi-discrete problem - P1}
	&\ v_h = \widetilde I_h (V_h \n_h) , \\ 
	&\ \int_{\Ga_h[\bfx]} \!\!\!\! \pa_2 K_h \, \mat_h \nu_h \cdot \vphi_h^\nu + \int_{\Ga_h[\bfx]} \!\!\!\! \nb_{\Ga_h[\bfx]} \nu_h \cdot \nb_{\Ga_h[\bfx]} \vphi_h^\nu 
	\nonumber  \\ &\ \qquad\qquad\qquad 
	= \int_{\Ga_h[\bfx]} \!\!\!\! |A_h|^2 \nu_h \cdot \vphi_h^\nu + \int_{\Ga_h[\bfx]} \!\!\!\! \pa_1 K_h \nb_{\Ga_h[\bfx]} u_h \cdot \vphi_h^\nu  , \\ 
	&\ \int_{\Ga_h[\bfx]} \!\!\!\! \pa_2 K_h \, \mat_h V_h \, \vphi_h^V + \int_{\Ga_h[\bfx]} \!\!\!\! \nb_{\Ga_h[\bfx]} V_h \cdot \nb_{\Ga_h[\bfx]} \vphi_h^V 
	\nonumber  \\ &\ \qquad\qquad\qquad 
	= \int_{\Ga_h[\bfx]} \!\!\!\! |A_h|^2 V_h \vphi_h^V - \int_{\Ga_h[\bfx]} \!\!\!\! \pa_1 K_h \, \mat_h u_h \, \vphi_h^V , \\
	\label{eq:semi-discrete - P1 - diffusion eq}
	&\ \diff \bigg( \int_{\Ga_h[\bfx]} \!\!\!\! u_h \, \vphi_h^u \bigg) + \int_{\Ga_h[\bfx]} \!\!\!\! \D(u_h) \, \nb_{\Ga_h[\bfx]} u_h \cdot \nb_{\Ga_h[\bfx]} \vphi_h^u = \int_{\Ga_h[\bfx]} \!\!\!\! u_h \, \mat_h \vphi_h^u , \\
	&\ \pa_t X_h(p_h,t) =  v_h . 
\end{align}
\end{subequations}

The initial values for the finite element functions solving this  system  are chosen to be the Lagrange interpolations on the initial surface  of the corresponding data for the PDE, $X^0, \nu^0,V^0 $ and $u^0$.
The initial  data is assumed consistent to be with the equation $V^0=-F(u^0,H^0)$.
\end{problem}

\begin{problem} \label{SemiDiscreteP2}

Find the finite element functions  $X_h(\cdot,t)\in S_h[\bfx^0]^3$,  $\nu_h(\cdot,t)\in S_h[\bfx(t)]^3$, $H_h(\cdot,t) \in S_h[\bfx(t)]$  and  $u_h(\cdot,t)\in S_h[\bfx(t)]$ such that for $t>0$ 
and for all $\vphi_h^\nu\ct \in S_h[\bfx(t)]^3$, $\vphi_h^H\ct \in S_h[\bfx(t)]$, and $\vphi_h^u\ct \in S_h[\bfx(t)]$ with discrete material derivative $\mat_h \vphi_h^u\ct \in S_h[\bfx(t)]$:
\begin{subequations}
\label{eq:semi-discrete problem - P2}
	\begin{align}
	&\ v_h = \widetilde I_h (V_h \n_h) , \\ 
	&\ \int_{\Ga_h[\bfx]} \! \frac{1}{\pa_2 F_h} \, \mat_h \nu_h \cdot \vphi_h^\nu + \int_{\Ga_h[\bfx]} \!\!\!\! \nb_{\Ga_h[\bfx]} \nu_h \cdot \nb_{\Ga_h[\bfx]} \vphi_h^\nu 
	\nonumber  \\ &\ \qquad\qquad\qquad 
	= \int_{\Ga_h[\bfx]} \!\!\!\! |A_h|^2 \nu_h \cdot \vphi_h^\nu + \int_{\Ga_h[\bfx]} \! \frac{\pa_1 F_h}{\pa_2 F_h} \nb_{\Ga_h[\bfx]} u_h \cdot \vphi_h^\nu  , \\ 
	&\  \int_{\Ga_h[\bfx]} \!\!\!\! \mat H_h \, \vphi_h^H
	\blueon  - \int_{\Ga_h[\bfx]} \!\!\!\! \nb_{\Ga_h[\bfx]} V_h \cdot \nb_{\Ga_h[\bfx]} \vphi_h^H = - \int_{\Ga_h[\bfx]} \!\!\!\! |A_h|^2 V_h \vphi_h^H , \blueoff \\
	&\ \label{eq:semi-discrete - P2 - diffusion eqo} \diff \bigg( \int_{\Ga_h[\bfx]} \!\!\!\! u_h \, \vphi_h^u \bigg) + \int_{\Ga_h[\bfx]} \!\!\!\! \D(u_h) \, \nb_{\Ga_h[\bfx]} u_h \cdot \nb_{\Ga_h[\bfx]} \vphi_h^u = \int_{\Ga_h[\bfx]} \!\!\!\! u_h \, \mat_h \vphi_h^u , \\
	&\ \int_{\Ga_h[\bfx]} \!\!\!\! V_h \vphi_h^V + \int_{\Ga_h[\bfx]} \!\!\!\! F(u_h,H_h) \, \vphi_h^V = 0 , \\
	&\ \pa_t X_h(p_h,t) =  v_h . 
\end{align}
\end{subequations}
The initial values for the finite element functions solving this system are chosen to be the Lagrange interpolations on the initial surface of the corresponding data for the PDE, $X^0, \nu^0,H^0 $ and $u^0$.

\end{problem}

\begin{remark}
\label{remark:semi-discrete mass conservation}
-- We note that, in view of the discrete transport property \eqref{eq:transport property of basis functions}, the last term in each of \eqref{eq:semi-discrete - P1 - diffusion eq} and \eqref{eq:semi-discrete - P2 - diffusion eqo} vanishes for all basis functions $\vphi_h^u = \phi_j[\bfx]$. 

-- Also by testing \eqref{eq:semi-discrete - P1 - diffusion eq} and \eqref{eq:semi-discrete - P2 - diffusion eqo} by $\vphi_h^u \equiv 1 \in S_h[\bfx]$ we observe that 
both semi-discrete systems preserve the mass conservation property of the continuous flow, cf.~Section~\ref{section:examples and properties}.  
\end{remark}
 \begin{remark}
Note that the approximate normal vector $\n_h$ and the approximate mean curvature $H_h$ are finite element functions $\n_h(\cdot,t) = \sum_{j=1}^\dof \nu_j(t) \, \phi_j[\bfx(t)] \in S_h[\bfx]^3$ and $H_h(\cdot,t) = \sum_{j=1}^\dof H_j(t) \, \phi_j[\bfx(t)] \in S_h[\bfx]^3$, respectively, and are \emph{not} the normal vector and the mean curvature of the discrete surface $\Ga_h[\bfx(t)]$. Similarly $V_h(\cdot,t) = \sum_{j=1}^\dof V_j(t) \, \phi_j[\bfx(t)] \in S_h[\bfx]$ is \emph{not} the normal velocity of $\Ga_h[\bfx(t)]$.
\end{remark}

%
%
\section{\bf Matrix--vector formulations}
\label{section:matrix-vector form}
The finite element nodal values of the unknown semi-discrete functions $v_h(\cdot,t)$, $\nu_h(\cdot,t)$ and $V_h(\cdot,t)$, $u_h(\cdot,t)$, and (if needed) $H_h(\cdot,t)$ are collected, respectively, into column vectors  $\bfv\t = (v_j\t) \in \R^{3N}$, $\bfn\t = (\n_j\t) \in \R^{3N}$, $\bfV\t = (V_j\t) \in \R^N$, $\bfu\t = (u_j\t) \in \R^N$, and $\bfH\t = (H_j\t) \in \R^N$. 
If it is clear from the context, the time dependencies will be often omitted.

\subsection{\bf Matrix vector evolution equations}

Recalling  the notation \blueon $A_h = \half (\nb_{\Ga_h[\bfx]} \nu_h + (\nb_{\Ga_h[\bfx]} \nu_h)^T)$, \blueoff $\pa_j K_h = \pa_j K(u_h,V_h)$, and $\pa_j F_h = \pa_j F(u_h,V_h)$ from the previous section,  
it is convenient to introduce the following non-linear maps:
\begin{align}
	\bff_1(\bfx,\bfn,\bfV,\bfu)|_{j+(\ell-1)N} 
	= &\ \int_{\Ga_h[\bfx]} \!\!\!\! |A_h|^2 \, (\nu_h)_\ell \, \phi_j[\bfx] 
	+ \int_{\Ga_h[\bfx]} \!\!\!\! \pa_1 K_h \, (\nb_{\Ga_h[\bfx]} u_h)_\ell \, \phi_j[\bfx] , \\
	\label{eq:rhs functions f_2}
	\bff_2(\bfx,\bfn,\bfV,\bfu;\dot\bfu)|_{j} 
	= &\ \int_{\Ga_h[\bfx]} \!\!\!\! |A_h|^2 V_h \, \phi_j[\bfx] 
	- \int_{\Ga_h[\bfx]} \!\!\!\! \pa_1 K_h \, \mat_h u_h \, \phi_j[\bfx] , 
\end{align}
for $j = 1, \dotsc, N$ and $\ell=1,2,3$. 
\blueon Since $\bff_2$ is linear in $\dot\bfu$, we highlight  this by the use of a semi-colon in the list of arguments. \blueoff

For convenience we introduce the following notation.  For $d \in \N$ (with the identity matrices $I_d \in \R^{d \times d}$), we define by the Kronecker products: 
$$
	\bfM^{[d]}(\bfx)= I_d \otimes \bfM(\bfx), \qquad
	\bfA^{[d]}(\bfx)= I_d \otimes \bfA(\bfx) .
$$
When no confusion can arise, we will write $\bfM(\bfx)$ for $\bfM^{[d]}(\bfx)$, $\bfM(\bfx,\bfu)$ for $\bfM^{[d]}(\bfx,\bfu)$, and $\bfA(\bfx)$ for $\bfA^{[d]}(\bfx)$. 
We will use both concepts for other matrices as well. 
Moreover, we use  $\bullet$ to denote the coordinate-wise multiplication for vectors $\bfy \in \mathbb R^{\dof}, \bfz \in \mathbb R^{3\dof}$
$$(\bfy \bullet \bfz)|_j = y_j z_j~\in \mathbb R^3  ~~\mbox{ for}~~j = 1, \dotsc, N . $$

\begin{problem}[Matrix--vector formulation of Problem~\ref{SemiDiscreteP1}]
\label{MVP1}
Using these definitions, and the transport property \eqref{eq:transport property of basis functions} for \eqref{eq:semi-discrete - P1 - diffusion eq}, the semi-discrete Problem~\ref{SemiDiscreteP1} can be written in the matrix--vector form:
\begin{subequations}
	\label{eq:matrix-vector form - P1 - pre}
	\begin{align}
	\label{eq:matrix-vector form - velocity law}
	\bfv =&\ \bfV \bullet \bfn , \\
	\bfM^{[3]}(\bfx,\bfu,\bfV) \dot{\bfn} + \bfA^{[3]}(\bfx) \bfn =&\ \bff_1(\bfx,\bfn,\bfV,\bfu) , \\
	\bfM(\bfx,\bfu,\bfV) \dot{\bfV} + \bfA(\bfx) \bfV =&\ \bff_2(\bfx,\bfn,\bfV,\bfu;\dot\bfu) , \\
	\diff \Big(\bfM(\bfx) \bfu \Big) + \bfA(\bfx,\bfu) \bfu =&\ \bfzero , \\
	\dot \bfx =&\ \bfv .
	\end{align} 
\end{subequations}
\end{problem}

\begin{problem}[Matrix--vector formulation of Problem~\ref{SemiDiscreteP2}]
\label{MVP2}
The semi-discrete Problem \ref{SemiDiscreteP2} can be written in the matrix--vector form (with non-linear matrix \blueon $\bfF$ and vectors $\bff_3, \bff_4$, defined according to \eqref{eq:semi-discrete problem - P2}): \blueoff
\begin{subequations}
	\label{eq:matrix-vector form - P2}
	\begin{align}
	\bfv =&\ \bfV \bullet \bfn , \\
	\bfM^{[3]}(\bfx,\bfu,\bfH) \dot{\bfn} + \bfA^{[3]}(\bfx) \bfn =&\ \bff_3(\bfx,\bfn,\bfH,\bfu) , \\
	\bfM(\bfx) \dot{\bfH} \blueon - \bfA(\bfx) \bfV =&\  \bff_4(\bfx,\bfn,\bfV) , \blueoff \\
	\bfM(\bfx) \bfV + \bfF(\bfx,\bfu,\bfH) =&\ \bfzero , \\
	\diff \Big(\bfM(\bfx) \bfu \Big) + \bfA(\bfx,\bfu) \bfu =&\ \bfzero , \\
	\dot \bfx =&\ \bfv .
	\end{align} 
\end{subequations}
\end{problem}
\begin{remark}
Upon noticing that the equations for $\bfn$ and $\bfV$ in Problem \ref{MVP1} are almost identical, we collect
\begin{equation*}
	\bfw := \begin{pmatrix} \bfn \\ \bfV \end{pmatrix} \in \R^{4N} .
\end{equation*}
Motivated by this abbreviation, we set $\bfM(\bfx,\bfu,\bfw) := \bfM(\bfx,\bfu,\bfV)$ (using these two notations interchangeably, if no confusion can arise), we then rewrite the system into
\begin{problem}[Equivalent matrix--vector formulation of Problem~\ref{SemiDiscreteP1}]

\begin{subequations}
	\label{eq:matrix-vector form - P1}
	\begin{align}
	\bfv =&\ \bfV \bullet \bfn , \\
	\bfM^{[4]}(\bfx,\bfu,\bfw) \dot{\bfw} + \bfA^{[4]}(\bfx) \bfw =&\ \bff(\bfx,\bfw,\bfu;\dot\bfu) , \\
	\diff \Big(\bfM(\bfx) \bfu \Big) + \bfA(\bfx,\bfu) \bfu =&\ \bfzero , \\
	\dot \bfx =&\ \bfv .
	\end{align} 
\end{subequations}
We remind that $\bff = (\bff_1,\bff_2)^T$ is linear in $\dot\bfu$.
\end{problem}
\end{remark}
\begin{remark}
We compare the above matrix--vector formulation \eqref{eq:matrix-vector form - P1} to the same formulas for \emph{forced} mean curvature flow \cite{MCF_soldriven}, with velocity law $v = -H\nu + g(u)\nu$, here $\bfw$ collects $\bfw = (\bfn,\bfH)^T$:
\begin{subequations}
\label{eq:MCF_forced matrix-vector form}
\begin{align}
	\bfv = &\ - \bfH \bullet \bfn , \\
	\bfM^{[4]}(\bfx) \dot{\bfw} + \bfA^{[4]}(\bfx) \bfw = &\ \bff(\bfx,\bfw,\bfu;\dot{\bfu}), \\
	\bfM(\bfx) \dot{\bfu} + \bfA(\bfx) \bfu = &\ \bfg(\bfx,\bfw,\bfu) , \\
	\dot \bfx = &\ \bfv ,
\end{align} 
\end{subequations}
and to \emph{generalised} mean curvature flow \cite[equation~(3.4)]{MCF_generalised}, with velocity law $v = -V(H) \nu$, here $\bfw$ collects $\bfw = (\bfn,\bfV)^T$:
\begin{subequations}
\label{eq:MCF_gen matrix-vector form}
\begin{align}
	\bfv =&\ \bfV \bullet \bfn , \\
	\bfM^{[4]}(\bfx,\bfw) \dot{\bfw} + \bfA^{[4]}(\bfx) \bfw &= \bff(\bfx,\bfw), \\
	\dot \bfx &= \bfv .
\end{align} 
\end{subequations}

The coupled system \eqref{eq:matrix-vector form - P1} has a similar structure to those of  \eqref{eq:MCF_forced matrix-vector form} and \eqref{eq:MCF_gen matrix-vector form}. 
Due to these similarities, in the stability proof we will use similar arguments to \cite{MCF_soldriven} and \cite{MCF_generalised} as wells as  those in \cite{MCF}.

\blueon Compared to previous works, the concentration dependency in the mass matrix $\bfM(\bfx,\bfu,\bfw)$ and in the stiffness matrix $\bfA(\bfx,\bfu)$ requires extra care in estimating the corresponding terms in the stability analysis. For which the results of Section~\ref{section:variable coefficient matrices} will play a key role. \blueoff 
\end{remark}

\subsection{\bf Defect and error equations}
\label{section:defects and error eqns}
We set $\bfu^*$ to be the nodal vector of the Ritz projection $\Rh^u u$ defined by \eqref{eq:definition quasi-linear Ritz map} on the interpolated surface $\Ga_h[\xs\t]$. 
The vectors   $\bfn^*\in \mathbb R^{3\dof}$ and $\bfV^*\in \mathbb R^\dof$ are the nodal vectors associated with  the Ritz projections $\Rh \nu$ and $\Rh V$ defined by \eqref{eq:definition Ritz map} of the normal and the normal velocity of the surface solving the PDE system. We set
\begin{equation*}
	\bfw^* := \begin{pmatrix} \bfn^* \\ \bfV^* \end{pmatrix} \in \R^{4N} .
\end{equation*}

It is convenient to introduce  the following  equations that define  {\it defect quantities}  $\dv,\dw,\du$ which occur when surface finite element interpolations and Ritz projections  of the exact solution (i.e.~$\xs$, $\vs$ and $\ws$, $\us$) are substituted into the matrix--vector equations defining the numerical approximations \eqref{eq:matrix-vector form - P1}. 
\begin{definition}[Defect equations] The {\it defects} $\dv,\dw,\du$ are defined by the following coupled system:
\begin{subequations}
	\label{eq:defect definition}
	\begin{align}
		\vs =&\ \bfV^* \bullet \bfn^* + \dv , \\
		\bfM(\xs,\us,\ws) \dotws + \bfA(\xs) \ws = &\ \bff(\xs,\ws,\us;\dotus) + \bfM(\xs)\dw, \\
		\diff \Big( \bfM(\xs) \us \Big) + \bfA(\xs,\us) \us = &\ \bfM(\xs)\du, \\
		\dotxs = &\ \vs .
	\end{align}
\end{subequations}
\end{definition}

The following {\it error equations} for the 
nodal values of the errors between the  exact and numerical solutions  are obtained by subtracting \eqref{eq:defect definition} from \eqref{eq:matrix-vector form - P1} where the errors are set to be  
\begin{equation*}
	\ex = \bfx - \xs, \quad \ev = \bfv - \vs, \quad \ew = \bfw - \ws, \quad \mbox{and} \quad \eu = \bfu - \us,
\end{equation*}
with corresponding finite element functions, respectively, $$e_x , \quad e_v, \quad e_w, \quad \mbox{and} \quad e_u.$$

\begin{definition}[Error equations]
The error equations are defined by the following system:
\begin{subequations}
	\label{eq:error equations}
	\begin{align}
	\label{eq:error equations - v}
	\ev = &\ \big(\bfV \bullet \bfn - \bfV^* \bullet \bfn^* \big) - \dv , \\
	\label{eq:error equations - w}
	\bfM(\bfx,\bfu,\bfw) \dotew + \bfA(\bfx) \ew 
	= &\ - \big(\bfM(\bfx,\bfu,\bfw) - \bfM(\bfx,\us,\ws)\big) \dotws \nonumber \\ 
	&\ - \big(\bfM(\bfx,\us,\ws) - \bfM(\xs,\us,\ws)\big) \dotws \nonumber \\
	&\ - \big(\bfA(\bfx)-\bfA(\xs)\big) \ws \nonumber \\
	&\ + \big(\bff(\bfx,\bfw,\bfu;\dot\bfu) - \bff(\xs,\ws,\us;\dotus) \big) \nonumber \\ 
	&\ - \bfM(\xs) \dw , \\
	\label{eq:error equations - u}
	\diff \Big( \bfM(\bfx) \eu \Big) +  \bfA(\bfx,\us) \eu 
	= &\ - \diff \Big( \big(\bfM(\bfx) - \bfM(\xs) \big) \us \Big) \nonumber \\
	&\ - \big(\bfA(\bfx,\bfu) - \bfA(\bfx,\us) \big) \eu \nonumber \\
	&\ - \big(\bfA(\bfx,\bfu) - \bfA(\bfx,\us)\big) \us \nonumber \\ 
	&\ - \big(\bfA(\bfx,\us) - \bfA(\xs,\us)\big) \us \nonumber \\ 
	&\ - \bfM(\xs) \du , \\
	\label{eq:error equations - x}
	\dotex = &\ \ev .
	\end{align}
\end{subequations}
\end{definition}
Note that by definition the initial data $\ex(0) = \bfzero$ and $\ev(0) = \bfzero$ whereas $\eu(0) \neq \bfzero$ and $\ew(0) \neq \bfzero$ in general.

\section{\bf Consistency, stability, and convergence}
\label{section:stability analysis}
In this section we prove the main results of this paper. We begin in Section~\ref{boundednessRitz} by noting the  uniform boundedness of some coefficients as a consequence of the approximation properties of the Ritz projections. In Section~\ref{section:consistency} we address the consistency of the finite element approximation by bounding the $L^2$ norms of the defects. 
\subsection{\bf Uniform bounds}
\subsubsection{Boundedness of the Ritz projections}\label{boundednessRitz}
We start by proving $h$- and $t$-uniform $W^{1,\infty}(\Ga_h[\xs\t])$ norm bounds for the finite element projections of the exact solutions (see Section~\ref{section:defects and error eqns}).
\begin{lemma}
	The finite element interpolations $x_h^*$ and  $v_h^*$ and the Ritz maps $w_h^*$, and $u_h^*$ of the exact solutions satisfy
	\begin{equation}
	\label{eq:W^{1,infty} bounds for exact solutions}
		\begin{alignedat}{3}
			&\ \|x_h^*\|_{W^{1,\infty}(\Ga_h[\xs\t])} + \|v_h^*\|_{W^{1,\infty}(\Ga_h[\xs\t])} && \\
			&\ \qquad + \|w_h^*\|_{W^{1,\infty}(\Ga_h[\xs\t])} + \|u_h^*\|_{W^{1,\infty}(\Ga_h[\xs\t])} \leq C & \for & 0 \leq t \leq T,\\
		\end{alignedat}
	\end{equation}
	uniformly in $h$.
\end{lemma}

\begin{proof}
	The $W^{1,\infty}$ bounds for the interpolations, $\Ih X = x_h^*$ and $\Ih v = v_h^*$, follow from the error estimates in \cite[Section~2.5]{Demlow2009}.
	
	On the other hand, the $W^{1,\infty}$ bounds on the Ritz maps ($\Rh w = w_h^*$ and $\Rh^u u = u_h^*$) are obtain, using an inverse estimate \cite[Theorem~4.5.11]{BrennerScott}, above interpolation error estimates of \cite{Demlow2009}, and the Ritz map error bounds \cite{highorderESFEM} and \cite{KPower_quasilinear}, by 
	\begin{equation}
	\label{eq:Ritz map W^1,infty bound}
		\begin{aligned}
			\|\Rh u\|_{W^{1,\infty}(\Ga_h^*)} \! 
			\leq &\ \|\Rh u - \Ih u\|_{W^{1,\infty}(\Ga_h^*)} + c \|I_h u\|_{W^{1,\infty}(\Ga)} \\
			\leq &\ c h^{-d/2} \|\Rh u - \Ih u\|_{H^1(\Ga_h^*))} + \|I_h u\|_{W^{1,\infty}(\Ga)} \\
			\leq &\ c h^{-d/2} \big( \|R_h u - u\|_{H^1(\Ga)} + \|u - I_h u\|_{H^1(\Ga)} \big) \\
			&\ + \|I_h u - u\|_{W^{1,\infty}(\Ga)} \! + \|u\|_{W^{1,\infty}(\Ga)} \\
			\leq &\ c h^{k-d/2} \|u\|_{H^{k+1}(\Ga)} 
			+ (c h + 1) \|u\|_{W^{2,\infty}(\Ga)} , 
		\end{aligned}
	\end{equation}
	with $k - d/2 \geq 0$, in dimension $d = 3$ here. Where for the last term we used the (sub-optimal) interpolation error estimate of \cite[Proposition~2.7]{Demlow2009} (with $p=\infty$).
	\qed 
\end{proof}

\subsubsection{A priori boundedness of numerical solution}
\label{subsec:A priori boundedness}

We note here that, by Assumption~\ref{eq:assumptions on pa_2 K}, along the exact solutions $u$, $V$ in the bounded time interval $[0,T]$ the factor $\pa_2 K(u,V)$ is uniformly bounded from  above and below by constants $K_1 \geq K_0 > 0$.

For the estimates of the non-linear terms we establish some $W^{1,\infty}$ norm bounds.
\begin{lemma}
\label{lemma:assumed bounds}
Let $\kappa > 1$.  There exists a maximal $T^* \in (0,T]$ such that the following inequalities hold: 
\begin{equation}
\label{eq:assumed bounds}
	\begin{aligned}
		\|e_x(\cdot,t)\|_{W^{1,\infty}(\Ga_h[\xs\t])} \leq &\ h^{(\kappa - 1)/2} , \\
		\|e_v(\cdot,t)\|_{W^{1,\infty}(\Ga_h[\xs\t])} \leq &\ h^{(\kappa - 1)/2} , \\ 
		\|e_w(\cdot,t)\|_{W^{1,\infty}(\Ga_h[\xs\t])} \leq &\ h^{(\kappa - 1)/2} , \\ 
		\|e_u(\cdot,t)\|_{W^{1,\infty}(\Ga_h[\xs\t])} \leq &\ h^{(\kappa - 1)/2} ,
	\end{aligned}
	\qquad \text{for} \quad t \in [0,T^*] .
\end{equation}
Then for $h$ sufficiently small and for $0 \leq t \leq T^*$, 
\begin{equation}
\label{eq:W^{1,infty} bounds for numerical solutions}
	x_h,v_h,w_h,u_h~~~\mbox{are uniformly bounded in}~~ W^{1,\infty}(\Ga_h[\xs\t]) .
\end{equation}

Furthermore, the functions $\pa_2 K_h^*=\pa_2 K(u_h^*,V_h^*)$ and $\pa_2 K_h = \pa_2 K(u_h,V_h)$ satisfy the following bounds
\begin{align}
	\label{eq:uniform bounds for pa_2 Kh*}
	0 \, < \, \tfrac23 K_0 \, \leq \, \|\pa_2 K_h^*\|_{L^\infty(\Ga_h[\xs(t)])} \, \leq \, \tfrac32 K_1 \qquad \text{$h$-uniformly for $0 \leq t \leq T$,} \\
	\label{eq:uniform bounds for pa_2 Kh}
	0 \, < \, \tfrac12 K_0 \, \leq \, \|\pa_2 K_h\|_{L^\infty(\Ga_h[\xs(t)])} \, \leq \, 2 K_1 \qquad \text{$h$-uniformly for $0 \leq t \leq T^*$.}
\end{align}

Then these $h$- and time-uniform upper and lower bounds imply that the norms $\|\cdot\|_{\bfM(\bfx)}$ and $\|\cdot\|_{\bfM(\bfx,\bfu,\bfw)}$ are indeed $h$- and $t$-uniformly equivalent, for any $\bfz \in \R^N$:
\begin{equation}
\label{eq:Mxuw norm equivalence}
	\tfrac12 K_0 \, \|\bfz\|_{\bfM(\bfx)}^2 \leq \|\bfz\|_{\bfM(\bfx,\bfu,\bfw)}^2 \leq 2 K_1 \, \|\bfz\|_{\bfM(\bfx)}^2 .
\end{equation} 
\end{lemma}

\begin{proof}
(a) Since we have assumed $\kappa>1$ we obtain that $T^*$ exists and is indeed positive. This is a consequence of the initial errors $e_x(\cdot,0) = 0$, $e_v(\cdot,0) = 0$, and, by an inverse inequality \cite[Theorem~4.5.11]{BrennerScott}, 
\begin{align*}
	\|e_w(\cdot,0)\|_{W^{1,\infty}(\Ga_h[\xs(0)])} \leq c h^{-1} \|e_w(\cdot,0)\|_{H^1(\Ga_h[\xs(0)])} \leq c h^{\kappa - 1} , \\
	\|e_u(\cdot,0)\|_{W^{1,\infty}(\Ga_h[\xs(0)])} \leq c h^{-1} \|e_u(\cdot,0)\|_{H^1(\Ga_h[\xs(0)])} \leq c h^{\kappa - 1} ,
\end{align*}
and for the last inequalities using the error estimates for the Ritz maps $R_h w$ and $R_h^u u$, \cite[Theorem~6.3 and 6.4]{highorderESFEM} and the generalisations of \cite[Theorem~3.1 and 3.2]{KPower_quasilinear}, respectively. 

The uniform bounds on numerical solutions over $[0,T^*]$ \eqref{eq:W^{1,infty} bounds for numerical solutions} is directly seen using \eqref{eq:assumed bounds}, \eqref{eq:W^{1,infty} bounds for exact solutions}, and a triangle inequality.

(b) We now show the $h$- and $t$-uniform upper- and lower-bounds for the coefficient functions $\pa_2 K_h^*=\pa_2 K(u_h^*,V_h^*)$ and $\pa_2 K_h = \pa_2 K(u_h,V_h)$. We use a few ideas from \cite{MCF_generalised}, where similar estimates were shown.

As a first step, it follows from applying inverse inequalities (see, e.g., \cite[Theorem~4.5.11]{BrennerScott}) on the finite element spaces and $H^1$ norm error bounds on the  Ritz maps $R_h$ and $R_h^u$  and $H^1$ and $L^\infty$ error bounds for interpolants  (e.g.~\cite{highorderESFEM}, \cite{EllRan21}, \cite{Demlow2009} and \cite{KPower_quasilinear}) that the following $L^\infty$ norm error bounds hold in dimension $d = 2$ (but stated for a general $d$ for future reference):
\begin{equation}
\label{eq:Linfty error of Ritz maps}
	\begin{aligned}
		\|(V_h^*)^\ell - V\|_{L^\infty(\Ga[X])} 
		\leq &\ c h^{2-d/2} \|V\|_{H^{2}(\Ga[X])} + c h^2 \|V\|_{W^{2,\infty}(\Ga[X])} , \\
		\text{and } \qquad \|(u_h^*)^\ell - u\|_{L^\infty(\Ga[X])} \leq &\ c h^{2-d/2} \|u\|_{H^{2}(\Ga[X])} + c h^2 \|u\|_{W^{2,\infty}(\Ga[X])} .
	\end{aligned}
\end{equation}

By the definition of the lift map we have the equality $\eta_h^*(x,t) = (\eta_h^*)^\ell(x^\ell,t)$ for any function $\eta_h^*\colon\Ga_h[\xs] \to \R$, and then by the triangle and reversed triangle inequalities and using the local Lipschitz continuity of $\pa_2 K$ in both variables and its uniform upper and lower bounds, in combination with \eqref{eq:W^{1,infty} bounds for exact solutions}, we obtain (with the abbreviations $\pa_2 K = \pa_2 K(u,V)$ and $\pa_2 K_h^* = \pa_2 K(u_h^*,V_h^*)$), written here for a $d$ dimensional surface (we will use $d=2$),
\begin{equation*}
	\begin{aligned}
		|\pa_2 K_h^*| \leq &\ |\pa_2 K| + |\pa_2 K - (\pa_2 K_h^*)^\ell| \\
		\leq &\ |\pa_2 K| + \|\pa_2 K - (\pa_2 K_h^*)^\ell\|_{L^\infty(\Ga[X(\cdot,t)])} \\
		\leq &\ K_1 + c \|V - (V_h^*)^\ell\|_{L^\infty(\Ga[X(\cdot,t)])} + c \|u - (u_h^*)^\ell\|_{L^\infty(\Ga[X(\cdot,t)])}  \\
		\leq &\ K_1 + c h^{2-d/2} , 
	\end{aligned}
\end{equation*}
and
\begin{equation*}
	\begin{aligned}
		|\pa_2 K_h^*| \geq &\ |\pa_2 K| - |\pa_2 K - (\pa_2 K_h^*)^\ell| \\
		\geq &\ |\pa_2 K| - \|\pa_2 K - (\pa_2 K_h^*)^\ell\|_{L^\infty(\Ga[X(\cdot,t)])} \\
		\geq &\ K_0 - c \|V - (V_h^*)^\ell\|_{L^\infty(\Ga[X(\cdot,t)])} - c \|u - (u_h^*)^\ell\|_{L^\infty(\Ga[X(\cdot,t)])} \\
		\geq &\ K_0 - c h^{2-d/2} ,
	\end{aligned} 
\end{equation*}
which proves \eqref{eq:uniform bounds for pa_2 Kh*} on $[0,T]$, \emph{independently} of \eqref{eq:assumed bounds}.

A similar argument comparing $\pa_2 K_h$ with $\pa_2 K_h^*$ now, using \eqref{eq:assumed bounds} (which only hold for $0 \leq t \leq T^*$) instead of \eqref{eq:Linfty error of Ritz maps}, together with \eqref{eq:uniform bounds for pa_2 Kh*}, yields the bounds \eqref{eq:uniform bounds for pa_2 Kh}.

In view of \eqref{eq:uniform bounds for pa_2 Kh} the norm equivalence \eqref{eq:Mxuw norm equivalence} is straightforward.
\qed 
\end{proof}




\subsection{\bf Consistency and stability}
\subsubsection{\bf Consistency}
\label{section:consistency}
\blueon For evolving surface finite elements of polynomial degree $k$, \blueoff the defects satisfy the following consistency bounds:
\begin{proposition}
\label{proposition:consistency}
For $t\in [0,T]$, it holds that 
\begin{align}
	\blueon \|d_v(\cdot,t)\|_{H^1(\Ga_h[\xs(t)])} = \|\dv\t\|_{\bfK(\xs\t)} \blueoff \leq &\ c h^k , \\
	\|d_w(\cdot,t)\|_{L^2(\Ga_h[\xs(t)])} = \|\dw\t\|_{\bfM(\xs\t)} \leq &\ c h^k , \\
	\|d_u(\cdot,t)\|_{L^2(\Ga_h[\xs(t)])} = \|\du\t\|_{\bfM(\xs\t)} \leq &\  c h^k .
\end{align}
\end{proposition}
\begin{proof}
The consistency analysis is heavily relying on \cite{KPower_quasilinear,MCF,Willmore,MCF_generalised}, and the high-order error estimates of \cite{highorderESFEM}. 

For the defect in the velocity $v$, using the \blueon $O(h^k)$ \blueoff error estimates of the finite element interpolation operator \blueon in the $H^1$ norm \blueoff \cite{Demlow2009,highorderESFEM}, \blueon similarly as they were employed in \cite[Section~6]{Willmore}, \blueoff we obtain the estimate \blueon $\|d_v(\cdot,t)\|_{H^1(\Ga_h[\xs(t)])}  = \|\dv\t\|_{\bfK(\xs\t)} \leq c h^k$. \blueoff

Regarding the geometric part, \eqref{eq:coupled system - P1 - normal}--\eqref{eq:coupled system - P1 - V}, the additional terms on the right-hand side compared to those in the evolution equations of pure mean
curvature flow in \cite{MCF} do not present additional difficulties in the consistency error analysis, while the non-linear weights on the left-hand side are treated exactly as in \cite{MCF_generalised}. 
Therefore, by combining the techniques and results of \cite[Lemma~8.1]{MCF} and \cite[Lemma~8.1]{MCF_generalised} we directly obtain for $d_w = (d_\nu,d_V)^T$ the consistency estimate $\|d_w(\cdot,t)\|_{L^2(\Ga_h[\xs(t)])} = \|\dw\t\|_{\bfM(\xs\t)} \leq c h^k$.

For the nonlinear diffusion equation on the surface, \eqref{eq:coupled system - P1 - u}, consistency is shown by the techniques of \cite[Theorem~5.1]{KPower_quasilinear}, and yields the bound $\|d_u(\cdot,t)\|_{L^2(\Ga_h[\xs(t)])} = \|\du\t\|_{\bfM(\xs\t)} \leq c h^k$.
\qed 
\end{proof}

\subsubsection{\bf Stability}

The stability proof is based on the following three key estimates for the surface, concentration, and velocity-law, whose clever combination is the key to our stability proof. These results are energy estimates proved by testing the error equations with the time derivatives of the error, cf.~\cite{MCF}. \bbk The first two stability bounds may formally look similar to those in \cite{MCF,MCF_soldriven,MCF_generalised}, yet their proofs are different and are based on the new results of Section~\ref{section:variable coefficient matrices}. \ebk The proofs are postponed to Section~\ref{section:proof - estimates}.

\begin{lemma}
\label{lemma:three estimates}
For the time interval $[0,T^*]$, where Lemma~\ref{lemma:assumed bounds} holds, there exist constants $c_0>0$ and $c>0$ independent of $h$ and $T^*$  such that the following bounds hold:
\begin{enumerate}
\item {\bf Surface estimate:}
\begin{equation}
\label{eq:A - combined estimate}
	\begin{aligned}
		&\ \frac{c_0}{2} \|\dotew\|_{\bfM}^2 + \frac{1}{2} \diff \| \ew \|_{\bfA}^2 \\
		\leq &\ c_1 \|\doteu\|_{\bfM}^2  
		+ c\,  \big( \| \ex \|_{\bfK}^2 + \| \ev \|_{\bfK}^2 + \| \ew \|_{\bfK}^2 + \| \eu \|_{\bfK}^2 \big)
		+ c \, \|\dw\|_{\bfM^*}^2  \\
		&\ - \diff \Big( \ew^T \big( \bfA(\bfx) - \bfA(\xs) \big) \ws \Big) .
	\end{aligned}
\end{equation}

\item{\bf Concentration estimate:}
\begin{equation}
\label{eq:B - combined estimate}
	\begin{aligned}
		\frac14 \|\doteu\|_{\bfM}^2 + \frac{1}{2} \diff \| \eu \|_{\bfA(\bfx,\us)}^2 
		\leq &\ c \, \big( \|\ex\|_{\bfK}^2 + \|\ev\|_{\bfK}^2 + \|\eu\|_{\bfK}^2 \big) 
		 + c \, \|\du\|_{\bfM^*}^2 \\
		&\ - \half \diff \Big( \eu^T \big(\bfA(\bfx,\bfu) - \bfA(\bfx,\us)\big) \eu \Big) \\
		&\ - \diff \Big( \eu^T \big(\bfA(\bfx,\bfu) - \bfA(\bfx,\us)\big) \us \Big) \\
		&\ - \diff \Big( \eu^T \big(\bfA(\bfx,\us) - \bfA(\xs,\us)\big) \us \Big) .
	\end{aligned}
\end{equation}
\item {\bf Velocity-law estimate:}
	\begin{equation}
	\label{eq:velocity error estimate}
	\begin{aligned}
		\|\ev\|_{\bfK} \leq &\ c \|\eu\|_{\bfK} + c \|\dv\|_{\bfK^*} .
	\end{aligned}
	\end{equation}
\end{enumerate}
\end{lemma}

\begin{remark}
\label{remark:notational conventions - 1}
	Regarding notational conventions: By $c$ and $C$ we will denote generic $h$-independent constants, which might take different values on different occurrences. 
	In the norms the matrices $\bfM(\bfx)$ and $\bfM(\xs)$ will be abbreviated to $\bfM$ and $\bfM^*$, i.e.~we will write $\|\cdot\|_{\bfM}$ for $\|\cdot\|_{\bfM(\bfx)}$ and $\|\cdot\|_{\bfM^*}$ for $\|\cdot\|_{\bfM(\xs)}$, and similarly for the other norms.
\end{remark}

The following result provides the key stability estimate.
\begin{proposition}[{Stability}]
\label{proposition:stability - semi-discrete}
	Assume that, for some $\kappa$ with $1 < \kappa \leq k$, the defects are bounded, for $0 \leq t \leq T$, by 
	\begin{equation}
		\label{eq:assumed defect bounds}
		\begin{aligned}
			\|\dv(t)\|_{\bfK(\xs(t))} \leq c h^\kappa , \quad
			\|\dw(t)\|_{\blueon \bfM(\xs(t))\blueoff } \leq c h^\kappa , \quad
			\|\du(t)\|_{\bfM(\xs(t))} \leq c h^\kappa ,
		\end{aligned}
	\end{equation}
	and that also the errors in the initial errors satisfy
	\begin{equation}
		\label{eq:assumed initial error bounds}
		\| \ew(0) \|_{\bfK(\xs(0))} \le c h^\kappa, \andquad \| \eu(0) \|_{\bfK(\xs(0))} \le c h^\kappa .
	\end{equation}
	Then, there exists  $h_0>0$ such that the following stability estimate holds for all $h\leq h_0$ and $0\leq t \leq T$:
	\begin{equation}
		\label{eq:stability estimate}
		\begin{aligned}
			& \| \ex(t) \|_{\bfK(\xs(t))}^2 + \|\ev\t\|_{\bfK(\xs(t))}^2 
			+ \| \ew(t) \|_{\bfK(\xs(t))}^2 + \| \eu(t) \|_{\bfK(\xs(t))}^2 \\[2mm]
			& \leq  C \big( \| \ew(0) \|_{\bfK(\xs(0))}^2 + \| \eu(0) \|_{\bfK(\xs(0))}^2 \big) \\
			&\ \quad + C \max_{0 \leq s \leq t} \|\dv(s)\|_{\bfK(\xs(s))}^2  
			+ C \int_0^t \big( \|\dw(s)\|_{\bfM(\xs(s))}^2 + \|\du(t)\|_{\bfM(\xs(s))}^2 \big) \d s ,
		\end{aligned}
	\end{equation}
	where $C$ is independent of $h$ and $t$, but depends exponentially on the final time $T$.
\end{proposition}

The proof to this result is obtained by an adept combination of the three estimates of Lemma~\ref{lemma:three estimates} and a Gronwall argument, and is postponed to Section~\ref{section:proof - stability}. We also note that by the consistency result, Proposition~\ref{proposition:consistency}, the estimates \eqref{eq:assumed defect bounds} hold with $\kappa = k$.

\subsection{\bf Convergence}
We are now in the position to state the main result of the paper, which provide optimal-order error bounds for the finite element semi-discretisation \eqref{eq:semi-discrete problem - P1}, for finite elements of polynomial degree $k \geq 2$.
\begin{theorem}
\label{theorem:semi-discrete error estimates} 
	Suppose Problem \ref{P1} admits a  sufficiently regular exact solution $(X,v,\nu,V,u)$  on the time interval $t \in [0,T]$ for which the flow map $X(\cdot,t)$ is non-degenerate so that $\Ga[X(\cdot,t)]$ is a regular orientable immersed hypersurface. 
	Then there exists a constant $h_0 > 0$ such that for all mesh sizes $h \leq h_0$ the following error bounds for the lifts of the discrete position, velocity, normal vector, normal velocity and concentration over the exact surface $\Ga[X(\cdot,t)]$ for $0 \leq t \leq T$:
	\begin{align*}
		\|x_h^L(\cdot,t) - \Id_{\Ga[X(\cdot,t)]}\|_{H^1(\Ga[X(\cdot,t)])^3} \leq &\ Ch^k, \\
		\|v_h^L(\cdot,t) - v(\cdot,t)\|_{H^1(\Ga[X(\cdot,t)])^3} \leq &\ C h^k, \\
		\|\nu_h^L(\cdot,t) - \nu(\cdot,t)\|_{H^1(\Ga[X(\cdot,t)])^3} \leq &\ C h^k, \\
		\|V_h^L(\cdot,t) - V(\cdot,t)\|_{H^1(\Ga[X(\cdot,t)])} \leq &\ C h^k, \\
		\|u_h^L(\cdot,t) - u(\cdot,t)\|_{H^1(\Ga[X(\cdot,t)])} \leq &\ C h^k, \\
	\text{and} \qquad \qquad 
		\|X_h^\ell(\cdot,t) - X(\cdot,t)\|_{H^1(\Ga^0)^3} \leq &\ Ch^k ,
	\intertext{furthermore, by \eqref{eq:velocity law-alt} and the smoothness of $K$, for the mean curvature we also have}
		\|H_h^L(\cdot,t) - H(\cdot,t)\|_{H^1(\Ga[X(\cdot,t)])} \leq &\ C h^k .
	\end{align*}
	The constant $C$ is independent of $h$ and $t$, but depends on bounds of higher derivatives of the solution $(X,v,\nu,V,u)$ of the coupled problem, and exponentially on the length $T$ of the time interval.
\end{theorem}

\begin{proof}
The errors are decomposed using the interpolation $I_h$ for $X$ and $v$, using the Ritz map $R_h$ \eqref{eq:definition Ritz map} for $w = (\nu,V)^T$, and using the quasi-linear Ritz map $R_h^u$ \eqref{eq:definition quasi-linear Ritz map} for $u$. For a variable $z \in \{X,v,w,u\}$ and the appropriate ESFEM projection operator $\calP_h$ (with $\calP_h = (\widetilde{\calP}_h)^\ell$), we rewrite the error as
\begin{equation*}
	z_h^L - z = \big( \widehat{z}_h - \widetilde{\calP}_h z \big)^\ell + \big( \calP_h z - z \big) .
\end{equation*}

In each case, the second terms are bounded as $c h^k$ by the error estimates for the above three operators, \cite{Demlow2009,highorderESFEM,KPower_quasilinear,EllRan21}. By the same arguments, the initial values satisfy the $O(h^k)$ bounds of \eqref{eq:assumed initial error bounds}.

The first terms on the right-hand side are bounded using the stability estimate Proposition~\ref{proposition:stability - semi-discrete} together with the defect bounds of Proposition~\ref{proposition:consistency} (and the above error estimates in the initial values), to obtain
\begin{equation*}
	\| \ex(t) \|_{\bfK(\xs(t))} + \|\ev\t\|_{\bfK(\xs(t))} 
	+ \| \ew(t) \|_{\bfK(\xs(t))} + \| \eu(t) \|_{\bfK(\xs(t))} \leq c h^k .
\end{equation*}

Combining the two estimates completes the proof.
\qed 
\end{proof}

\medskip
Sufficient regularity assumptions are the following: with bounds that are uniform in $t\in[0,T]$, we assume $X(\cdot,t) \in  H^{k+1}(\Ga^0)^3$, $v(\cdot,t) \in H^{k+1}(\Ga[X(\cdot,t)])^3$, and for $w=(\nu,V,u)$ we assume $\ w(\cdot,t), \mat w(\cdot,t) \in W^{k+1,\infty}(\Ga[X(\cdot,t)])^5$. 

\begin{remark}
Under these regularity conditions on the solution, for the numerical analysis we only require local Lipschitz continuity of the non-linear functions in Problem~\ref{P1}. These local-Lipschitz conditions are, of course, not sufficient to ensure the existence of even just a weak solution. For regularity results we refer to \cite{diss_Buerger} and \cite{ABG_MCFdiff}. 
Here we restrict our attention to cases where a sufficiently regular solution exists, excluding the formation of singularities but not self-intersections, which we can then approximate with optimal-order under weak conditions on the nonlinearities.
\end{remark}
\begin{remark}
	We note here that the above theorem remains true if we add a non-linear term $f(u,\nbg u)$ (locally Lipschitz in both variables) to the diffusion equation \eqref{diffeqn}. This is due to the fact that we already control the $W^{1,\infty}$ norm of both the exact and numerical solutions (see~\eqref{eq:W^{1,infty} bounds for exact solutions} and \eqref{eq:W^{1,infty} bounds for numerical solutions}). Hence the corresponding terms in the stability analysis can be estimated analogously to the non-linear terms in the geometric evolution equations, see Section~\ref{subsection:energy est for geometric eqns}.
\end{remark}

\begin{remark}
\label{remark:convergence result}
	The remarks made after the convergence result in \cite[Theorem~4.1]{MCF} apply also here, which we briefly recap here. 
	
	-- A key issue in the proof is to ensure that the $W^{1,\infty}$ norm of the position error of the surfaces remains small. The $H^1$ error bound and an inverse estimate yield an $O(h^{k-1})$ error bound in the $W^{1,\infty}$ norm. For two-dimensional surfaces, this is small only for $k \geq 2$, which is why we impose the condition $k \geq 2$ in the above result. For higher-dimensional surfaces a larger polynomial degree is required.
	
	-- Provided the flow map $X$ parametrises sufficiently regular surfaces $\Ga[X]$, the admissibility of the numerical triangulation over the whole time interval $[0,T]$ is preserved for sufficiently fine grids. 
\end{remark}

\section{\bf Proof of Lemma~\ref{lemma:three estimates}}
\label{section:proof - estimates}

The proof Lemma~\ref{lemma:three estimates} is separated into three subsections for the three estimates.

The proofs extend the main ideas of the proof of Proposition~7.1 of \cite{MCF} to the coupled mean curvature flow and diffusion system. Together they form the main technical part of the stability analysis. The first two estimates are based on energy estimates testing the error equations \eqref{eq:error equations - w} and \eqref{eq:error equations - u} with the time-derivative of the corresponding error. The third bound for the error in the velocity is shown using Lemma~5.3 of \cite{Willmore}. 
The proofs combines the approach of \cite[Proposition~7.1]{MCF} with those of \cite[Theorem~4.1]{MCF_soldriven} on handling the time-derivative term in $\bff$ in \eqref{eq:error equations - w}, of \cite[Proposition~7.1]{MCF_generalised} on dealing with the solution-dependent mass matrix $\bfM(\bfx,\bfu,\bfw)$ in \eqref{eq:error equations - w}.

The estimates for the terms with $\bfu$-dependent stiffness matrices in \eqref{eq:error equations - u} require new and more elaborate techniques, which are developed here, slightly inspired by the estimates for the stiffness-matrix differences in \cite{KLLP2017}.

Due to these reasons, a certain degree of familiarity with these papers (but at least \cite{MCF}) is required for a concise presentation of this proof.

\begin{remark}
	In addition to Remark~\ref{remark:notational conventions - 1}, throughout the present proof we will use the following conventions: 
	References to the proof Proposition~7.1 in \cite{MCF} are abbreviated to \cite{MCF}, unless a specific reference therein is given. For example, (i) in part (A) of the proof of Proposition~7.1 of \cite{MCF} is referenced as \cite[(A.i)]{MCF}. 	
\end{remark}

\subsection{\bf {Proof of \eqref{eq:A - combined estimate}}}
\label{subsection:energy est for geometric eqns}
\begin{proof}
We test \eqref{eq:error equations - w} with \blueon $\dotew$ \blueoff  and obtain:
\begin{equation}
	\begin{aligned}
		\dotew^T \bfM(\bfx,\bfu,\bfw) \dotew + \dotew^T \bfA(\bfx) \ew 
		= &\ - \dotew^T \big(\bfM(\bfx,\bfu,\bfw) - \bfM(\bfx,\us,\ws)\big) \dotws \\ 
		&\ - \dotew^T \big(\bfM(\bfx,\us,\ws) - \bfM(\xs,\us,\ws)\big) \dotws \\
		&\ - \dotew^T \big(\bfA(\bfx)-\bfA(\xs)\big) \ws \\
		&\ + \dotew^T \big(\bff(\bfx,\bfw,\bfu;\dot\bfu) - \bff(\xs,\ws,\us;\dotus) \big) \\
		&\ - \dotew^T \bfM(\xs) \dw .
	\end{aligned}
\end{equation}


(i) For the first term on the left-hand side, by the definition of $\bfM(\bfx,\bfu,\bfw) = \bfM(\bfx,\bfu,\bfV)$ and the $h$-uniform lower bound from \eqref{eq:uniform bounds for pa_2 Kh}, 
 we have
\begin{align*}
	\dotew^T \bfM(\bfx,\bfu,\bfw) \dotew = &\ \dotew^T \bfM(\bfx,\bfu,\bfV) \dotew \\
	= &\ \int_{\Ga_h[\bfx]} \!\!\!\! \pa_2 K(u_h,V_h) \, |\mat_h e_w|^2 \geq \tfrac12 K_0 \int_{\Ga_h[\bfx]} \!\!\!\! |\mat_h e_w|^2 = c_0  \|\dotew\|_{\bfM}^2 ,
\end{align*}
with the constant $c_0 = \tfrac12 K_0$, see Lemma~\ref{lemma:assumed bounds}.

(ii) By the symmetry of $\bfA$ and \eqref{matrix derivative bounds} we obtain
\begin{align*}
	\dotew^T \bfA(\bfx) \ew 
	= &\ - \frac{1}{2} \ew^T \diff \big(\bfA(\bfx)\big) \ew 
	+ \frac{1}{2} \diff \Big( \ew^T \bfA(\bfx) \ew \Big) 
	\\ &
	\geq - c \| \ew \|_{\bfA}^2 + \frac{1}{2} \diff \| \ew \|_{\bfA}^2 .
\end{align*}

(iii) On the right-hand side the two terms with differences of the solution-dependent mass matrix (recall the notation $\bfM(\bfx,\bfu,\bfw) = \bfM(\bfx,\bfu,\bfV)$) are estimated using Lemma~\ref{lemma:solution-dependent mass matrix diff}, together with \eqref{eq:assumed bounds} and \eqref{eq:W^{1,infty} bounds for exact solutions}, \eqref{eq:W^{1,infty} bounds for numerical solutions}. 

For the first term, by Lemma~\ref{lemma:solution-dependent mass matrix diff} (iv) (with $\dotew$, $\dotws$, and $\bfu$, $\us$ in the role of $\bfw$, $\bfz$, and $\bfu$, $\us$, respectively), together with \eqref{eq:assumed bounds}, the uniform bounds \eqref{eq:W^{1,infty} bounds for exact solutions} and \eqref{eq:W^{1,infty} bounds for numerical solutions}, and the $W^{1,\infty}$ bound on $\mat_h u_h^*$, proved in \cite[(A.iii)]{MCF}. Using the norm equivalence \eqref{norm equivalence}, we altogether obtain
\begin{align*}
	- \dotew \big( \bfM(\bfx,\bfu,\bfV) - \bfM(\bfx,\us,\Vs) \big) \dotws 
	\leq &\ c \big( \| \eu \|_{\bfM} + \| \bfe_\bfV \|_{\bfM} \big) \, \| \dotew \|_{\bfM} \\
	\leq &\ c \big( \| \eu \|_{\bfM} + \| \ew \|_{\bfM} \big) \, \| \dotew \|_{\bfM} .
\end{align*}

For the other term we apply Lemma~\ref{lemma:solution-dependent mass matrix diff} (ii) (with $\dotew$, $\dotws$, and $\us$ in the role of $\bfw$, $\bfz$, and $\bfu$, respectively), and obtain
\begin{align*}
	- \dotew^T \big( \bfM(\bfx,\us,\Vs)-\bfM(\xs,\us,\Vs) \big) \dotws
	\leq c \| \ex \|_{\bfA}  \, \| \dotew \|_{\bfM} ,
\end{align*}
where we again used the norm equivalence \eqref{norm equivalence}.

(iv) The third term on the right-hand side is estimated exactly as in \cite[(A.iv)]{MCF} by (recalling $\bfK = \bfM + \bfA$)
\begin{align*}
	- \dotew^T \big( \bfA(\bfx) - \bfA(\xs) \big) \ws 
	\leq &\ - \diff \Big( \ew^T \big( \bfA(\bfx) - \bfA(\xs) \big) \ws \Big) 
		\\ &\ 
	+ c \| \ew \|_{\bfA} \big( \| \ev \|_{\bfK} + \| \ex \|_{\bfK} \big) .
\end{align*}

(v) Before estimating the non-linear terms $\bff$, let us split off the part which depends linearly on $\dot{\bfu}$, cf.~\eqref{eq:rhs functions f_2}:
\begin{equation*}
	\bff(\bfx,\bfw,\bfu;\dot\bfu) = \widetilde{\bff}(\bfx,\bfw,\bfu) + \bfF(\bfx,\bfw,\bfu)\dot\bfu .
\end{equation*}

Since the estimates for the non-linear term \cite[(A.v)]{MCF} were shown for a general locally Lipschitz function, they apply for the estimates for the difference $\widetilde{\bff} -  \widetilde{\bff}^*$ as well (note, however, that $\bff$ defined in \cite[Section~3.3]{MCF} is different from the one here), and yield 
\begin{align*}
	\dotew^T \big( \widetilde{\bff}(\bfx,\bfw,\bfu) - \widetilde{\bff}(\xs,\ws,\us) \big) 
	&\leq c \| \dotew \|_{\bfM} \big( \| \ex \|_{\bfK} + \| \ew \|_{\bfA} + \| \eu \|_{\bfA} \big) .
\end{align*}

The remaining difference is bounded
\begin{align*}
	&\ \dotew^T \big( \bfF(\bfx,\bfw,\bfu) \dot\bfu - \bfF(\xs,\ws,\us) \dotus \big) \\
	\leq &\ \dotew^T \bfF(\bfx,\bfw,\bfu) \doteu + \dotew^T \big( \bfF(\bfx,\bfw,\bfu) - \bfF(\xs,\ws,\us) \big) \dotus \\
	\leq &\ c \|\dotew\|_{\bfM} \|\doteu\|_{\bfM} + c \|\dotew\|_{\bfM} \big( \| \ex \|_{\bfK} + \| \ew \|_{\bfA} + \| \eu \|_{\bfA} \big) ,
\end{align*}
where for the first term we have used \eqref{eq:rhs functions f_2} and the $W^{1,\infty}$ boundedness of the numerical solutions \eqref{eq:W^{1,infty} bounds for numerical solutions}, while the second term is bounded by the same arguments used for the previous estimate.

(vi) The defect term is simply bounded by the Cauchy--Schwarz inequality and a norm equivalence \eqref{norm equivalence}:
\begin{align*}
	- \dotew^T \bfM(\xs) \dw \leq &\ c \| \dotew \|_{\bfM} \|\dw\|_{\bfM^*} . 
\end{align*}

Altogether, collecting the estimates in (i)--(vi), and using Young's inequality and absorptions to the left-hand side, we obtain the desired estimate \eqref{eq:A - combined estimate}.
\qed
\end{proof}

\subsection{\bf {Proof of \eqref{eq:B - combined estimate}}}

\begin{proof}
We test \eqref{eq:error equations - u} with $\doteu$, and obtain:
\begin{equation}
\label{eq:error equation B - tested}
	\begin{aligned}
		\doteu^T \diff \Big( \bfM(\bfx) \eu \Big) + \doteu^T \bfA(\bfx,\us) \eu 
		= &\ - \doteu^T \diff \Big( \big(\bfM(\bfx) - \bfM(\xs) \big) \us \Big) \\
		&\ - \doteu^T \big( \bfA(\bfx,\bfu) - \bfA(\bfx,\us) \big) \eu \\
		&\ - \doteu^T \big( \bfA(\bfx,\bfu) - \bfA(\bfx,\us) \big) \us \\ 
		&\ - \doteu^T \big( \bfA(\bfx,\us) - \bfA(\xs,\us) \big) \us \\ 
		&\ - \doteu^T \bfM(\xs) \du .
	\end{aligned}
\end{equation}

To estimate these terms we again use the same techniques as in Section~\ref{subsection:energy est for geometric eqns}.

(i) For the first term on the left-hand side, using \eqref{matrix derivative bounds} and the Cauchy--Schwarz inequality, we obtain
\begin{align*}
	\doteu^T \diff \Big( \bfM(\bfx) \eu \Big) 
	= &\ \doteu^T \diff \big( \bfM(\bfx) \big) \eu + \|\doteu\|_{\bfM}^2 \\
	\geq &\ \|\doteu\|_{\bfM}^2 - c \|\doteu\|_{\bfM} \|\eu\|_{\bfM} .
\end{align*}
Then Young's inequality yields
\begin{equation}
\label{eq:c error est - i}
	\doteu^T \diff \Big( \bfM(\bfx) \eu \Big) 
	\geq \half \|\doteu\|_{\bfM}^2 - c \|\eu\|_{\bfM}^2 . 
\end{equation}

(ii) The second term on the left-hand side is bounded, using the symmetry of $\bfA(\bfx,\us)$ and \eqref{u-stiffness matrix derivative}, (via a similar argument as in (A.ii)), by
\begin{align*}
	\doteu^T \bfA(\bfx,\us) \eu 
	= &\  \frac{1}{2} \diff \Big( \eu^T \bfA(\bfx,\us) \eu \Big) - \frac{1}{2} \eu^T \diff \big(\bfA(\bfx,\us)\big) \eu \\ 
	\geq &\	\frac{1}{2} \diff \| \eu \|_{\bfA(\bfx,\us)}^2  - c \| \eu \|_{\bfA}^2 .
\end{align*}
For the estimate \eqref{u-stiffness matrix derivative}, which follows from Lemma~\ref{lemma:solution-dependent stiffness matrix time difference}, the latter requires the bounds
\begin{equation}
\label{eq:boundedness of mat_h u_h^*}
	\begin{aligned}
		\|\mat_h u_h^*\|_{L^{\infty}(\Ga_h^*)} \leq &\ R , \\
		\|\D'(u_h^*)\|_{L^{\infty}(\Ga_h^*)} \leq &\ R .
	\end{aligned}
\end{equation}
The first estimate is proved exactly as in \cite[(A.iii)]{MCF}, while the second one is shown (by a similar idea) using the local Lipschitz continuity of $\D'$ (Assumption~\ref{eq:assumptions D}) and the bounds \eqref{eq:W^{1,infty} bounds for exact solutions}, for sufficiently small $h$:
\begin{align*}
	\|\D'(u_h^*)\|_{L^{\infty}} \leq \|\D'(u_h^*) - \D'(u)\|_{L^{\infty}} + \|\D'(u)\|_{L^{\infty}} \leq 2 R .
\end{align*}

(iii) The time-differentiated mass matrix difference, the first term on the right-hand side of \eqref{eq:error equation B - tested}, is bounded, by the techniques for the analogous term in (A.iii) of the proof of \cite[Proposition~6.1]{KLLP2017}, by
\begin{align*}
	&\ - \doteu^T \diff \Big( \big( \bfM(\bfx) - \bfM(\xs) \big) \us \Big) \\
	\leq &\ c \| \mat_h e_u \|_{L^2(\Ga_h[\bfx])} \| \nb_{\Ga_h[\bfx]}e_x \|_{L^2(\Ga_h[\bfx])} \| \mat_h u_h^* \|_{L^\infty(\Ga_h[\bfx])} \\
	&\ + c \|\mat_h e_u\|_{L^2(\Ga_h[\bfx])} \big(\| \nb_{\Ga_h[\bfx]}e_x \|_{L^2(\Ga_h[\bfx])} + \| \nb_{\Ga_h[\bfx]}e_v \|_{L^2(\Ga_h[\bfx])} \big) \| \mat_h u_h^* \|_{L^\infty(\Ga_h[\bfx])} \\
	\leq &\ c \|\doteu\|_{\bfM} \|\ex\|_{\bfA} 
	+ c \|\doteu\|_{\bfM} \big(\|\ex\|_{\bfA} + \|\ev\|_{\bfA} \big) .
\end{align*}
For the last inequality we use \eqref{eq:boundedness of mat_h u_h^*}.

(iv) 
%
%
By the symmetry of the matrices $\bfA(\bfx,\bfu)$ and $\bfA(\bfx,\us)$, and the product rule, we obtain
\begin{equation}
\label{eq:estimate for term (iv) - pre}
	\begin{aligned}
		- \doteu^T \big(\bfA(\bfx,\bfu) - \bfA(\bfx,\us)\big) \eu 
		= &\ - \half \diff \Big( \eu^T \big(\bfA(\bfx,\bfu) - \bfA(\bfx,\us)\big) \eu \Big) \\
		&\ + \eu^T \diff \big(\bfA(\bfx,\bfu) - \bfA(\bfx,\us)\big) \eu .
	\end{aligned}
\end{equation}
The first term will be estimated later after an integration in time, while the second term is bounded similarly to \cite[(A.iv)]{MCF}. 

Lemma~\ref{lemma:solution-dependent stiffness matrix time difference} and the Leibniz formula yields, for any vectors $\bfw, \bfz \in \R^N$ (which satisfy $\mat_h w_h = \mat_h z_h = 0$), 
\begin{equation}
\label{eq:stiffness matrix difference diff - c variant}
\begin{aligned}
	&\ \bfw^T \diff \big(\bfA(\bfx,\bfu) - \bfA(\bfx,\us)\big) \bfz \\
	= &\ \diff \int_{\Ga_h[\bfx]} \big( \D(u_h) - \D(u_h^*) \big) \, \nb_{\Ga_h[\bfx]} w_h \cdot \nb_{\Ga_h[\bfx]} z_h \\
	= &\ \int_{\Ga_h[\bfx]} \mat_h \big( \D(u_h) - \D(u_h^*) \big) \, \nb_{\Ga_h[\bfx]} w_h \cdot \nb_{\Ga_h[\bfx]} z_h \\
	&\ + \int_{\Ga_h[\bfx]} \big( \D(u_h) - \D(u_h^*) \big) \, \mat_h (\nb_{\Ga_h[\bfx]} w_h) \cdot \nb_{\Ga_h[\bfx]} z_h \\
	&\ + \int_{\Ga_h[\bfx]} \big( \D(u_h) - \D(u_h^*) \big) \, \nb_{\Ga_h[\bfx]} w_h \cdot \mat_h(\nb_{\Ga_h[\bfx]} z_h) \\
	&\ + \int_{\Ga_h[\bfx]} \big( \D(u_h) - \D(u_h^*) \big) \, \nb_{\Ga_h[\bfx]} w_h \cdot \nb_{\Ga_h[\bfx]} z_h \big( \nb_{\Ga_h[\bfx]} \cdot v_h \big) \\
	=: &\ J_1 + J_2 + J_3 + J_4 .
\end{aligned}
\end{equation}
Here $v_h$ is the velocity of the discrete surface $\Ga_h[\bfx]$ with nodal values $\bfv$.

We now estimate the four terms separately. 
For the first term in $J_1$, we compute
\begin{align*}
	\big| \mat_h \big( \D(u_h) - \D(u_h^*) \big) \big|
	= &\ \big| \D'(u_h)  (\mat_h u_h - \mat_h u_h^*)  + ( \D'(u_h) - \D'(u_h^*) ) \mat_h u_h^* \big| \\
	\leq &\ |\D'(u_h)| |\mat_h e_u| + |\D'(u_h) - \D'(u_h^*)| |\mat_h u_h^*| .
\end{align*}
The local Lipschitz continuity of $\D'$ and \eqref{eq:boundedness of mat_h u_h^*}, together with a H\"older inequality then yields
\begin{equation*}
	J_1 
	\leq c \big( \|\mat_h e_u\|_{L^2(\Ga_h[\bfx])} + \|e_u\|_{L^2(\Ga_h[\bfx])} \big) \|\nb_{\Ga_h[\bfx]} w_h\|_{L^2(\Ga_h[\bfx])} \|z_h\|_{W^{1,\infty}(\Ga_h[\bfx])} .
\end{equation*}

The two middle terms are estimated by first interchanging $\mat_h$ and $\nb_{\Ga_h[\bfx]}$ via the formula \eqref{eq:mat-grad formula}.
Using \eqref{eq:mat-grad formula} together with $\mat_h w_h = \mat_h z_h = 0$ and the boundedness of $\nu_{\Ga_h[\bfx]}$ and $v_h$ \eqref{eq:W^{1,infty} bounds for numerical solutions}, we obtain the estimate
\begin{align*}
	J_2 + J_3 
	\leq &\ c \|e_u\|_{L^2(\Ga_h[\bfx])} \|\nb_{\Ga_h[\bfx]} w_h\|_{L^2(\Ga_h[\bfx])} \|z_h\|_{W^{1,\infty}(\Ga_h[\bfx])} .
\end{align*}

The last term $J_4$ is estimated by a similar argument as the first, now using the local Lipschitz continuity of $\D'$ and the $h$-uniform $W^{1,\infty}$ boundedness of $v_h$ \eqref{eq:W^{1,infty} bounds for numerical solutions}:
\begin{equation*}
	J_4 
	\leq c \|e_u\|_{L^2(\Ga_h[\bfx])} \|\nb_{\Ga_h[\bfx]} w_h\|_{L^2(\Ga_h[\bfx])} \|z_h\|_{W^{1,\infty}(\Ga_h[\bfx])} .
\end{equation*}

By combining these bounds, and recalling the $W^{1,\infty}$ norm bounds \eqref{eq:assumed bounds} and \eqref{eq:W^{1,infty} bounds for numerical solutions}, we obtain
\begin{align*}
	&\ \eu^T \diff \big(\bfA(\bfx,\bfu) - \bfA(\bfx,\us)\big) \eu \\
	\leq &\ c \big( \|\mat_h e_u\|_{L^2(\Ga_h[\bfx])} + \|e_u\|_{L^2(\Ga_h[\bfx])} \big) \|\nb_{\Ga_h[\bfx]} e_u\|_{L^2(\Ga_h[\bfx])} \bbk \| e_u \|_{W^{1,\infty}(\Ga_h[\bfx])} \ebk \\
	\leq &\ c \big( \|\doteu\|_{\bfM} + \|\eu\|_{\bfM} \big) \|\eu\|_{\bfA} .
\end{align*}

Altogether we obtain
\begin{align*}
	- \doteu^T \big(\bfA(\bfx,\bfu) - \bfA(\bfx,\us)\big) \eu 
	\leq &\ - \half \diff \Big( \eu^T \big(\bfA(\bfx,\bfu) - \bfA(\bfx,\us)\big) \eu \Big) \\
	&\ + c \big( \|\doteu\|_{\bfM} + \|\eu\|_{\bfM} \big) \|\eu\|_{\bfA} ,
\end{align*}
which does not contain a critical term $\|\doteu\|_{\bfA}$.

(v) Almost verbatim as the argument in (iv) we rewrite and estimate the third term on the right-hand side of \eqref{eq:error equation B - tested} as
\begin{align*}
	&\ - \doteu^T \big(\bfA(\bfx,\bfu) - \bfA(\bfx,\us)\big) \us \\
	= &\ - \diff \Big( \eu^T \big(\bfA(\bfx,\bfu) - \bfA(\bfx,\us)\big) \us \Big) \\
	&\ + \eu^T \diff \Big( \bfA(\bfx,\bfu) - \bfA(\bfx,\us) \Big) \us  + \eu^T \big(\bfA(\bfx,\bfu) - \bfA(\bfx,\us)\big) \dotus \\
	\leq &\ - \diff \Big( \eu^T \big(\bfA(\bfx,\bfu) - \bfA(\bfx,\us)\big) \us \Big) \\
	&\ + c \big( \|\doteu\|_{\bfM} + \|\eu\|_{\bfM} \big) \|\eu\|_{\bfA} + c \|\eu\|_{\bfM} \|\eu\|_{\bfA} .
\end{align*}
For the non-differentiated term here we used Lemma~\ref{lemma:solution-dependent stiffness matrix diff} (ii) (together with \eqref{eq:W^{1,infty} bounds for exact solutions} and \eqref{eq:W^{1,infty} bounds for numerical solutions}) and the $W^{1,\infty}$ variant of \eqref{eq:boundedness of mat_h u_h^*}. The latter is shown (omitting the argument $t$) by
\begin{equation}
\label{eq:boundedness of mat_h u_h^* - grad}
	\begin{aligned}
		&\ \|\mat_h u_h^*\|_{W^{1,\infty}(\Ga_h^*)} \leq 
		c \| (\mat_h u_h^*)^\ell  \|_{W^{1,\infty}(\Ga[X])}
		\\
		& \leq c \| (\mat_h u_h^{*})^\ell  - I_h \mat u\|_{W^{1,\infty}(\Ga[X])} \\
		& \quad +
		c \| I_h  \mat u -  \mat u \|_{W^{1,\infty}(\Ga[X])}
		+ c \|   \mat u \|_{W^{1,\infty}(\Ga[X])} \\
		& \le
		\frac c h \| (\mat_h u_h^{*})^\ell  - I_h \mat u\|_{H^1(\Ga[X])} 
		\\ & \quad +
		c \| I_h  \mat u -  \mat u \|_{W^{1,\infty}(\Ga[X])}
		+ c \|    \mat u \|_{W^{1,\infty}(\Ga[X])}
		\\
		&\le  \frac c h \| (\mat_h u_h^{*})^\ell  -  \mat u\|_{H^1(\Ga[X])} +
		\frac c h \|  \mat u - I_h \mat u\|_{H^1(\Ga[X])} 
		\\ & \quad +
		c \| I_h  \mat u -  \mat u \|_{W^{1,\infty}(\Ga[X])}
		+ c \|   \mat u \|_{W^{1,\infty}(\Ga[X])}
		\\
		& \leq Ch^{k-1} + Ch^{k-1} + Ch^k +  C .
	\end{aligned}
\end{equation}
Here we subsequently used the norm equivalence for the lift operator (see \cite[(2.15)--(2.16)]{Demlow2009}) in the first inequality,
an inverse inequality \cite[Theorem~4.5.11]{BrennerScott} in the second inequality, and the known error bounds for interpolation (see \cite[Proposition~2]{Demlow2009}) 
and for the Ritz map $u_h^* = \Rh^u u$ (a direct modification of \cite[Theorem~3.1]{KPower_quasilinear}) in the last inequality.

(vi) The argument for the fourth term on the right-hand side of \eqref{eq:error equation B - tested} is slightly more complicated since it compares stiffness matrices on different surfaces. We will estimate this term by a $\D$-weighted extension of the argument of \cite[(A.iv)]{MCF}.

We start by rewriting this term as a total derivative
\begin{equation*}
	\begin{aligned}
		&\ - \doteu^T \big(\bfA(\bfx,\us) - \bfA(\xs,\us)\big) \us \\
		= &\ - \diff \Big( \eu^T \big(\bfA(\bfx,\us) - \bfA(\xs,\us)\big) \us \Big) \\
		&\ + \eu^T \diff \Big( \bfA(\bfx,\us) - \bfA(\xs,\us) \Big) \us  + \eu^T \big(\bfA(\bfx,\us) - \bfA(\xs,\us)\big) \dotus \\
		\leq &\ - \diff \Big( \eu^T \big(\bfA(\bfx,\us) - \bfA(\xs,\us)\big) \us \Big) \\
		&\ + \eu^T \diff \Big( \bfA(\bfx,\us) - \bfA(\xs,\us) \Big) \us + c \|\eu\|_{\bfA} \|\ex\|_{\bfA} ,
	\end{aligned}
\end{equation*}
where we now used Lemma~\ref{lemma:solution-dependent stiffness matrix diff} (i) (together with \eqref{eq:W^{1,infty} bounds for exact solutions} and \eqref{eq:W^{1,infty} bounds for numerical solutions}).

The remaining term is bounded similarly to \eqref{eq:stiffness matrix difference diff - c variant}. In the setting of Section~\ref{section:relating surfaces}, analogously to Lemma~\ref{results collection}, using Leibniz formula we obtain, for any vectors $\bfw, \bfz \in \R^N$, \bbk but for a fixed $\us \in \R^N$ in both matrices, \ebk 
\begin{equation}
\label{eq:stiffness matrix difference diff - x variant}
	\begin{aligned}
		& \bfw^T \Big(\diff \big( \bfA(\bfx,\us)-\bfA(\xs,\us) \big) \Big) \bfz \\
		& =\frac{\d}{\d t} \int_0^1 \int_{\Ga_h^\theta} \D(u_h^{*,\theta}) \, \nb_{\Ga_h^\theta} w_h^\theta \cdot (D_{\Ga_h^\theta} e_x^\theta) \nb_{\Ga_h^\theta}  z_h^\theta \ \d\theta  \\
		&= 
		\int_0^1 \int_{\Ga_h^\theta} \mat_{\Ga_h^\theta} \big( \D(u_h^{*,\theta}) \big) \nb_{\Ga_h^\theta} w_h^\theta \cdot (D_{\Ga_h^\theta} e_x^\theta)  \nb_{\Ga_h^\theta}  z_h^\theta \ \d\theta  \\
		&\quad +
		\int_0^1 \int_{\Ga_h^\theta} \D(u_h^{*,\theta}) \,\mat_{\Ga_h^\theta} \big(\nb_{\Ga_h^\theta} w_h^\theta \big) \cdot (D_{\Ga_h^\theta} e_x^\theta)  \nb_{\Ga_h^\theta}  z_h^\theta \ \d\theta  \\
		&\quad +
		\int_0^1 \int_{\Ga_h^\theta} \D(u_h^{*,\theta}) \,\nb_{\Ga_h^\theta} w_h^\theta \cdot \mat_{\Ga_h^\theta} \big( D_{\Ga_h^\theta} e_x^\theta \big) \nb_{\Ga_h^\theta}  z_h^\theta \ \d\theta  \\
		&\quad 
		+ \int_0^1 \int_{\Ga_h^\theta} \D(u_h^{*,\theta}) \,\nb_{\Ga_h^\theta} w_h^\theta \cdot (D_{\Ga_h^\theta} e_x^\theta) \mat_{\Ga_h^\theta} \big( \nb_{\Ga_h^\theta}  z_h^\theta \big) \ \d\theta  \\
		&\quad 
		+ \int_0^1 \int_{\Ga_h^\theta} \D(u_h^{*,\theta}) \,\nb_{\Ga_h^\theta} w_h^\theta \cdot (D_{\Ga_h^\theta} e_x^\theta) \nb_{\Ga_h^\theta}  z_h^\theta \, (\nb_{\Ga_h^\theta} \cdot  v_{\Ga_h^\theta}) \d\theta \\
		&=: \int_0^1 \big( J_0^\theta + J_1^\theta + J_2^\theta + J_3^\theta + J_4^\theta \big) \,\d\theta .
	\end{aligned}
\end{equation}
We recall from Section~\ref{section:relating surfaces} that
$\Ga_h^\theta(t)$ is the discrete surface with nodes $\xs(t)+\theta \ex(t)$ (with unit normal field $\nu_h^\theta := \nu_{\Ga_h^\theta}$), and with finite element space $S_h[\xs(t)+\theta \ex(t)]$. 
The function $u_h^{*,\theta}=u_h^{*,\theta}(\cdot,t) \in S_h[\xs(t)+\theta \ex(t)]$ with $\theta$-independent nodal values $\us\t$.
\bbk \ebk 
We denote by $w_h^\theta(\cdot,t)$ and $z_h^\theta(\cdot,t)$ the finite element functions in $S_h[\xs(t)+\theta \ex(t)]$ with the time- and $\theta$-independent nodal vectors $\bfw$ and $\bfz$, respectively.
The velocity of $\Ga_h^\theta(t)$ is $v_{\Ga_h^\theta}(\cdot,t)$ (as a function of $t$), which is the finite element function in $S_h[\xs(t)+\theta \ex(t)]$ with nodal vector $\dotxs(t) + \theta \dot\bfe_\bfx(t) = \vs(t) + \theta \ev(t)$. Related to this velocity, $\mat_{\Ga_h^\theta}$ denotes the corresponding material derivative on ${\Ga_h^\theta}$. We thus have
\begin{equation}
\label{vhtheta}
	v_{\Ga_h^\theta} = v_h^{*,\theta} + \theta e_v^\theta .
\end{equation} 

The various time- and $\theta$-independencies imply
\begin{equation}
\label{eq:mat derivatives of things - zeros}
	\begin{gathered}
		\mat_{\Ga_h^\theta} w_h^\theta = 0, \qquad \mat_{\Ga_h^\theta} z_h^\theta = 0 , 
	\end{gathered}
\end{equation}
and since for the nodal vectors we have $\dotex = \ev$ \eqref{eq:error equations - x} we also have
\begin{equation}
\label{eq:mat derivative of e_x}
	\mat_{\Ga_h^\theta} e_x^\theta = e_v^\theta .
\end{equation}

\bbk 
The terms $J_k^\theta$ for $k = 0,\dotsc,4$ are bounded almost exactly as the analogous terms in \cite[(A.iv)]{MCF}.

For the first term we have $\mat_{\Ga_h^\theta} \big( \D(u_h^{*,\theta}) \big) = \D'(u_h^{*,\theta}) \, \mat_{\Ga_h^\theta} u_h^{*,\theta}$, this together with \eqref{eq:boundedness of mat_h u_h^*} and recalling that $u_h^{*,\theta}$ is $\theta$-independent yields $\|\D'(u_h^{*,\theta}) \, \mat_{\Ga_h^\theta} u_h^{*,\theta}\|_{L^\infty(\Ga_h^\theta)} \leq R^2$, cf.~the proof of Lemma~\ref{lemma:solution-dependent stiffness matrix time difference}. We then obtain bound
\begin{equation*}
	J_0^\theta \leq c \|\nb_{\Ga_h^\theta} w_h\|_{L^2(\Ga_h^\theta)} \|\nb_{\Ga_h^\theta} e_x^\theta\|_{L^2(\Ga_h^\theta)} \|z_h^\theta\|_{W^{1,\infty}(\Ga_h^\theta)} .
\end{equation*}
\ebk 

The identities in \eqref{eq:mat derivatives of things - zeros} in combination with the interchange formula \eqref{eq:mat-grad formula} yield
\begin{equation*}
	J_1^\theta + J_3^\theta \leq c \|\nb_{\Ga_h^\theta} w_h\|_{L^2(\Ga_h^\theta)} \|\nb_{\Ga_h^\theta} e_x^\theta\|_{L^2(\Ga_h^\theta)} \|z_h^\theta\|_{W^{1,\infty}(\Ga_h^\theta)} ,
\end{equation*}
where we have used the uniform boundedness of $K$ \eqref{eq:assumptions D}.

The interchange formula for $\mat_{\Ga_h^\theta}$ and $D_{\Ga_h^\theta}$, cf.~\cite[equation~(7.27)]{MCF}, analogous to \eqref{eq:mat-grad formula}, and reads
\begin{equation}
\label{eq:mat-D formula}
	\begin{aligned}
		\mat_{\Gamma_h^\theta} (D_{\Ga_h^\theta} e_x^\theta) = &\ \mat_{\Gamma_h^\theta} \Big( \textnormal{tr}(\nb_{\Ga_h^\theta} e_x^\theta) - \big( \nb_{\Ga_h^\theta} e_x^\theta+(\nb_{\Ga_h^\theta} e_x^\theta)^T \big) \Big) \\
		=&\ D_{\Ga_h^\theta} (\mat_{\Gamma_h^\theta} e_x^\theta) 
		+ \textnormal{tr}(\bar E^\theta) - (\bar E^\theta+(\bar E^\theta)^T) ,
	\end{aligned}
\end{equation}
with $\bar E^\theta= - \bigl(\nabla_{\Gamma_h^\theta} v_{\Ga_h^\theta} -	\nu_h^\theta(\nu_h^\theta)^T(\nabla_{\Gamma_h^\theta } v_{\Ga_h^\theta} )^T \bigr) \nabla_{\Gamma_h^\theta} e_x^\theta$, as follows from \cite[Lemma~2.6]{DziukKronerMuller} and the definition of the first order linear differential operator $D_{\Gamma_h^\theta}$.

The interchange identity \eqref{eq:mat-D formula} and \eqref{eq:mat derivative of e_x} (together with \eqref{eq:assumptions D}) then yields
\begin{equation*}
	J_2^\theta \leq c \|\nb_{\Ga_h^\theta} w_h^\theta\|_{L^2(\Ga_h^\theta)} \big( \|\nb_{\Ga_h^\theta} e_x^\theta\|_{L^2(\Ga_h^\theta)} + \|\nb_{\Ga_h^\theta} e_v^\theta\|_{L^2(\Ga_h^\theta)} \big) \|z_h^\theta\|_{W^{1,\infty}(\Ga_h^\theta)} .
\end{equation*}

The last term is directly bounded, using the $W^{1,\infty}$ boundedness of $v_h$, as
\begin{equation*}
	J_4^\theta \leq c \|\nb_{\Ga_h^\theta} w_h\|_{L^2(\Ga_h^\theta)} \|\nb_{\Ga_h^\theta} e_x^\theta\|_{L^2(\Ga_h^\theta)} \|z_h^\theta\|_{W^{1,\infty}(\Ga_h^\theta)} .
\end{equation*}

Using the norm equivalences \eqref{eq:lemma - theta-independence} for the bounds of $(J_k^\theta)_{k=0}^4$ we obtain
\begin{align*}
	& \bfw^T \Big(\diff \big( \bfA(\bfx,\us)-\bfA(\xs,\us) \big) \Big) \bfz \\
	\leq &\ c \|\nb_{\Ga_h[\bfx]} w_h\|_{L^2(\Ga_h[\bfx])} \big( \|\nb_{\Ga_h[\bfx]} e_x\|_{L^2(\Ga_h[\bfx])} + \|\nb_{\Ga_h[\bfx]} e_v\|_{L^2(\Ga_h[\bfx])} \big) \|z_h\|_{W^{1,\infty}(\Ga_h[\bfx])} .
\end{align*} 

Altogether, using the bound \eqref{eq:boundedness of mat_h u_h^* - grad}, we obtain
\begin{equation*}
	\begin{aligned}
		- \doteu^T \big(\bfA(\bfx,\us) - \bfA(\xs,\us)\big) \us 
		\leq &\ - \diff \Big( \eu^T \big(\bfA(\bfx,\us) - \bfA(\xs,\us)\big) \us \Big) \\
		&\ + c \|\eu\|_{\bfA} \big( \|\ex\|_{\bfA} + \|\ev\|_{\bfA} \big) + c \|\eu\|_{\bfA} \|\ex\|_{\bfA} .
	\end{aligned}
\end{equation*}

(vii) Finally, the defect terms are bounded by
\begin{equation*}
	- \doteu^T \bfM(\xs) \du \leq c \| \doteu \|_{\bfM} \|\du\|_{\bfM^*} . 
\end{equation*}

Altogether, collecting the above estimates in (i)--(vii), and using Young's inequality and absorptions to the left-hand side, we obtain the desired  estimate (recalling $\bfK = \bfM + \bfA$). 
\qed
\end{proof}

\subsection{\bf {Proof of \eqref{eq:velocity error estimate}}}
\begin{proof}
The error equation for the velocity law is estimated exactly as the velocity error equation for Willmore flow \cite[(B)]{Willmore} (based on Lemma~5.3 therein), and recalling $\bfu = ( \bfn , \bfV )^T \in \R^{4N}$, we obtain
\begin{equation}
\label{eq:velocity error estimate - in proof}
	\begin{aligned}
	\|\ev\|_{\bfK} \leq &\ c \big( \|\bfe_\bfV\|_{\bfK} + \|\bfe_\bfn\|_{\bfK} \big) + \|\dv\|_{\bfK} \\
	\leq &\ c \|\eu\|_{\bfK} + c \|\dv\|_{\bfK^*} ,
	\end{aligned}
\end{equation}
where for the last inequality we used a norm equivalence \eqref{norm equivalence}.
\qed
\end{proof}

\section{\bf Proof of Proposition \ref {proposition:stability - semi-discrete} }
\label{section:proof - stability}

\begin{proof}
The aim is to to combine the three  estimates of Lemma~\ref{lemma:three estimates}, following the main ideas of \cite{MCF}. 

First we use the differential equation $\dotex = \ev$ \eqref{eq:error equations - x}, using \eqref{matrix derivative bounds}, to show the bound
\begin{equation}
\label{eq:position error estimate}
	\begin{aligned}
		\|\ex\t\|_{\bfK(\xs\t)}^2 = &\ \int_0^t \frac{\d}{\d s} \|\ex\s\|_{\bfK(\xs\s)}^2 \d s \\
		\leq &\ c \int_0^t \|\ev\s\|_{\bfK(\xs\s)}^2 \d s + c \int_0^t \|\ex\s\|_{\bfK(\xs\s)}^2 \d s  .
	\end{aligned}
\end{equation}

 We first take the weighted linear combination of \eqref{eq:A - combined estimate} and \eqref{eq:B - combined estimate} with weights $1$ and $8 c_1$, respectively, to absorb the term $c_1 \|\doteu\|_{\bfM}^2$ from \eqref{eq:A - combined estimate}. We then obtain
\begin{equation}
\label{eq:combined estimates - pre int}
	\begin{aligned}
		&\ \frac{c_0}{2} \|\dotew\|_{\bfM}^2 + \frac{1}{2} \diff \| \ew \|_{\bfA}^2 
		+ c_1 \, \|\doteu\|_{\bfM}^2 + 4 c_1 \, \diff \| \eu \|_{\bfA(\bfx,\us)}^2 \\
		\leq &\
		c\, \big( \| \ex \|_{\bfK}^2 + \| \ev \|_{\bfK}^2 + \| \ew \|_{\bfK}^2 + \| \eu \|_{\bfK}^2 \big) \\
		&\ + c \, \big( \|\du\|_{\bfM^*}^2 + \|\dw\|_{\bfM^*}^2 \big) \\
		&\ - \diff \Big( \ew^T \big( \bfA(\bfx) - \bfA(\xs) \big) \ws \Big) \\
		&\ - c \half \diff \Big( \eu^T \big(\bfA(\bfx,\bfu) - \bfA(\bfx,\us)\big) \eu \Big) \\
		&\ - c \diff \Big( \eu^T \big(\bfA(\bfx,\bfu) - \bfA(\bfx,\us)\big) \us \Big) \\
		&\ - c \diff \Big( \eu^T \big(\bfA(\bfx,\us) - \bfA(\xs,\us)\big) \us \Big) .
	\end{aligned}
\end{equation}

We now connect $\|\dot\bfe\|_{\bfM}^2$ with $\d / \d t \, \|\bfe\|_{\bfM}^2$, \blueon and will use the result either for $\ex$ or for $\eu$ in place of $\bfe$. \blueoff We estimate, using \eqref{matrix derivative bounds}, by
\begin{equation}
\label{eq:dote - e estimate}
	\begin{aligned}
		\diff \|\bfe\|_{\bfM}^2 = &\ 2 \bfe^T \bfM(\bfx) \dot\bfe + \bfe^T \diff \big(\bfM(\bfx)\big) \bfe^T \\
		\leq &\ c \|\dot\bfe\|_{\bfM} \|\bfe\|_{\bfM} + c \|\bfe\|_{\bfM}^2 \\
		\leq &\ \varrho \|\dot\bfe\|_{\bfM}^2 + \varrho^{-1} c \|\bfe\|_{\bfM}^2 ,
	\end{aligned}
\end{equation}
where for the last step we used Young's inequality with an arbitrary number $\varrho > 0$, to be chosen later on independently of $h$. We will also use its time-integrated version:
\begin{equation}
\label{eq:dote - e estimate - time-int}
	\|\bfe\t\|_{\bfM\t}^2 \leq \varrho \int_0^t \|\dot\bfe\s\|_{\bfM\s}^2 \d s + \varrho^{-1} c\int_0^t  \|\bfe\s\|_{\bfM\s}^2 \d s + \|\bfe(0)\|_{\bfM(0)}^2 .
\end{equation}

Using \eqref{eq:dote - e estimate} for $\bfe = \ew$ with $\varrho = 1$, the left-hand side of \eqref{eq:combined estimates - pre int} simplifies to 
\begin{equation*}
	\frac{c_0}{2} \diff \|\ew\|_{\bfM}^2 + \frac{1}{2} \diff \| \ew \|_{\bfA}^2 
	+ c_1 \, \|\doteu\|_{\bfM}^2 + 4 c_1 \, \diff \| \eu \|_{\bfA(\bfx,\us)}^2 ,
\end{equation*}
with additional terms on the right-hand side which already appeared before.

In order to estimate the remaining time derivatives on the left-hand side of \eqref{eq:combined estimates - pre int} we integrate both sides form $0$ to $t$ and use the norm equivalence \eqref{eq:Axu norm equivalence}, recalling $\bfK = \bfM + \bfA$ we obtain (after factoring out a constant and dividing through):
\begin{equation}
\label{eq:combined estimates - pre tricks}
	\begin{aligned}
		&\ \|\ew\t\|_{\bfK\t}^2 + c_1 \int_0^t \|\doteu\s\|_{\bfM\s}^2 \d s + \| \eu\t \|_{\bfA\t}^2 \\
		\leq &\ c \int_0^t \big( \| \ex\s \|_{\bfK\s}^2 + \| \ev\s \|_{\bfK\s}^2 + \| \ew\s \|_{\bfK\s}^2 + \| \eu\s \|_{\bfK\s}^2 \big) \d s \\
		&\ + c \int_0^t \big( \|\dw\s\|_{\bfM^*\s}^2 + \|\du\s\|_{\bfM^*\s}^2 \big) \d s \\
		&\ + \|\ew(0)\|_{\bfK(0)}^2 + \|\eu(0)\|_{\bfM(0)}^2 + \| \eu(0) \|_{\bfA(\bfx(0),\us(0))}^2 \\
		&\ + c \, \| \ew\t \|_{\bfK\t} \| \ex\t \|_{\bfK\t} \\
		&\ + c \, \| \eu\t \|_{\bfA\t} \| \eu\t \|_{\bfM\t} + c \, \| \eu(0) \|_{\bfA(0)} \| \eu(0) \|_{\bfM(0)} \\
		&\ + c \, \| \eu\t \|_{\bfK\t} \| \ex\t \|_{\bfK\t} ,
	\end{aligned}
\end{equation}
where for the four time-differentiated terms on the right-hand side (cf.~\eqref{eq:combined estimates - pre int}), we used, in order, \eqref{matrix difference bounds e_x}, Lemma~\ref{lemma:solution-dependent stiffness matrix diff} (ii) twice, and Lemma~\ref{lemma:solution-dependent stiffness matrix diff} (i), in combination with the uniform $W^{1,\infty}$ norm bounds \eqref{eq:assumed bounds} and \eqref{eq:W^{1,infty} bounds for exact solutions}, \eqref{eq:W^{1,infty} bounds for numerical solutions}.

Using Young's inequality and absorptions to the left-hand side yields (after factoring out a constant and dividing through)
\begin{equation}
\label{eq:combined estimates - pre Gronwall}
	\begin{aligned}
		&\ \|\ew\t\|_{\bfK\t}^2 + c_1 \int_0^t \|\doteu\s\|_{\bfM\s}^2 \d s + \| \eu\t \|_{\bfA\t}^2  \\
		\leq &\ c \int_0^t \big( \| \ex\s \|_{\bfK\s}^2 + \| \ev\s \|_{\bfK\s}^2 + \| \ew\s \|_{\bfK\s}^2 + \| \eu\s \|_{\bfK\s}^2 \big) \d s \\
		&\ + c \int_0^t \big( \|\dw\s\|_{\bfM^*\s}^2 + \|\du\s\|_{\bfM^*\s}^2 \big) \d s \\
		&\ + c \big( \|\eu(0)\|_{\bfK(0)}^2 + \|\eu(0)\|_{\bfK(0)}^2 \big) \\
		&\ + c \, \| \ex\t \|_{\bfK\t}^2 
		+ c_2 \, \| \eu\t \|_{\bfM\t}^2 .
	\end{aligned}
\end{equation}

In order to apply Gronwall's inequality we need to estimate the last two terms of \eqref{eq:combined estimates - pre Gronwall}: The first of which is bounded by \eqref{eq:position error estimate}. The second is estimated using \eqref{eq:dote - e estimate - time-int} for $\bfe = \eu$ with a factor $\varrho > 0$ such that $c_2 \varrho < c_1 / 2$, allowing an absorption to the left-hand side. This, and using \eqref{eq:dote - e estimate - time-int} for $\bfe = \eu$ with a factor $\varrho = 1$ now on the left-hand side, yields 
\begin{equation}
\label{eq:combined estimates - pre Gronwall - 2}
\begin{aligned}
	&\ \|\ew\t\|_{\bfK\t}^2 + \|\eu\t\|_{\bfK\t}^2  \\
	\leq &\ c \int_0^t \big( \| \ex\s \|_{\bfK\s}^2 + \| \ev\s \|_{\bfK\s}^2 + \| \ew\s \|_{\bfK\s}^2 + \| \eu\s \|_{\bfK\s}^2 \big) \d s \\
	&\ + c \int_0^t \big( \|\dw\s\|_{\bfM^*\s}^2 + \|\du\s\|_{\bfM^*\s}^2 \big) \d s \\
	&\ + c \big( \|\ew(0)\|_{\bfK(0)}^2 + \|\eu(0)\|_{\bfK(0)}^2 + \| \eu(0) \|_{\bfA(\bfx(0),\us(0))}^2 \big).
\end{aligned}
\end{equation}

Substituting \eqref{eq:combined estimates - pre Gronwall - 2} into the right-hand side of \eqref{eq:velocity error estimate}, and then summing up the resulting estimate with \eqref{eq:combined estimates - pre Gronwall - 2} with \eqref{eq:position error estimate} (analogously as in \cite[(C)]{MCF}), and using the norm equivalence \eqref{norm equivalence}, and then finally using Gronwall's inequality we obtain the stated stability estimate \eqref{eq:stability estimate} for all $0 \leq t \leq T^*$.

Finally, it remains to show that in fact $T^\ast = T$ for $h$ sufficiently small, i.e.~that Lemma~\ref{lemma:assumed bounds} holds for thw whole time interval $[0,T]$. 
Upon noting that by the assumed defect bounds \eqref{eq:assumed defect bounds} and \eqref{eq:assumed initial error bounds}, combined with the obtained stability bound \eqref{eq:stability estimate}, yields
\begin{equation*}
\| \ex(t) \|_{\bfK(\xs\t)}
+ \| \ev(t) \|_{\bfK(\xs\t)}
+ \| \ew(t) \|_{\bfK(\xs\t)}
+ \| \eu(t) \|_{\bfK(\xs\t)}
\leq C h^\kappa , 
\end{equation*}
and therefore by an inverse inequality \cite[Theorem~4.5.11]{BrennerScott} we obtain, for $t \in [0,T^*]$,
\begin{equation}
\label{eq:showing assumed bound beyond t*}
\begin{aligned}
&\ \| e_x(\cdot,t) \|_{W^{1,\infty}(\Ga_h[\xs\t])} 
+ \| e_v(\cdot,t) \|_{W^{1,\infty}(\Ga_h[\xs\t])} \\
&\ \qquad 
+ \| e_w(\cdot,t) \|_{W^{1,\infty}(\Ga_h[\xs\t])} 
+ \| e_u(\cdot,t) \|_{W^{1,\infty}(\Ga_h[\xs\t])} \\
&\ \ \leq \frac{c}{h} \Big( \| \ex(t) \|_{\bfK(\xs\t)}
+\| \ev(t) \|_{\bfK(\xs\t)}
+\| \ew(t) \|_{\bfK(\xs\t)} 
+\| \eu(t) \|_{\bfK(\xs\t)} \Big) \\
&\ \ \leq c \, C \, h^{\kappa-1} \leq \frac{1}{2} h^{(\kappa-1)/2} ,
\end{aligned}
\end{equation}
for sufficiently small $h$. This means that the bounds \eqref{eq:assumed bounds}, (and hence all other estimates in Lemma~\ref{lemma:assumed bounds}), can be extended beyond $T^*$, contradicting the maximality of $T^\ast$, unless $T^\ast = T$ already. Therefore we have show the stability bound \eqref{eq:stability estimate} over the whole time interval $[0,T]$.
\qed
\end{proof}

\section{Numerical experiments}
\label{section:numerics}

\blueon We performed  numerical simulations and experiments  for the flow \eqref{eq:MCFdiff} formulated as Problem~\ref{P1} in which: \blueoff 
\begin{itemize}
	\item[-] \blueon The rate of convergence in an example involving  a radially symmetric exact solution is studied in order to illustrate the  theoretical results of Theorem~\ref{theorem:semi-discrete error estimates}.\blueoff 
	\item[-]  \blueon Flows decreasing the energy $\calE(\Ga,u) = \int_\Ga G(u)$ and their qualitative properties, see Section~\ref{section:examples and properties},  are investigated in several simulations.  \blueoff 
	\item[-] \blueon Numerical experiments exhibiting loss of convexity and self-intersection are presented. This is in contrast to mean curvature flow which \emph{preserves convexity}, and  for which \emph{self-intersections} are not possible. 
	\blueoff 
	\item[-]  \blueon The preservation under discretisation of  mass conservation and  the existence of a weak maximum principle together with the  energy decay, and mean convexity properties enjoyed by the underlying PDE system are studied.  \blueoff 
\end{itemize}

The numerical experiments use quadratic evolving surface finite elements.  Quadratures of sufficiently high order  are employed  to compute the finite element vectors and matrices so that the resulting quadrature error does not feature in the discussion of the accuracies of the schemes. Similarly, sufficiently high-order linearly implicit BDF time discretisations, see the next Section~\ref{eq:BDF},  with small time steps were employed for the solution of the time dependent ODE system.
The parametrisation of the quadratic elements was inspired by \cite{BCH2006}. 
The initial meshes were all generated using DistMesh \cite{distmesh}, without taking advantage of any symmetry of the surface.
\subsection{\bf Linearly implicit backward difference full discretisation}
\label{section:BDF}

For the time discretisation of the system of ordinary differential equations \eqref{eq:matrix-vector form - P1} we use a $q$-step linearly implicit backward difference formula (BDF method). For a step size $\tau>0$, and with $t_n = n \tau \leq T$, we determine the approximations to all variables $\bfx^n$ to $\bfx(t_n)$, $\bfv^n$ to $\bfv(t_n)$, $\bfw^n = (\bfn^n,\bfV^n)^T$ to $\bfw(t_n) = (\bfn(t_n),\bfV(t_n))^T$, and $\bfu^n$ to $\bfu(t_n)$ by the fully discrete system of \emph{linear} equations, for $n \geq q$,
\begin{subequations}
	\label{eq:BDF}
	\begin{align}
	\label{eq:BDF -- v}
	\bfv^n =&\ \bfV^n \bullet \bfn^n , \\
	\label{eq:BDF -- w}
	\bfM^{[4]}(\widetilde \bfx^n,\widetilde \bfu^n;\widetilde \bfw^n) \dot{\bfw}^n + \bfA^{[4]}(\widetilde \bfx^n) \bfw^n =&\ \bff(\widetilde \bfx^n,\widetilde \bfw^n,\widetilde \bfu^n;\dot\bfu^n) , \\
	\label{eq:BDF -- u}
	\Big(\bfM(\widetilde \bfx^n) \bfu^n \Big)^{\bm \cdot} 
	+ \bfA(\widetilde \bfx,\widetilde \bfu^n) \bfu^n =&\ \bfzero , \\
	\label{eq:BDF -- x}
	\dot \bfx^n =&\ \bfv^n ,
	\end{align}
\end{subequations}
where we denote the discretised time derivatives
\begin{equation}
\label{eq:backward differences def}
	\dot \bfx^n = \frac{1}{\tau} \sum_{j=0}^q \delta_j \bfx^{n-j} , 
	\qquad n \geq q ,
\end{equation}
while by $\widetilde \bfx^n$ we denote the extrapolated values 
\begin{equation}
\label{eq:extrapolation def}
	\widetilde \bfx^n = \sum_{j=0}^{q-1} \gamma_j \bfx^{n-1-j} , 
	\qquad n \geq q .
\end{equation}
Both notations are used for all other variables, in particular note the BDF time derivative of the the product $(\bfM(\widetilde \bfx^n) \bfu^n)^{\bm \cdot}$.

The starting values $\bfx^i$ and $\bfu^i$ ($i=0,\dotsc,q-1$) are assumed to be given. Furthermore, we set $\widetilde \bfx^i = \bfx^i$ and $\widetilde \bfu^i = \bfu^i$ ($i=0,\dotsc,q-1$). The initial values can be precomputed using either a lower order method with smaller step sizes or an implicit Runge--Kutta method.

The method is determined by its coefficients, given by 
\begin{equation*}
	\delta(\zeta)=\sum_{j=0}^q \delta_j \zeta^j=\sum_{\ell=1}^q \frac{1}{\ell}(1-\zeta)^\ell 
	\andquad 
	\gamma(\zeta) = \sum_{j=0}^{q-1} \gamma_j \zeta^j = (1 - (1-\zeta)^q)/\zeta .
\end{equation*}
The classical BDF method is known to be zero-stable for $q\leq6$ and to have order $q$; see \cite[Chapter~V]{HairerWannerII}.
This order is retained, for $q \leq 5$ see \cite{LubichMansourVenkataraman_bdsurf}, also by the linearly implicit variant using the above coefficients $\gamma_j$; cf.~\cite{AkrivisLubich_quasilinBDF,AkrivisLiLubich_quasilinBDF}.
In \cite{Akrivisetal_BDF6}, the multiplier techniques of \cite{NevanlinnaOdeh} have been recently extended, via a new approach, to the six-step BDF method. 

We again point out that the fully discrete system \eqref{eq:BDF} is extremely similar to the fully discrete system for the mean curvature flow \cite[equations (5.1)--(5.4)]{MCF}, and Theorem~6.1 in \cite{MCF} proves optimal-order error bounds for the combined ESFEM--BDF full discretisation
of the mean curvature flow system, for finite elements of polynomial degree $k \geq 2$ and BDF methods of order $2 \leq q \leq 5$.

We note that in each time step the method decouples and hence only requires solving a few linear systems (with symmetric positive definite matrices): first \eqref{eq:BDF -- u} is solved with  $\delta_0 \bfM(\widetilde \bfx^n) + \tau\bfA(\widetilde \bfx^n,\widetilde \bfu^n)$, then, since $\dot\bfu^n$ is already known for $\bff$, \eqref{eq:BDF -- w} is solved with $\delta_0 \bfM(\widetilde \bfx^n,\widetilde \bfu^n,\widetilde \bfw^n) + \tau\bfA(\widetilde \bfx^n)$, and finally \eqref{eq:BDF -- x} with \eqref{eq:BDF -- v} is computed.

\subsection{\bf Convergence experiment}

\blueon We will construct a radially symmetric solution to Problem~\ref{P1} of the form $u(\cdot,t) \equiv u(t)$ on $\Ga[X(\cdot,t)]$, where the surface $\Ga[X(\cdot,t)] \subset \R^{m+1}$ is a sphere of radius $R\t$ with $R(0) = R_0$.
We choose $\D(u) \equiv 1$ and $F(u,H)=-g(u) H$, so that  $K(u,V)= - V / g(u)$. The positive function $g$ here will be chosen later on.

\blueon Since the flow preserves the radial symmetry of $\Ga[X]$, it remains a sphere of radius $R\t$, and  inspection of the diffusion equation  yields that $u(\cdot,t)$ remains spatially constant. For more details on this example we refer to \cite[Section~3.4]{diss_Buerger}.

The  velocity and mean curvature of the evolving sphere in $\mathbb R^m$ of radius $R(t)$  are  $V = \dot R$, and  $H = m/R\t$ so that  \begin{align*}
	v = &\ V \nu = - g(u) H \nu 
\end{align*} yields \begin{equation}
	\label{eq:ODE for R - pre}
	\dot R\t = - g(u\t) \frac{m}{R\t} .
\end{equation}
\blueoff 
On the other hand, by  mass conservation, see~Section~\ref{section:examples and properties}, we have
\begin{equation}
\label{eq:exact solution - u}
	u\t = u_0 \, \Big( \frac{R_0}{R\t} \Big)^m 
\end{equation}
and together with \eqref{eq:ODE for R - pre} this yields
\begin{equation}
\label{eq:ODE for R}
	\dot R\t = - g\bigg( u_0 \, \Big( \frac{R_0}{R\t} \Big)^m \bigg) \frac{m}{R\t} .
\end{equation}

\blueon We set, with $\alpha \in \R$ to be chosen later on,
\begin{align}
	\label{eq:def g - pol}
	g(r) 
	= &\ (1+\alpha) r^{-\alpha} ,
	\intertext{which corresponds to the energy density} 
	\label{eq:def G - pol}
	G(r) = &\ r^{-\alpha} ,
\end{align}
recalling that $g(r) = G(r) - G'(r)r$ for the gradient flow \eqref{eq:MCF diff - gradient flow}. With these coefficient functions this problem satisfies our assumptions from Section~\ref{section:weak formulation and properties}. \redoff 

By choosing \eqref{eq:def g - pol} in the ODE \eqref{eq:ODE for R}, we obtain
\begin{equation}
\label{eq:ODE for R - pol}
	\begin{aligned}
		\dot R\t = &\ - (1+\alpha)\bigg( u_0 \, \Big( \frac{R_0}{R\t} \Big)^m \bigg)^{-\alpha} \frac{m}{R\t} \\
		= &\ - m \, (1+\alpha) \Big( u_0 \, R_0^m \Big)^{-\alpha} R\t^{\alpha m - 1} \\ 
		= &\ - b \, R\t^{\alpha m - 1}  ,
	\end{aligned}
\end{equation}
where the constant $b$ collects all time-independent factors.
Note that for $\alpha = 0$ we recover the classical mean curvature flow ($b = m$).
The solution of the above separable ODE, with initial value $R(0) = R_0$, is 
\begin{equation}
\label{eq:exact solution - R}
	R(t) = \Big( R_0^{2 - \alpha m} - t b (2 - \alpha m) \Big)^{\frac{1}{2 - \alpha m}} ,
\end{equation}
on the time interval $[0,T_{\max}]$.
In the $m$-dimensional case, if $2 - \alpha m \geq 0$ the sphere $\Ga[X]$ shrinks to a point in finite time, $T_{\max} = R_0^{2 - \alpha m} \, (b (2 - \alpha m))^{-1}$, while if $2 - \alpha m < 0$ a solution exists for all times.

\medskip

For the convergence experiment we chose the following initial values and parameters: The initial surface $\Ga^0$ is a two-dimensional sphere of radius $R_0 = 1$, the initial concentration is $u^0(x,0) = 1$ for all $x \in \Ga^0$. The parameter $\alpha$ in \eqref{eq:def G - pol}--\eqref{eq:def g - pol} is chosen to be $\alpha = 2$. That is we are in a situation where a solution exists on $[0,\infty)$. The exact solutions for $\Ga[X]$ and $u(\cdot,t)$ are given in \eqref{eq:exact solution - R} and \eqref{eq:exact solution - u}, respectively.
We started the algorithms from the nodal interpolations of the exact initial values $\Ga[X(\cdot,t_i)]$, $\nu(\cdot,t_i)$, $V(\cdot,t_i) = -g(u(\cdot,t_i)) H(\cdot,t_i)$, and $u(\cdot,t_i)$, for $i=0,\dotsc,q-1$.
In order to illustrate the convergence results of Theorem~\ref{theorem:semi-discrete error estimates}, we have computed the errors between the numerical solution \eqref{eq:BDF} and (the nodal interpolation of the) exact solutions of Problem~\ref{P1} for the above radially symmetric geometric solution in dimension $m=2$.  
The solutions are plotted in Figure~\ref{fig:solutions_conv}.

In Figure~\ref{fig:conv_space} and \ref{fig:conv_time} we report the errors between the numerical solution and the interpolation of the exact solution until the final time $T=1$, for a sequence of meshes (see plots) and for a sequence of time steps $\tau_{k+1} = \tau_k / 2$. 
The logarithmic plots report on the $L^\infty(H^1)$ norm of the errors against the mesh width $h$ in Figure~\ref{fig:conv_space}, and against the time step size $\tau$ in Figure~\ref{fig:conv_time}.
The lines marked with different symbols and different colours correspond to different time step sizes and to different mesh refinements in Figure~\ref{fig:conv_space} and \ref{fig:conv_time}, respectively.

In Figure~\ref{fig:conv_space} we can observe two regions: a region where the spatial discretisation error dominates, matching the $O(h^2)$ order of convergence of Theorem~\ref{theorem:semi-discrete error estimates} (see the reference lines), and a region, with small mesh size, where the temporal discretisation error dominates (the error curves flatten out). For Figure~\ref{fig:conv_time}, the same description applies, but with reversed roles. Convergence of fully discrete methods is not shown, but $O(\tau^2)$ is expected for the 2-step BDF method, cf.~\cite{MCF}.

The convergence in time and in space as shown by Figures~\ref{fig:conv_time} and \ref{fig:conv_space}, respectively, is in agreement with the theoretical convergence results (note the reference lines).

\begin{figure}[htbp]
	\includegraphics[width=\textwidth,clip,trim={100 115 55 110}] {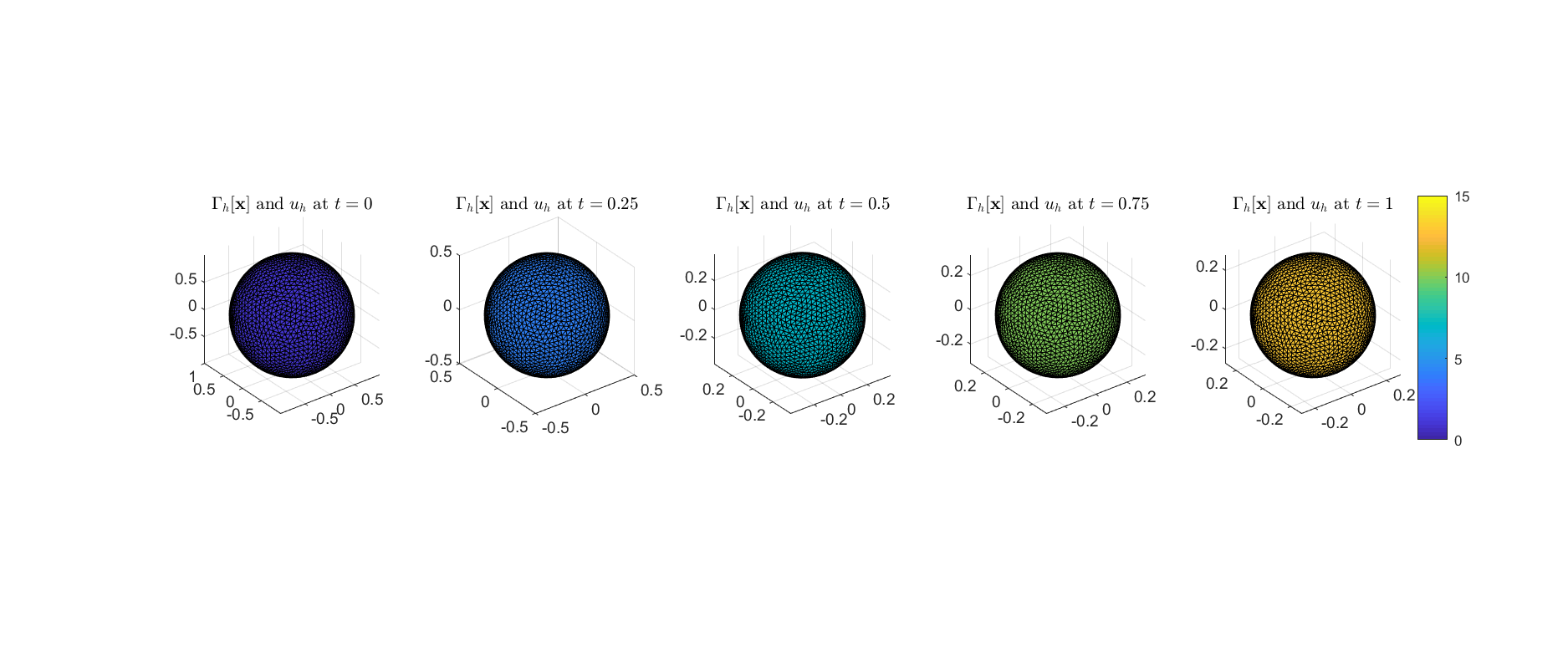}
	\caption{Solutions ($\Ga_h[\bfx]$ and $u_h$) of the radially symmetric flow with $\alpha = 2$ computed using BDF2 / quadratic ESFEM (colorbar applies to all plots).}
	\label{fig:solutions_conv}
\end{figure}

\begin{figure}[htbp]
	\includegraphics[width=\textwidth]{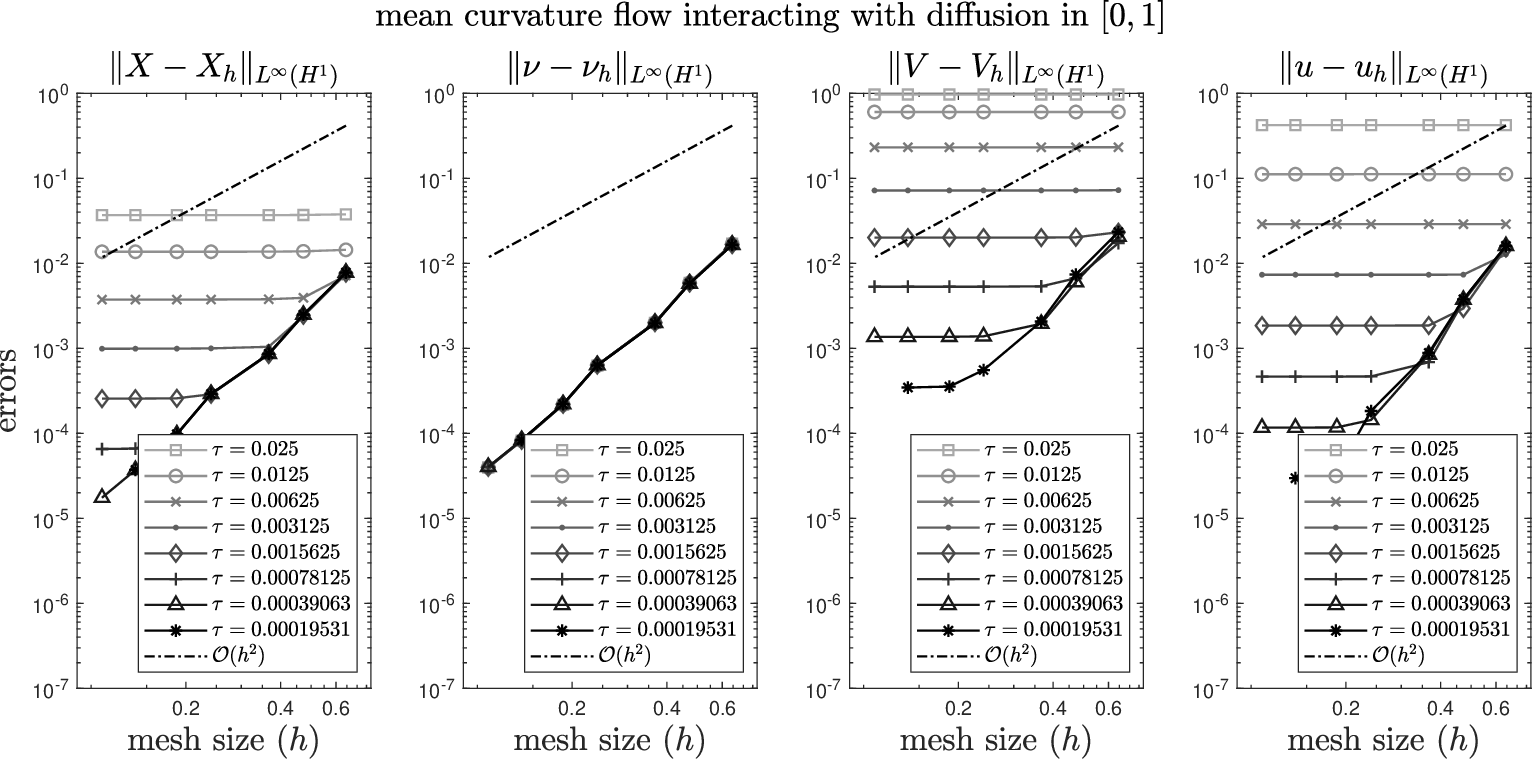}
	\caption{Spatial convergence of the BDF2 / quadratic ESFEM discretisation for Problem~\ref{P1} with $T = 1$ and $\alpha = 2$.}
	\label{fig:conv_space}
\end{figure}

\begin{figure}[htbp]
	\includegraphics[width=\textwidth]{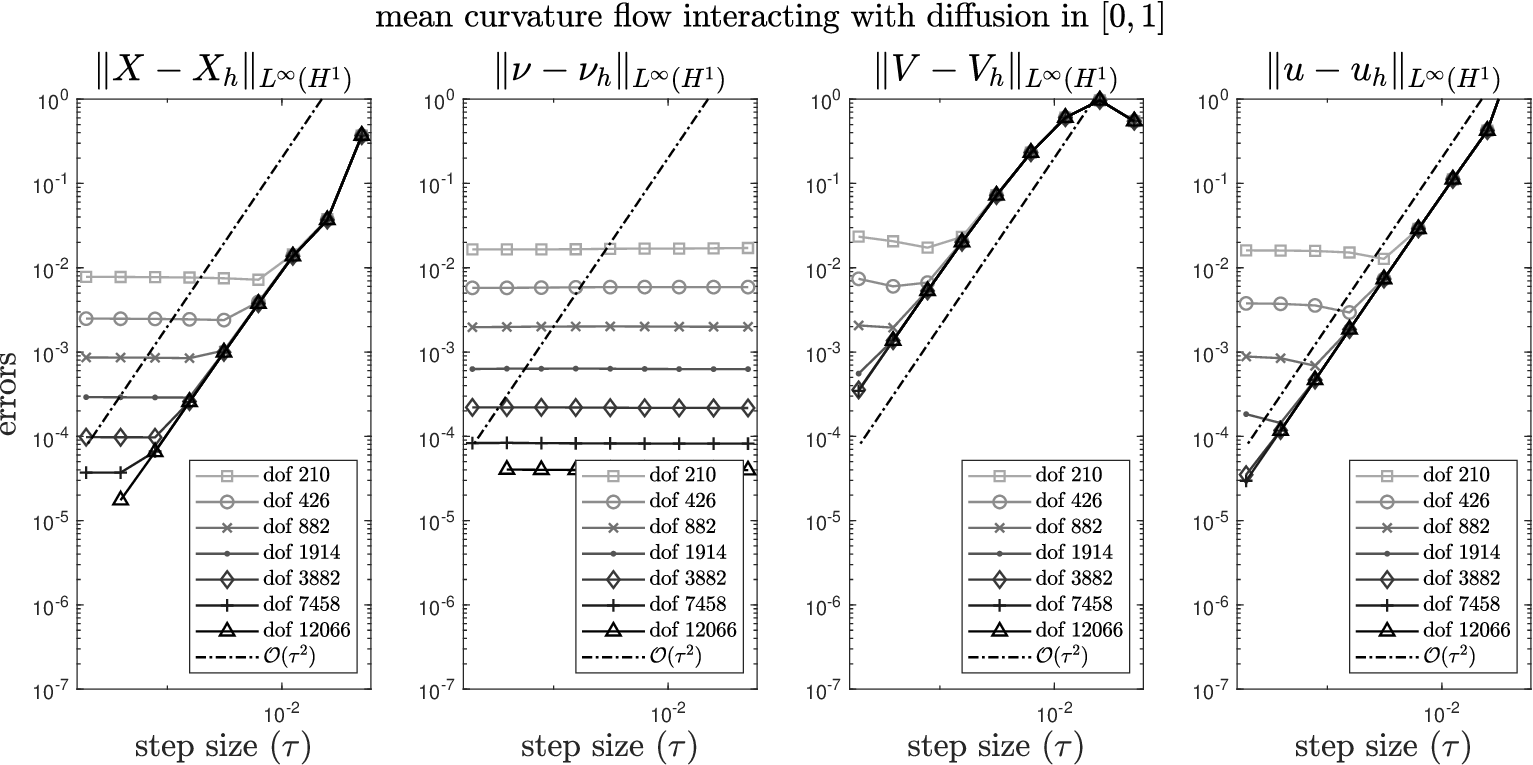}
	\caption{Temporal convergence of the BDF2 / quadratic ESFEM discretisation for Problem~\ref{P1} with $T = 1$ and $\alpha = 2$.}
	\label{fig:conv_time}
\end{figure}

\subsection{\bf Convexity/non-convexity along the flow}

It is well known that for mean curvature flow, \cite{Huisken1984},  a strictly convex surface shrink to a round point in finite time, and stays strictly convex throughout. For the flow \eqref{eq:MCFdiff} this is not true. In Figure~\ref{fig:solutions_elong_ellipsoid_long} and \ref{fig:conservation_elong_ellipsoid_long}, \blueon $F(u,H) = g(u)H$ and $\D(u) \equiv 1$ where $g(u) = G(u) - G'(u)u = 3 u^{-2}$, which corresponds to the energy density $G(u) = u^{-2}$. This problem satisfies our assumptions from Section~\ref{section:weak formulation and properties}. \blueoff We respectively report on the evolution of an elongated ellipsoid as initial surface with an initial value $u^0$ (concentrated along the tips with $\text{values} \approx 5$ decreasing in the middle to $0.5$), and on the mass conservation, weak maximum principle, energy decay, and the mean convexity along the flow. 
The plots of Figure~\ref{fig:solutions_elong_ellipsoid_long} show the flow until final time $T=7.5$, where a pinch singularity occurs. The simulations    \blueon used 
$\text{dof} =4002$ for the number of degrees of freedom and a time step  $\tau = 0.01$\blueoff, (the colorbar applies to all plots). We point out the crucial observation that the \emph{approaching} singularity is detectable on the mass conservation plot. 
Note that in Remark~\ref{remark:semi-discrete mass conservation} only the mass conservation of the spatial semi-discretisation was studied, but not that of the fully discrete numerical method \eqref{eq:BDF}.
For Figure~\ref{fig:conservation_elong_ellipsoid_long} (and the other experiments as well) the mass and energy were computed by quadratures.


\begin{figure}[htbp]
	\includegraphics[width=\textwidth,clip,trim={95 27 55 10}] {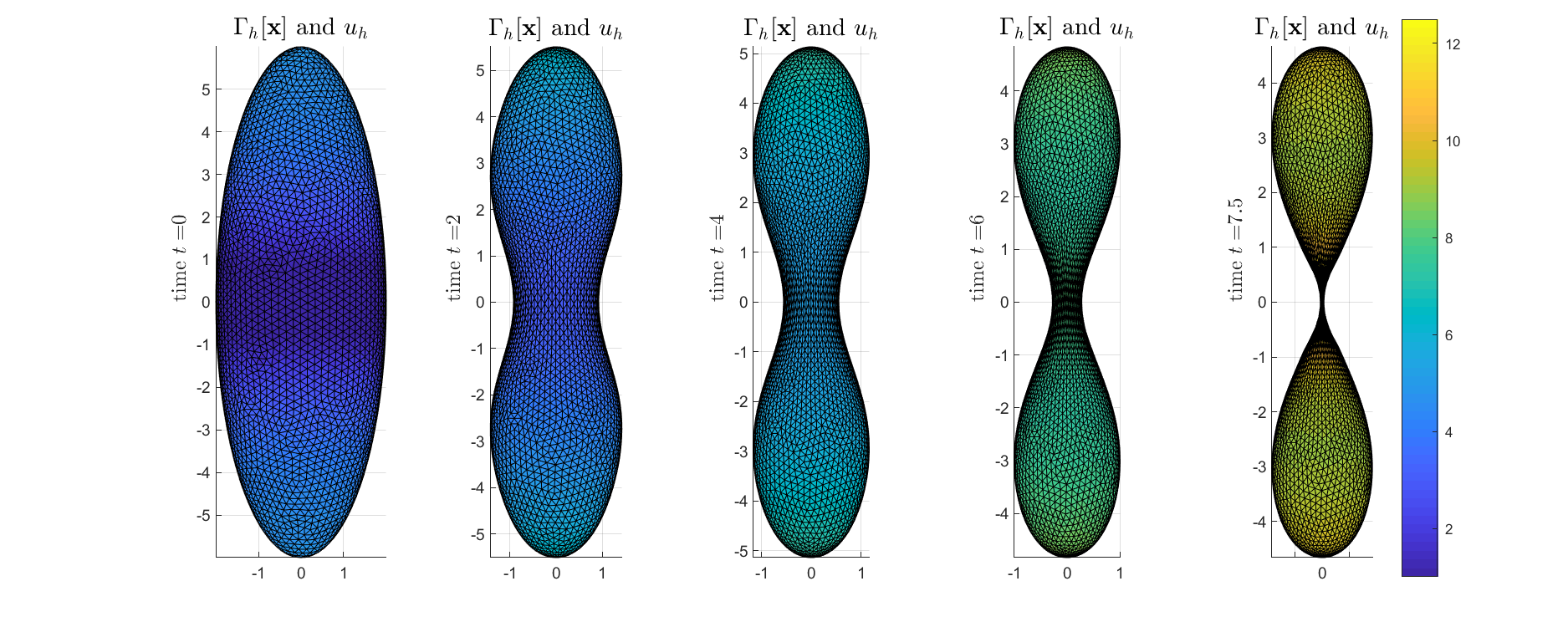}
	\caption{Unlike for mean curvature flow, convex surfaces do not remain convex along the flow \eqref{eq:MCFdiff} with $F(u,H) = g(u)H = 3 u^{-2} H$ (the colorbar applies to all plots).}
	\label{fig:solutions_elong_ellipsoid_long}
\end{figure}
\begin{figure}[htbp]
	\includegraphics[width=\textwidth]{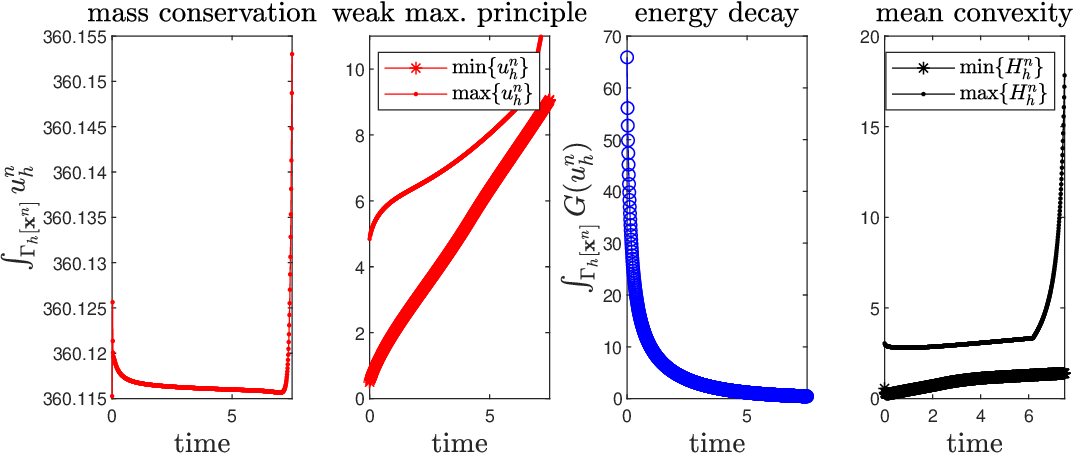}
	\caption{Mass conservation, weak maximum principle, energy decay, and mean convexity for the example in Figure~\ref{fig:solutions_elong_ellipsoid_long} along the flow \eqref{eq:MCFdiff} with $g(u) = 3 u^{-2}$.}
	\label{fig:conservation_elong_ellipsoid_long}
\end{figure}

\subsection{\bf Slow diffusion through a tight neck}

The diffusion speed on surfaces is greatly influenced by the geometry, see, e.g., the insightful paper of Ecker \cite{Ecker_DMV}. To report on such an experiment for the flow \eqref{mcfeqn}--\eqref{diffeqn}, as an initial surface we take a dumbbell (given by \cite[equation~(2.3)]{ElliottStyles_ALEnumerics}) and initial data $u^0$ which is $0.8$ from the neck above and smoothly transitioning to $10^{-4}$ for the neck and below.
\blueon The experiment takes $\D(u) \equiv 1$ and  $F(u,H) = uH$. 
\blueoff In view of these choices, and the initial data $u^0$, the bottom part barely moves, while the top part quickly shrinks before the concentration could pass through the neck. We note that for pure mean curvature flow of this initial surface a pinch-off singularity would occur in finite time, cf.~\cite[Figure~4]{MCF} which uses the exact same initial surface $\Ga^0$. \blueoff 

The experiment uses a mesh with $\text{dof} =10522$ and a time step size $\tau = 0.001$. 
In Figure~\ref{fig:solutions_dumbbell} we observe that, since the diffusion speed at the neck is rather slow, the concentration $u_h$ cannot equidistribute before the top part is vanishing. A pinch-off does not occur, contrary to standard mean curvature flow, see \cite[Section~13.2]{MCF} (using the same initial surface).
Similarly as before, Figure~\ref{fig:conservation_dumbbell} reports on the mass conservation, weak maximum principle, and the mean convexity along the flow, we note however that the initial surface is not mean convex: $\min\{H(\cdot,0)\} \approx -1.72 \cdot 10^{-12}$. Note the axis-limits in Figure~\ref{fig:conservation_dumbbell}.

\begin{figure}[htbp]
	\includegraphics[width=\textwidth,clip,trim={65 60 80 52}] {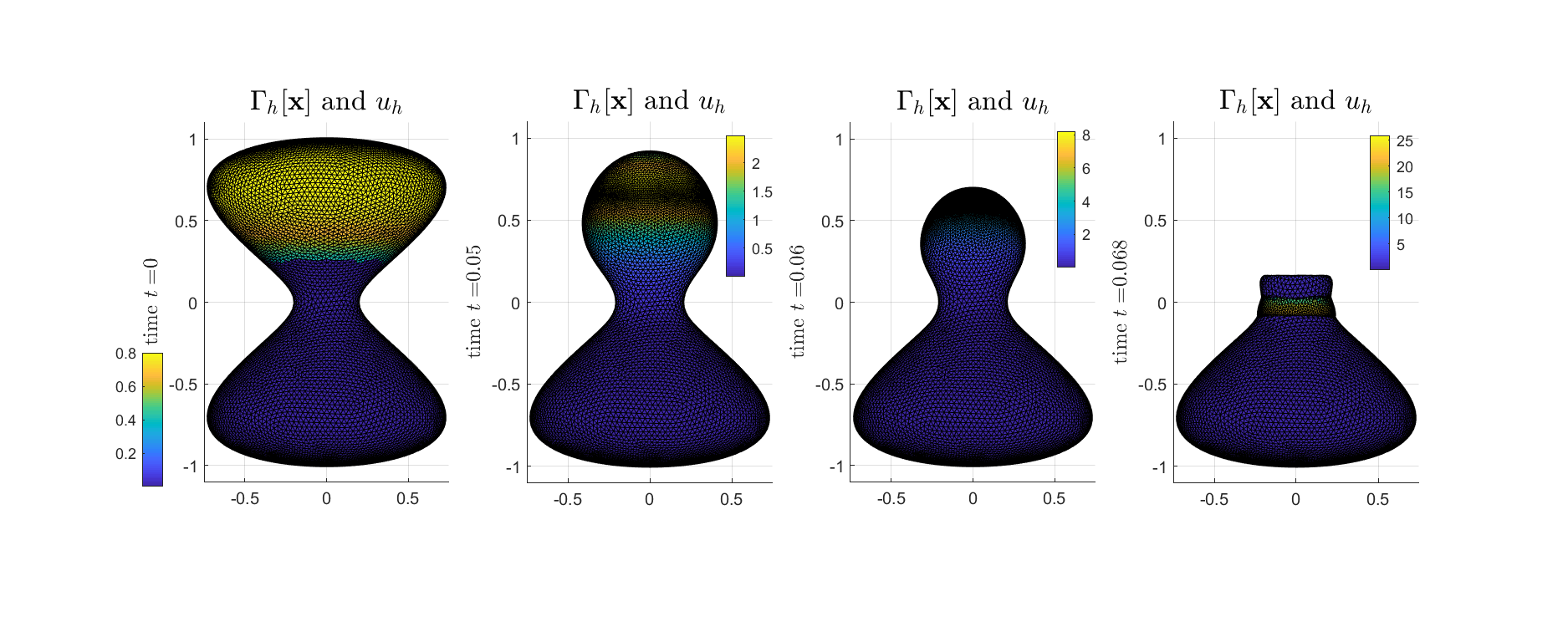}
	\caption{The diffusion speed is slow at a tight neck, and hence the large concentration differences cannot equilibrate before the top part shrinks along the flow \eqref{eq:MCFdiff} with $F(u,H) = uH$ and $\D(u) \equiv 1$.
	}
	\label{fig:solutions_dumbbell}
\end{figure}
\begin{figure}[htbp]
	\includegraphics[width=\textwidth,clip,trim={40 7 150 12}] {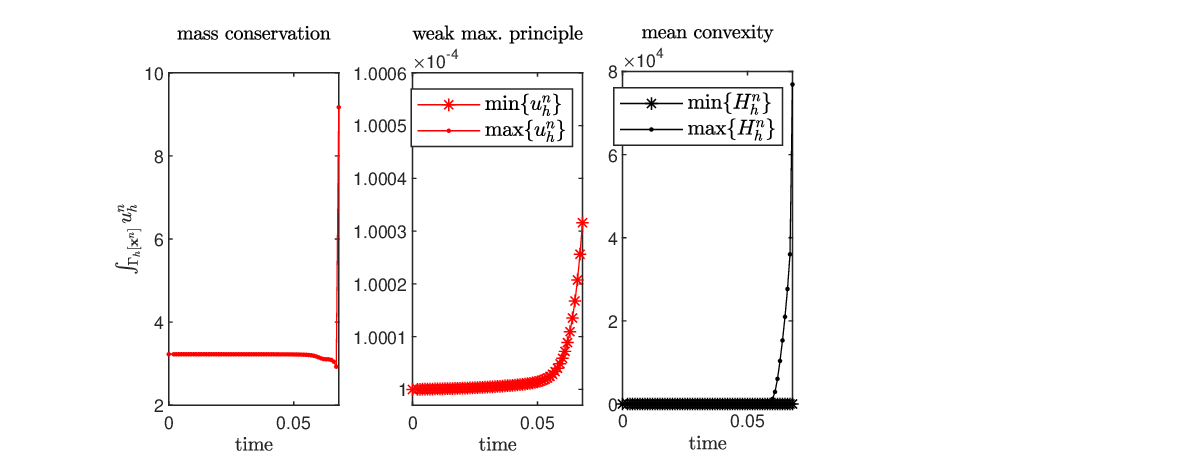}
	\caption{Mass conservation, weak maximum principle (note the axis-scaling), and mean convexity for the example in Figure~\ref{fig:solutions_dumbbell}, along the flow\eqref{eq:MCFdiff} with $F(u,H) = uH$ and $\D(u) \equiv 1$.}
	\label{fig:conservation_dumbbell}
\end{figure}

\subsection{\bf A self-intersecting flow}

The flow \eqref{eq:MCFdiff} may describe surface evolutions where $\Ga[X]$ is self intersecting, i.e.~$X(\cdot,t) \colon \Ga^0 \to \R^3$ is not a parametrisation, but an immersion, see the similar construction in \cite[Figures~5.3--5.5]{diss_Buerger}.
\blueon We consider the flow \eqref{eq:MCFdiff} with $F(u,H) = g(u)H$ and $\D(u) \equiv 1$ where $g(u) = G(u) - G'(u)u = 5 u^{-4}$, which corresponds to the energy density $G(u) = u^{-4}$. This problem satisfies our assumptions from Section~\ref{section:weak formulation and properties}. \blueoff

For a cup shaped surface\footnote{Generated in Blender: \texttt{blender.org}.} (with $\text{dof} =4002$) and a suitably chosen initial value self-intersections are possible. \blueon The initial datum is chosen such that $u^0$ is constant 10 over the whole surface $\Ga^0$, except on the outer-bottom where it is gradually decreased to a smaller $\text{value} \approx 1$ as shown in the leftmost plot in Figure~\ref{fig:solutions_cup_bottom}. \blueoff
In Figure~\ref{fig:solutions_cup_cut} and \ref{fig:solutions_cup_bottom} we present the numerical solution obtained by the 2-step BDF method with $\tau = 10^{-3}$. The self-intersection is clearly observable on both figures, e.g., note the bright patch in Figure~\ref{fig:solutions_cup_bottom} after the self intersection.
Of course the self-intersection does not influence the mass conservation, weak maximum principle, and energy decay, see Figure~\ref{fig:conservation_cup}. (The initial surface is clearly not convex.)

\begin{figure}[htbp]
	\includegraphics[width=\textwidth,clip,trim={82 40 75 30}] {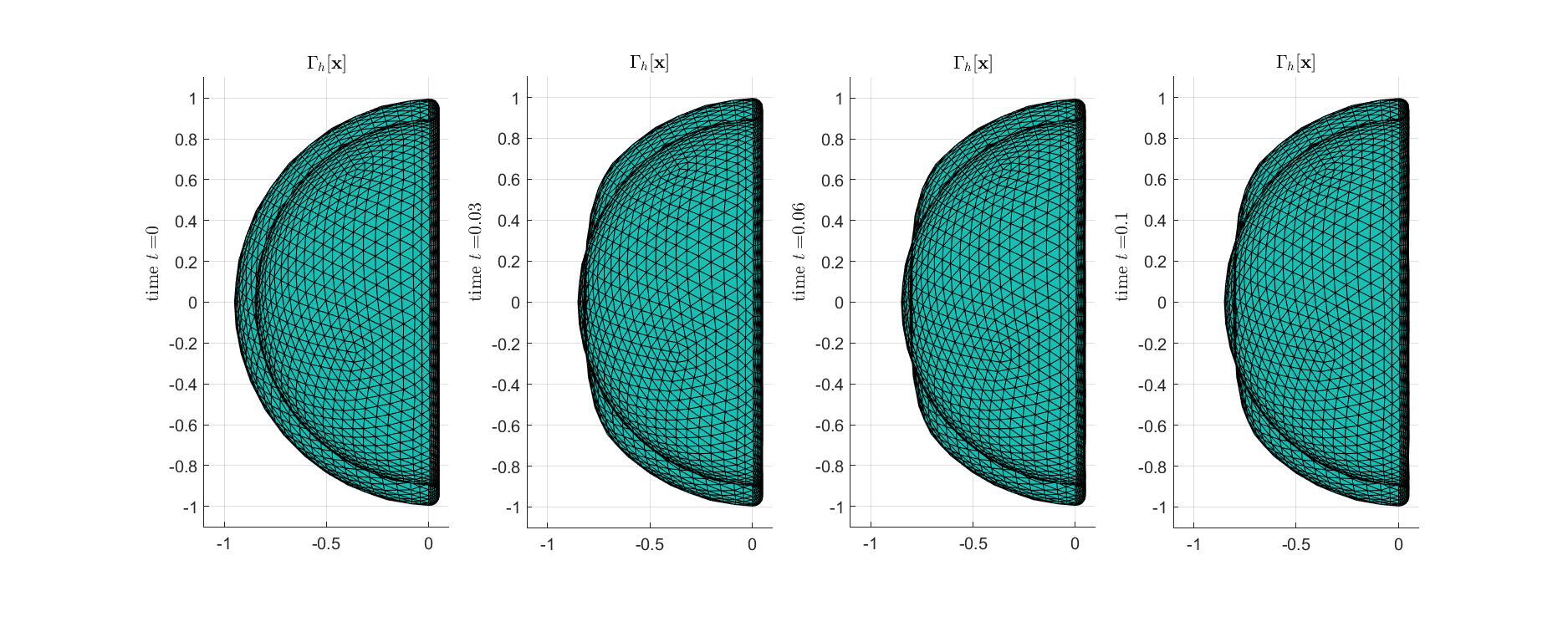}
	\caption{Snapshots (cross section at the $x = 0$ plane) of a self-intersecting evolution \eqref{eq:MCFdiff} with $F(u,H) = g(u) H = 5u^{-4} H$ and $\D(u) \equiv 1$.}
	\label{fig:solutions_cup_cut}
\end{figure}
\begin{figure}[htbp]
	\includegraphics[width=\textwidth,clip,trim={82 90 45 80}] {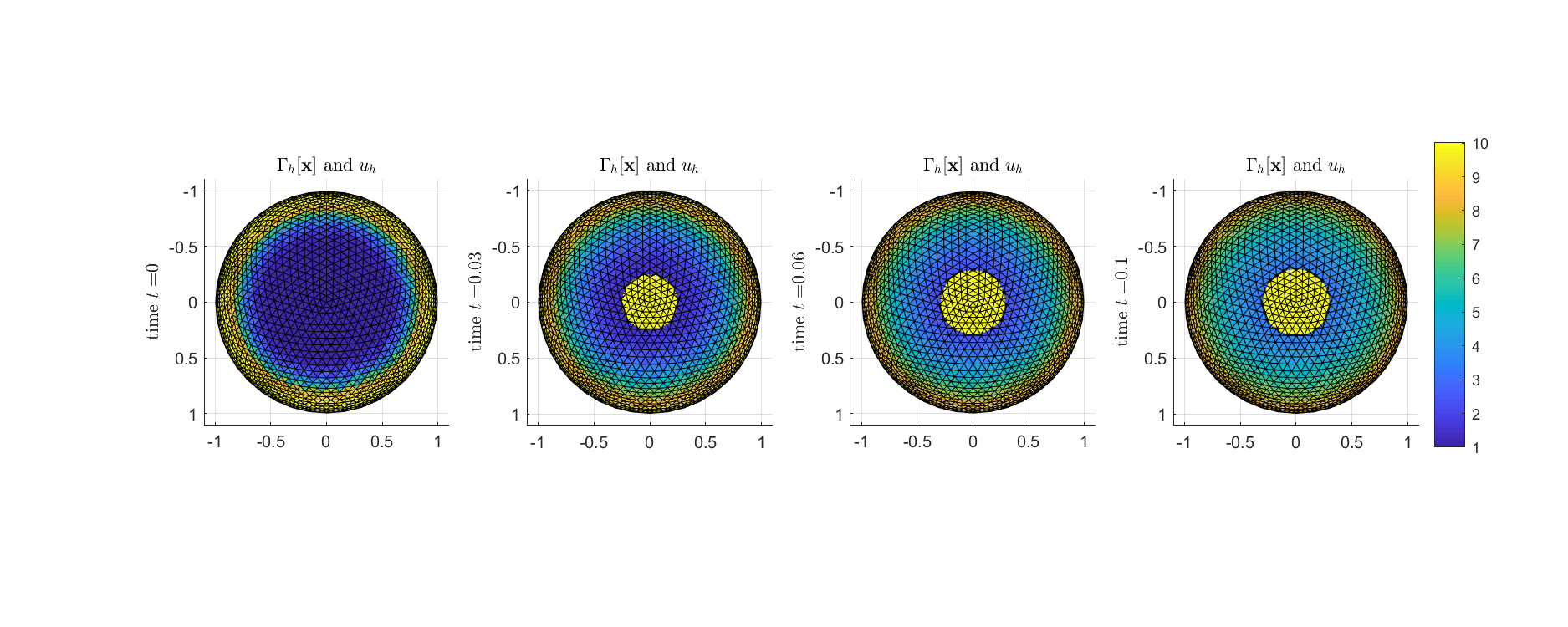}
	\caption{Snapshots of the surface $\Ga_h[\bfx]$ and the concentration $u_h$ (bottom view of the $x$--$y$-projection) of a self-intersecting evolution \eqref{eq:MCFdiff} with $F(u,H) = g(u) H = 5u^{-4} H$ and $\D(u) \equiv 1$.}
	\label{fig:solutions_cup_bottom}
\end{figure}
\begin{figure}[htbp]
	\includegraphics[width=\textwidth] {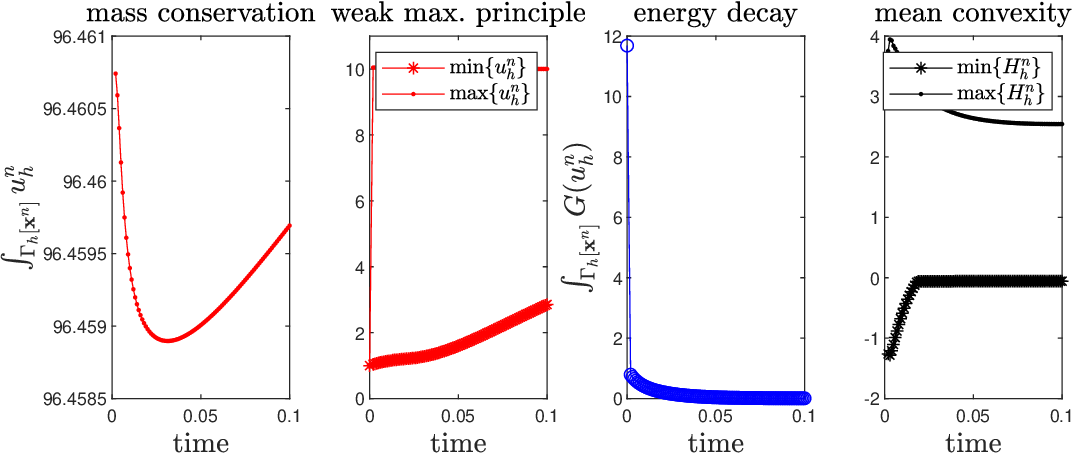}
	\caption{Mass conservation, weak maximum principle and energy decay for the self-intersecting example in Figure~\ref{fig:solutions_cup_cut}--\ref{fig:solutions_cup_bottom} along the flow \eqref{eq:MCFdiff} with $g(u) = 5u^{-4}$.}
	\label{fig:conservation_cup}
\end{figure}

\section*{Acknowledgments}

We thank Stefan Schmidt for his Blender magic regarding the initial cup surface.
The work of Harald Garcke and Bal\'azs Kov\'acs is supported by the DFG Graduiertenkolleg 2339 \emph{IntComSin} -- Project-ID 321821685.
The work of Bal\'azs Kov\'acs is funded by the Heisenberg Programme of the Deutsche Forschungsgemeinschaft (DFG, German Research Foundation) -- Project-ID 446431602.

\bibliographystyle{alpha}
\bibliography{evolving_surface_literature}

\end{document}